\documentclass[11pt]{article} 

\usepackage[utf8]{inputenc}

\usepackage{geometry} 
\usepackage{graphicx} 

\usepackage{booktabs} 
\usepackage{array} 
\usepackage{paralist} 
\usepackage{verbatim} 
\usepackage{subfig} 
\usepackage{lastpage} 
\usepackage{extramarks} 
\usepackage{graphicx} 
\usepackage{lipsum} 
\usepackage{color}
\usepackage{empheq}
\usepackage{mathtools}
\usepackage{amssymb}
\usepackage{mdframed}
\usepackage{hyperref}
\usepackage{amsthm}
\usepackage{cleveref}
\usepackage{stmaryrd}
\usepackage{appendix}
\hypersetup{%
    bookmarksnumbered, bookmarksopen=true, bookmarksopenlevel=1,%
}
\newtheorem{theorem}{Theorem}
\newtheorem{lemma}[theorem]{Lemma}

\newtheorem{example}{Example}[section]
\newtheorem{remark}{Remark}
	\newtheorem{prop}{Proposition}

\usepackage{fancyhdr} 
\newcommand{\A}{\mathcal{A}}
\newcommand{\B}{\mathcal{B}}

\newcommand{\g}{\mathbf{g}}
\newcommand{\G}{\mathbb{G}}

\newcommand{\z}{\mathbf{z}}
\newcommand{\Z}{\mathbb{Z}}
\newcommand{\F}{\mathcal{F}}
\newcommand{\Ftu}{\mathcal{F}^{d,~\rm{tail}}}

\DeclarePairedDelimiter{\dbracket}{\llbracket}{\rrbracket}

\usepackage{cite}
\title{A free central-limit theorem for dynamical systems. }
\author{Morgane Austern\\{\small Harvard University}}
\date{} 

\begin{document}
\maketitle

\begin{abstract}

The free central-limit theorem, a fundamental theorem in free probability, states that empirical averages of freely independent random variables are asymptotically semi-circular. We extend this theorem to general dynamical systems of operators that we define using a free random variable $X$ and a group of *-automorphims describing the evolution of $X$. We introduce free mixing coefficients that measure how far a dynamical system is from being freely independent. Under conditions on those coefficients, we prove that the free central-limit theorem also holds for these processes and provide Berry-Essen bounds. We generalize this to triangular arrays and u-statistics. Finally we draw connections with classical probability and random matrix theory with a series of examples.

\end{abstract}

\section{Introduction}

The free central-limit theorem is a fundamental theorem of free probability. It states that the distribution of a sum of $n$ freely independent operators is asymptotically semi-circular as $n$ goes to infinity. This mirrors the classical central limit theorem with the condition of independence replaced by free independence (sometimes called freeness) and the Gaussian limit replaced by a semi-circular limit.
More precisely, let $\mathcal{A}$ be a Von-Neumann algebra with a normal faithful semi-definite trace $\tau$ and adjoint operator $*$.
If $X_1,\dots,X_n$ are free copies of a  self-adjoint free random variable $X$ satisfying $\tau(X)=0$ and $\tau(X^2)=1$ then the free central-limit theorem guarantees that $$S_n:=\frac{X_1+\dots+X_n}{\sqrt{n}}\xrightarrow{d} \mu_{\rm{sc}}.$$ This theorem shares deep connections with the semi-circular law in random matrix theory for which it notably provides an alternative proof.
~However this result, and subsequent ones, rely on the assumption that $(X_1,\dots,X_n)$ are free or ``asymptotically" free. This is a strong assumption that might not be respected by many processes of interest such as sequences of dependent matrices. For example, if $(X^n_i)$ is a sequence of dependent random matrices with correlated Gaussian entries then the free central limit theorem does not apply to the study of the eigenvalue distribution of $\frac{1}{\sqrt{n}}\sum_{i\le n}X^n_i$.
~To be applicable, the free central-limit theorem  would need to be generalized to weakly ``freely dependent" sequences of free random variables. Another shortfall of the free central limit theorem is that it applies exclusively to empirical averages and not to more complex quantities such as u-statistics of which $\sum_{i,j\le n}\Phi(X_i, X_j)$ (for a continuoud function $\Phi$) is an example. This  represents a disconnect with what is known for the classical central-limit theorem where the assumption of independence has been successfully relaxed for general dynamical systems \cite{bolthausen1982central,austern2018limit}, and where the limit of u-statistics are extensively studied.

\noindent This paper bridges this gap by generalizing the free central-limit theorem to empirical averages and u-statistics of dynamical systems.
A non-commutative dynamical system is defined by a free random variable $X$ and a group of $*$-automorphims $(K_g)_{g\in \G}$ describing the evolution of $X$. Examples include stationary and quantum exchangeable sequences of free random variables. Our goal is to establish conditions under which empirical averages and u-statistics of $(X_g):=(K_g(X))$ are asymptotically semi-circular. In the classical setting, the central limit theorem is generalized to the dependent setting by defining \cite{austern2018limit} mixing conditions that quantify how far a process is from being independent and imposing conditions on the speed of decay of those coefficients.
Mirroring this, in \cref{main} we define ``free mixing" coefficients that quantify how far  from being (conditionally) free a process is. If those free mixing coefficients are small enough, then in \cref{main} we demonstrate that empirical averages of the dynamical system  are asymptotically semi-circular with a radius that depends on the structure of the process. 
 We extend this, under moment conditions, to unbounded operators as well as to triangular arrays, we also provide bounds of the type of Berry-Essen. Then, in \cref{u_st} we consider non-commutative u-statistics and under similar free mixing conditions prove that their limiting distribution is also asymptotically semi-circular. Finally, we illustrate the utility of this new notion of `` free mixing" by a series of examples notably from random matrix theory in \cref{ex_random_matrices,ex_u_stat}.

\subsection{Related literature} 
The non-commutative law of large numbers \cite{bercovici1996law} states that if $X_1,\dots,X_n$ are freely independent and identically distributed random variables satisfying $\tau(X_1)=0$ then under moment conditions we have $\frac{1}{n}\sum_{i\le n}X_i\rightarrow 0$. This key result was greatly extended to stationary sequences of non-commutative operators  \cite{lance1976ergodic}; as well as to general dynamical systems of operators \cite{conze1978ergodic}. We study, under free mixing conditions, the speed of convergence of those generalized laws of large numbers.
\\\noindent The free central limit-theorem, first introduced by Voiculescu \cite{voiculescu1986addition}, is a second-order result that establishes a speed of convergence for the classical law of large numbers. 
 It was later extended in \cite{bercovici1993free}, under moment conditions, to free convolutions of unbounded operators, and by Speicher to the multivariate case in \cite{speicher1990new}. Moreover, it was generalized to operators that are conditionally free (also called 'free with amalgamation') in \cite{popa2007freeness}, as well as to operators that satisfy a slightly weaker notion of free independence in \cite{kargin2007proof}.\\\noindent Just as in classical probability, Berry-Essen type bounds guarantee a speed of convergence for the free central-limit theorem. They were first established for sums of bounded free operators \cite{kargin2007berry} then later for free unbounded operators \cite{chistyakov2008limit}, and finally to the multivariate and conditionally free case \cite{mai2013operator}. 
\\\noindent 
Semi-circular limits have also been of great interest in the random matrix literature. It is well-known that Wigner random matrices have an empirical spectral distribution that is asymptotically semi-circular \cite{wigner1967random,arnold1971wigner}. This classical result has been successfully extended to random matrices with dependent entries \cite{hochstattler2016semicircle} as well to random matrices with exchangeable entries \cite{adamczak2016circular}. Finally inspired by application in communication theory operator-valued matrices have been the object of increased interest \cite{banna2018operator, mai2013operator}. In \cref{ex_random_matrices} we prove, under general conditions, that the empirical average of dependent patterned random matrices is asymptotically semi-circular. The study of this class of random matrices has been motivated by applications in \cite{handel2017structured}. \cite{bandeira2021matrix} studied the spectral norm of those matrices by using ideas coming from free probability.

\section{Main results}\label{main}
\subsection{Definitions and notations}
Let $\mathcal{A}$ be a Von Neumann algebra with a normal faithful semi-finite trace $\tau$ and adjoint operator $*$.  We write $\pi(\mathcal{A})$ its faithful representation on a Hilbert space $\mathcal{H}_{\lambda},$ and for simplicity  identify it with $\mathcal{A}$. Throughout this paper, an important example will be $\A_m:=L_{\infty}\otimes M_m(\mathbb{C}) $, the algebra of random matrices of size $m$ with essentially bounded entries. In this case, the trace is taken to be $X\rightarrow \frac{1}{m}\mathbb{E}\big(\rm{Tr}(X)\big)$.  {We say that an operator $X$ is affiliated with $\mathcal {A}$ if all its spectral projections belong to $\mathcal{A}$. It is, in addition, $\tau$-measurable if for all $\delta$ there is a projection $p$ of $\mathcal{H}_{\lambda}$ such that $\tau(I-p)\le \delta$ and $pH\subset \mathcal{D}(X)$ where $\mathcal{D}(X)$ designates the domain of $X$. We write $\mathcal{A}_{\tau}$ the set of $\tau$-measurable operators, this forms a Hausdroff complete $*$-algebra.} 

\noindent \textbf{Free random variable and non-commutative $L_p$-space} A free random variable is an operator $X\in \mathcal{A}_{\tau}$, for example a random matrix of size $m$.
 We say that it is self-adjoint if it verifies $X^*=X$. We denote by $\sigma(X)$ the spectrum of $X$. According to the spectral theorem if $X$ is self-adjoint then there is a unique projection-valued measure $P_X$ on $\sigma(X)$ such that $X=\int_{\sigma(X)}\lambda dP_X(\lambda)$. We call $\mu_{X}:=\tau\circ P_X$ the distribution of $X$ which notably satisfies $\tau(X^k):=\int_{\mathbb{R}} x^k d\mu_X(x)$.
 \\~We define the $L_p$-space and $L_p$- norm of $\A_{\tau}$ in the following way: $$L_p(\mathcal{A},\tau):=\Big\{X\in \mathcal{A}_{\tau}\big|~\int_{-\infty}^{\infty}t^p d\mu_{|X|}(t)<\infty\Big\},\quad\rm{and}\quad \|X\|_p=\Big[\int_{-\infty}^{\infty}t^p d\mu_{|X|}(t)\Big]^{\frac{1}{p}}.$$

  \noindent Similarly as in the classical setting, we can define a notion of multivariate distribution. The challenge however is that as free random variables do not necessarily commute we cannot rely on the spectral theorem. Instead, we remark that $X, Y\in\mathcal{A}$  have the same distributions if and only if all their moments are equal. This is the notion we extend to the multivariate setting.
We denote by $\mathbb{C}\big<x_1,\dots,x_k,x^*_1,\dots,x^*_k\big>$ the algebra of non commutative *-polynomials in the formal random variables $x_1,\dots,x_k$.
We say that  $(X_i)\in  \mathcal{A}^{\mathbb{N}}$ has the same distribution that $(Y_i)\in \mathcal{A}^{\mathbb{N}}$ if and only if  for all integers $k\in \mathbb{N}$ and for all polynomials $P\in \mathbb{C}\big<x_1,\dots,x_k,x^*_1,\dots,x_k^*\big>$ we have: $$\tau\big(P(X_{i_1},\dots,X_{i_k},X_{i_1}^*,\dots,X^*_{i_k})\big)=\tau\big(P(Y_{i_1},\dots,Y_{i_k},Y_{i_1}^*,\dots,Y_{i_k}^*)\big);$$
and we write $(X_i)\overset{d}{=}(Y_i)$. To extend this to elements of $\mathcal{A}_{\tau}$, we note that general elements of $\A_{\tau}$ do not necessarily have finite moments. We say instead that  $(X_i)\in  \mathcal{A}_{\tau}^{\mathbb{N}}$ has the same distribution that $(Y_i)\in \mathcal{A}_{\tau}^{\mathbb{N}}$ if there are sequences $(p_{i,n})$ and $(p'_{i,n})$ of projectors of $\mathcal{H}_{\lambda}$ satisfying:\begin{itemize}
     \item $\tau(1-p_{i,n}),\tau(1-p'_{i,n})\xrightarrow{n\rightarrow \infty} 0$ for all $i\in\mathbb{N}$;
     \item $(X_{i_1}p_{i_1,n},\dots,X_{i_k}p_{i_k,n})\in \mathcal{A}^k$,~ $(Y_{i_1}p'_{i_1,n},\dots,Y_{i_k}p'_{i_k,n})\in \mathcal{A}^k$ for all $n\in \mathbb{N};$
     \item $(X_{i_1}p_{i_1,n},\dots,X_{i_k}p_{i_k,n})\overset{d}{=}(Y_{i_1}p'_{i_1,n},\dots,Y_{i_k}p'_{i_k,n})$ for all $n\in \mathbb{N}$.
 \end{itemize}
\noindent Finally we say that a sequence $(a_n)\in \mathcal{A}^{\mathbb{N}}_{\tau}$ converges almost everywhere (a.e) to $a\in \mathcal{A}_{\tau}$ if for all $\epsilon>0$ there is a projection of $\mathcal{H}_{\lambda}$ that we denote $p_{\epsilon}$ such that (i) $\tau(p_{\epsilon})\ge 1-\epsilon$; and (ii) $\|(a_n-a)p_{\epsilon}\big\|_{{\infty}}\rightarrow 0$.
 
\vspace{5mm}
\noindent \textbf{L.s.c.H amenable group}
Let $\G$ be a topological group. If its topology is locally compact, second countable and Haussdroff we say that the group is locally compact, second countable and Haussdroff (l.c.s.c.H). For example if $\G$ is a countable group equipped with the discrete topology it is l.c.s.c.H. We always equip $\G$ with its Borel $\sigma$-algebra $\mathcal{B}(\G)$. For every l.c.s.c.H group we can find a measure $|\cdot|$ on $\mathcal{B}(\G)$ that is left-invariant: \begin{equation}\label{jk}|gS|=|S|,\qquad \forall g\in \G ~\rm{and}~ S\in \mathcal{B}(\G).\end{equation}Such  a measure is called a \emph{Haar}-measure. It is unique up to a multiplicative constant. For example if $\G=\mathbb{R}$ the Lebesgue measure is a Haar measure or if $\G$ is a countable group the cardinality is also a Haar measure. In general we note that if $K\subset \G$ is a compact subset then its Haar measure is finite $|K|<\infty$. Informally a Haar measure generalize the notion of volume and the \cref{jk} shows that a set can be shifted without changing its volume. In the following $|\cdot|$ will always designate a Haar measure. Similarly as the volume is left invariant if $\G$ is l.c.s.c.H we can always find a metric $d(\cdot,\cdot)$ on $\G$ that is left invariant: $$d(\cdot, \cdot)=d(g\cdot, g \cdot)\qquad \forall g\in \G.$$ For example if $\G$ is a finitely generated nilpotent group then the word metric is left invariant or if $\G=\mathbb{R}$ is the group of reals then $d:(x,y)\rightarrow|x-y|$ is left invariant. We denote $\mathbb{B}(\phi,t)$ the metric  ball in $\G$ centered around $\phi\in \G$ and of radius $t>0$. We shorthand $\mathbb{B}(t):=\mathbb{B}(e,t)$ where $e$ designates the neutral element of $\G$. Finally we write $\bar{d}$ the Haussdroff distance induced by $d$: $$\bar{d}(G_1,G_2):=\min_{g_1\in G_1,g_2\in G_2} d(g_1,g_2).$$

\noindent A group is said to be \emph{amenable} if we can find a sequence $(A_n)$ of subsets of $\G$ that satisfy for all compact set $K\subset \G$: $$\frac{|A_n\cap KA_n|}{|A_n|}\xrightarrow{n\rightarrow \infty} 0.$$ Such a sequence is called a  F\o{}lner sequence. For example if $\G=\mathbb{R}^r$ then the sequence of subsets $\big([0,{n}]^r\big)$ is a F\o{}lner sequence. Similarly, if $G=\mathbb{S}(\mathbb{N})$ is the group of permutations of $\mathbb{N}$ and if we write $\mathbb{S}_n\subset \mathbb{S}(\mathbb{N})$ the subgroup of permutations that leave $[|n,\infty)$ invariant then the sequence $(\mathbb{S}_n)$ is a F\o{}lner sequence. We say that $(A_n)$ is tempered if $$\Big|\bigcup_{k<n} A_k^{-1}A_n\Big|\le c|A_n|\qquad \rm{for~some~}c>0 ~\rm{and~all}~ n\in \mathbb{N}.$$ Not every F\o{}lner sequence is tempered but every l.c.s.c.H amenable group $\G$ contains a F\o{}lner sequence that is also tempered.

\noindent\textbf{Group indexed process and definition of $\mathcal{F}_G(\cdot)$} 
Let $X:=(X_g)_{g\in \G}$  be a sequence of free random variables indexed by a l.c.s.c.H amenable group $\G$ with F\o{}lner sequence $(A_n)$. We call $(X_g)$ a \emph{non-commutative process indexed by $\G$}.\\ \noindent For all subsets $G\subset \G$ we denote $\mathcal{F}_{G}(X)$ the unital algebra generated by $\{X_g,g\in G\}$. For example if $\G=\mathbb{Z}$ is the group of integers and $G=\dbracket{i}$ then $\mathcal{F}_{G}(X)$ is the unital algebra generated by $\{X_l,~ 0\le l\le i\}$.


\noindent \textbf{Invariant dynamical systems }A special case of non-commutative processes are non-commutative dynamical systems. They are defined by a free random variable $\tilde X$ and a group of $*$-automorphims describing the evolution of $\tilde X$. 
More specifically, let $(K_g)_{g\in\G}$ be a net of *-automorphims of $\mathcal{A}_{\tau}$ that satisfies \begin{enumerate}\item  $\tau(K_g(a))=\tau(a)$ for all $a\in L_1(\A,\tau)$\qquad \qquad $(H_1).$\item$K_g\circ K_{g'}=K_{g g'}$\qquad \qquad \qquad \qquad\qquad \qquad ~ $(H_2).$
\end{enumerate} We define $X:=(X_g)$ as the sequence of images of $\tilde X$: $X_g:=K_g(\tilde X)$.
We call $X$ a \emph{dynamical system} and say that $(K_g)$ defines a group action on $\mathcal{A}_{\tau}$. Examples of dynamical systems include stationary fields for which we take $\G$ to be $\mathbb{Z}^r$ or exchangeable sequences of free random variables for which we take $\G$ to be $\mathbb{S}(\mathbb{N})$. We note that the condition $(H_1)$ implies that the distribution of $X_g$ is the same for all $g\in \G$. We say that $X$ is distributionally invariant under the action of $\G$. Classically when we take $\G=\mathbb{Z}^r$ we call such an invariant process stationary; and when we take $\G=\mathbb{S}(\mathbb{N})$ we call such a process exchangeable.
 The converse is also true: Any distributionally invariant process $(X_g)$ can alternatively be defined as a dynamical system.
 
\begin{prop}\label{eq_suis}Let $\tilde X\in L_1(\mathcal{A},\tau)$ be a self adjoint operator and $(K_g)$ be a set of *-automorphisms satisfying conditions $(H_1)$-$(H_2)$.  If we write $X_g=K_g(\tilde X)$, then the process $X:=(X_g)$ satisfies $${(X_{g_1},\dots,X_{g_k})}\overset{d}{=}{(X_{gg_1},\dots,X_{gg_k})},\qquad \forall g,g_1,\dots g_k\in \mathbb{G}.$$
Conversely let $(Z_g)\in \mathcal{A}_{\tau}^{\mathbb{G}}$ be a sequence indexed by the group  $\mathbb{G}$. Denote $\B\subset \A_{\tau}$ the unital sub-algebra generated by $(Z_g)$. If for all $g,g_1,\dots g_k\in \mathbb{G}$ we have ${(Z_{g_1},\dots,Z_{g_k})}\overset{d}{=}{(Z_{gg_1},\dots,Z_{gg_k})}$, and if $g\rightarrow Z_{g}$ is continuous almost everywhere then there is a net $(K_g)$ of *-automorphisms of $\B$ that verifies conditions $(H_1)$-$(H_2)$ and is such that $K_g(Z_{g'})=Z_{gg'}$ for all $g,g'\in \mathbb{G}$. 
\end{prop}

\noindent The invariant algebra is defined as $$\F^{\rm{tail}}(X):=\{a| a\in \mathcal{A}_{\tau},~K_g(a)=a,~\forall g\in \mathbb{G}\}\subset \mathcal{A}_{\tau}.$$ 
There is a unique linear map $E:\mathcal{A}_{\tau}\rightarrow \F^{\rm{tail}}(X)$ that satisfies (i) $E(aXb)=aE(X)b$ for all $X\in \mathcal{A}_{\tau}$ and $a,b\in \F^{\rm{tail}}(X);$ and (ii) $\tau(X)=\tau(E(X))$ for all $X\in \A_{\tau}$. 
We call $E$ the non-commutative conditional expectation on $\F^{\rm{tail}}(X)$ (see \cite{umegaki1954conditional,umegaki1956conditional} for more background).
~Finally, we say that $(X_g)$ is \emph{ergodic} if for all $A,B\in \mathcal{F}_{\G}$ we have: $$\lim_{n\rightarrow\infty}\frac{1}{|A_n|}\int_{A_n}\tau\Big(K_g(A)B\Big)d|g  |=\tau(A)\tau(B).$$ When the dynamical system $(X_g)$ is ergodic then the invariant-algebra $\F^{\rm{tail}}(X)$ is trivial.

\noindent \textbf{Ergodic theorem} In classical probability, if $\G=\Z$ is the group of integers then the average of observations $(Y_i)$ over $\dbracket{n}$ is called an empirical average. Intuitively, the F\o{}lner sequence $(A_n)$ plays the same role for $\G$ than $\dbracket{n}$ does for $\Z$: it is an exhaustive and stable sequence of subsets. Therefore we call $\bar{X^n}:=\frac{1}{|A_n|}\int_{A_n}X_gd|g|$ an empirical average. For example if $\G=\mathbb{Z}^r$ and the F\o{}lner sequence is chosen to be $A_n:=\dbracket{n}^r$ then we have $\bar{X^n}=\frac{1}{n^r}\sum_{i_1,\dots,i_r\le n} X_{i_1,\dots,i_r}$. Our goal is to study the asymptotic of this estimator. The ergodic theorem for dynamical systems \cite{conze1978ergodic} states that empirical averages of dynamical systems converge to their conditional expectation $E(X)$.
\begin{theorem}\label{thm111}\cite{conze1978ergodic}
Let $\tilde X\in L_1(\mathcal{A},\tau)$ and $(K_g)$ be a sequence of *-automorphisms respecting $(H_1)-(H_2)$. Define $X_g=K_g(\tilde X)$ and choose $(A_n)$ to be a tempered F\o{}lner sequence of $\G$. Then empirical averages converge to their conditional expectation: $$\frac{1}{|A_{n}|}\int_{A_n} X_g d|g|\xrightarrow{L_1} E(X_e).$$
\end{theorem}\noindent This theorem is the generalization of the classical ergodic theorem established by E. Linderstrauss \cite{lindenstrauss1999pointwise} to the non-commutative setting.

\noindent Our goal is to establish a speed of convergence for \cref{thm111}. Mirroring how the central limit theorem can be extended to weakly dependent processes, we require that $(X_g)$ is ``not too far" from being free. We quantify this through mixing coefficients.



\noindent \textbf{``Free-mixing" coefficients} 
We define free mixing coefficients that quantify how far $(X_g)$ is from being freely independent. 
In classical probability, the dependence of a  stationary sequence $(Z_i)$ is quantified through strong-mixing coefficients, alternatively called $\alpha$-mixing coefficients.  They are defined as $$\alpha(i):=\sup_{A\in \sigma(Z_{-\infty:0})}\sup_{B\in \sigma(Z_{i,\infty})} \Big|P(A,B)-P(A)P(B)\Big|,$$ where  $\sigma(Z_{-\infty:0})$ and $\sigma(Z_{i,\infty})$ designate the sigma-fields of events generated by the observations $\dots, X_{-1},X_0$ and respectively by the observations $ X_{i},X_{i+1},\dots$ The faster $\alpha(i)$ decreases as a function of $i$  the weaker the dependence between the observations $(Z_l)$ is. The central-limit theorems has been extended to dependent sequences by enforcing conditions on the strong mixing coefficients \cite{bolthausen1982central}. This notion has been generalized to general dynamical systems $(Z_g)$ \cite{austern2018limit} by choosing a metric on the underlying group and upper-bounding the correlations between events depending on $\{Z_{g_1}, Z_{g_2}\}$ and events depending on $\{Z_g,~g\in \tilde G\}$ for a subset $\tilde G$ ``far away " from $g_1$ and $g_2.$ 
The free mixing coefficients we define resemble those strong-mixing coefficients. However due to the non-commutativity of the process $(X_g)$ we will need to the control various alternating products of the type $E(a_1b_1a_2b_2a_3)$ for $a_1,a_2,a_3$ and $b_1,b_2$ belonging to some ``far away" algebras.
We write $\mathcal{C}[b]$ the collection of sets that are ``far away" from each other
\begin{equation*}\begin{split}\label{chol}
\mathcal{C}[b]:=\Big\{( G_1, G_2)\big|~   G_1,G_2\subset \G,~ \rm{s.t}~\bar{d}(G_1, G_2)\ge b~\rm{and}~\rm{card}(G_1)\le 2\Big\}
.\end{split}\end{equation*}  
\noindent{For a complex number $\gamma\in \mathbb{C}\setminus \mathbb{R}$ we denote $R_{\gamma}:\A_{\tau}\rightarrow\A$ the following function $R_{\gamma}:A\rightarrow\rm{Im}(\gamma)[A-\gamma1_{\A}]^{-1}$. The function $R_{\gamma}$ is called resolvent and plays a central role in functional calculus. We write $$\mathcal{H}_{\lambda}:=\{R_{\gamma} | ~\rm{Im}(\gamma)>\lambda\}.$$}
If $(G_1,G_2)\in \mathcal{C}[b]$ are ``far away" from each other; and $a_1,a_2,a_3\in \mathcal{F}^{G_1}(X)$ and $b_1,b_2\in \mathcal{F}^{G_2}(X)$ then we hope that $a_1,a_2$ and $a_3$ are almost freely independent from $b_1$ and $b_2$. Our free mixing coefficients will capture this for specific choices of $a_1,a_2,a_3$ and $b_1,b_2$. In this goal for $a\in L_1(A,\tau)$ we write $\overline{a}:=A-E(a)$; we centralize and normalize $X_g$ and define $X_{g}^{\mathcal{N}}:= \frac{X_g-E(X_g)}{\|X_g\|_2}$. For $f\in\mathcal{H}_{\lambda}$ we write $f^{K}:= f\Big(\frac{1}{\sqrt{|K|}}\int_K X_gd|g|\Big)$.
~The free mixing coefficients are defined as 
{\begin{equation}\begin{split}&\label{suis_suisse}\aleph^{j,\lambda}[b|\G]:=\hspace{-2mm}\sup_{{\big(\{g,g'\}, G\big)\in \mathcal{C}[b]}}~\sup_{\substack{f\in \mathcal{H}_{\lambda}}}\max\begin{cases}\Big\|E\Big(f^GX_{g}^{\mathcal{N}} ~\overline{f^G}~X^{\mathcal{N}}_{g'} ~f^G\Big)\Big\|_1\\~\\\Big\|E\Big(f^G\overline{X_{g}^{\mathcal{N}} ~E(f^G) ~X^{\mathcal{N}}_{g'} }~f^G\Big)\Big\|_1\end{cases}\end{split}\end{equation}}
{\begin{equation}\begin{split}&\aleph^s[b|\G]:=\hspace{-2mm}\sup_{{\big(\{g\}, G\cup \{g'\}\big)\in \mathcal{C}[b]}}~\sup_{\substack{Y_{1:3}\in \mathcal{F}_G(X)\\\max_{i\le 3}\|Y_i\|_{\infty}\le 1}}\Big\|E\Big(Y_1~X^{\mathcal{N}}_{g} ~Y_2~X^{\mathcal{N}}_{g'} ~Y_3\Big)\Big\|_1\end{split}\end{equation}}

 \noindent 
We call \emph{free mixing coefficients} of $(X_g)$ the elements $$\aleph[b|\G]=(\aleph^{j,\lambda}[b|\G], \aleph^s[b|\G]).$$ We remark that if $(X_g)$ is free then for all $b\ge1$ we have $\aleph[b|\G]=(0,0)$. This is also the case if $(X_g)$ is free with amalgamation (see \cref{mbap}). We remark that if $(X_g)$ is ergodic then the conditional expectation is trivial: $E(a)= \tau(a)\mathbf{1}_{\A}$ for all $a\in \A.$ Therefore as $\tau$ is a normal trace then the expression of the free mixing coefficients can be simplified to control alternating products of size 4.
\\\noindent Let $A,B$ be two independent random matrices in the Gaussian Unitary Ensemble $GUE(m)$. It is well known that $A$ and $B$ are asymptotically free as the size of the matrices $m$ goes to infinity \cite{anderson2014asymptotically}. We extend this beyond the Gaussian Unitary Ensemble to patterned random matrices and dependent sequences of random matrices. We show under general conditions that ``almost independent" random matrices are also ``almost free". In \cref{good_30} we show that under general conditions if $(X^{i,m})_{i\in \Z}$ is a stationary sequence of random matrices with Gaussian entries then we can bound its free-mixing coefficients in terms of its $\alpha$-mixing coefficients. We present here an example of this for block-independent random matrices.

\begin{example}\label{good_intro}
    Let $(X^{i,m})_{i\in \Z}$ be a stationary sequence of self-adjoint random matrices of size $m\times m$ with centered Gaussian entries.
    Write $(\alpha_m[\cdot])$ the strong-mixing coefficients of $(X^{i,m})_{i\in \Z}$. Let $\mathcal{P}_m=(\mathcal{P}_1^m,\dots,\mathcal{P}_{N_m}^m)$ be a symmetric partition of $[|m|]^2$.
    Suppose that for all indexes $(k_1,l_1)\in \mathcal{P}_{j_1,n}$ and indexes $(k_2,l_2)\in \mathcal{P}_{j_2,n}$ that belong to different elements of the partitions $j_1\ne j_2$, the entries $\big(X^{i,m}_{k_1,l_1}\big)_{{i\in \Z}}$ are independent of $\big(X^{i,m}_{k_2,l_2}\big)_{\substack{i\in \Z}}$. Assume that for all $i,j\le m$ we have $\mathbb{E}(X^{1,m}_{i,j})=m^{-1}$ then the following holds for all $\lambda>0$
    $$\aleph^{j,\lambda}_m[b|\mathbb{Z}]\lesssim \alpha_m[b]^{\frac{\epsilon}{2+\epsilon}}+ \frac{\max_{k\le N_m}\# P_{k}^m}{m}, \qquad \aleph^s_m[b|{\mathbb{Z}}]\lesssim \alpha_m[b]^{\frac{\epsilon}{2+\epsilon}}.$$
\end{example}

\noindent \textbf{Operator-valued radius and Stieljes transform} 
We call a $\F^{\rm{tail}}(X)$-valued radius a completely positive map from $\A$ into $\F^{\rm{tail}}(X)$.  We remark that, as $\F^{\rm{tail}}(X)$ is not necessarily trivial, the mixing coefficients $(\aleph[b|\G])$ can be vanishing without the process $(X_g)$ being ergodic in which case the operator radius might be operator valued. We define the $\F^{\rm{tail}}(X)$-valued Stieljes transform of a self adjoint free random variable $Y$ by: $$S_Y:\gamma \rightarrow E\Big(\Big[Y-\gamma \mathbf{1}_{\mathcal{A}}\Big]^{-1}\Big).$$We say that a  self-adjoint free random variable $Y$ follows an operator valued semi-circular law with radius $\eta$ if its Stiejles transform satisfies: \begin{equation}\begin{split}\label{def}&\eta\big(S_Y(\gamma)\big)S_Y(\gamma)+\gamma S_Y(\gamma)+\mathbf{1}_{\mathcal{A}}=0,\qquad\forall \gamma \in \mathbb{C}\setminus\mathbb{R}
\\&S_Y(\gamma)\sim -\frac{1}{\gamma}\mathbf{1}_{\mathcal{A}},\qquad \rm{as}~\gamma\rightarrow\infty.
\end{split}\end{equation} 
\subsection{Main result for empirical averages.}
 Let $(\A^n)$ be a sequence of Von Neumann algebras with normal faithful and semidefinite trace $(\tau_n)$. Choose $(\mathbb{G}_n)$ to be a sequence of l.c.s.c.H~amenable groups with F\o{}lner sequence $(A_{i,n})$. Let $d_n(\cdot,\cdot)$ be a left-invariant distance over $\mathbb{G}_n$ and denote by $B^n(\cdot,\cdot)$ the induced balls.  
 ~Let $X^n=(X^n_g)_{g\in\G_n}$ be a dynamical system of self-adjoint free random variables. We denote by $\F^{\rm{tail}}(X^n)$ the tail-algebra of $(X^n_g)$ and write $E_n$ the non-commutative conditional expectation on $\F^{\rm{tail}}(X^n)$. Our goal is to study the asymptotics of $$W_n:=\frac{1}{\sqrt{|A_{n,n}|}}\int_{A_{n,n}}X^n_g -E_n(X^n_g) d|g|.$$ In this goal, we write $\mu_n$ the distribution of $W_n$ and denote $(\aleph_n[b|\mathbb{G}_n])$ the free mixing coefficients of $(X^n_g)$. To control how fast those mixing coefficients decrease, we define $$\mathcal{R}^s_n[b]:=\sum_{k\ge b }\Big|\Big[B^n_{k+1}\setminus B^n_k\Big]\bigcap A_{n,n}\Big|~\aleph^s_n[k|\mathbb{G}_n].$$ 
~If the free mixing coefficients decrease fast enough we show that $W_n$ is asymptotically semi-circular. More specifically, we make the following hypothesis: For all $\lambda>0$ \begin{align*}
\limsup_{n\rightarrow \infty}\Big\{\mathcal{R}^s_n[b]+|B^n_b|\aleph_n^{j,\lambda}[b|\G]\Big\}\xrightarrow{b\rightarrow \infty} 0\qquad \qquad  ({H_{\rm{mixing}}})
\end{align*} The distribution of $W_n$ is compared to a semi-circular distribution with radius $\eta_n$ defined as the following completely positive map $$\eta_n:a\rightarrow \int_{A_{n,n}}E_n(X_e^n a  X^n_{g})d|g|.$$We write $S_n(\cdot)$ the operator-valued Stieljes transform of $W_n$, and $S_n^{sc}(\cdot)$ the Stiejles transform of the semi-circular operator $Y^{sc,\eta_n}$ with radius $\eta_n(\cdot)$.
\begin{theorem}
\label{theorem_1} Let $(X^n)$ be a sequence array of self-adjoint free random variables satisfying $\sup_{\substack{n\in \mathbb{N}}}\|X^n\|_{3}<\infty$. Let $(K^n_g)$ be a sequence of $*$-automorphisms satisfying conditions $(H_1)-(H_2)$. Write $X^n_g:=K_g^n(X^n)$ and denote by $(\aleph_n[b|\G_n])$ the free mixing coefficients of $(X^n_g)$; suppose that $({H_{\rm{mixing}}})$ holds.
\noindent Let $\gamma_{x,\nu}=x+i\nu\in \mathbb{C}\setminus\mathbb{R} $ be a complex number with $\nu>0$. Set $(b_n)$ to be a sequence of integers. There is a constant  $C$ not depending of $X^n$, $\G_n$, $n$, $b_n$ or $\gamma_{x,\nu}$ such that if we write { $K=C(\Big[\Big\| {X_{e}^n}\Big\|_{3}^3+\|Y^{sc,\eta_n}\|_3^3\Big]\vee 1\Big)$} we have

\begin{equation*}\begin{split}&\Big\|S_n(\gamma_{x,\nu})-S_n^{sc}(\gamma_{x,\nu})\Big\|_{1}
\\&\le K\Big[ \frac{1 }{\nu^3}\Big(\mathcal{R}^s_{n}[b_n]+\frac{|A_n\triangle B^n_{b_n}A_n|}{|A_{n,n}|}+|B^n_{b_n}|\aleph^{j,\nu}_n[b_n|\mathbb{G}_n]\Big)
+\frac{|B^n_{2b_n}|^2}{\sqrt{|A_{n,n}|}\nu^4}\Big]\longrightarrow 0.
\end{split}\end{equation*}
\end{theorem}

\begin{remark}
We note that the size of the Berry-Esseen bound depends on how fast the free mixing coefficients $(\aleph_n[b|\mathbb{G}_n])$ decrease as a function of $b$. Notably if $\aleph_n[b|\G_n]=0$ for all $b>0$ then we obtain \begin{equation*}\begin{split}&\Big\|S_n(\gamma_{x,\nu})-S_n^{sc}(\gamma_{x,\nu})\Big\|_{1}
=O\Big(\frac{1}{\sqrt{|A_{n,n}|}\nu^4}
\Big).
\end{split}\end{equation*}
Note that in the case where $\G_n=\mathbb{Z}$, and $(X_i^n)$ are freely independent then the upper-bound we obtain is of order $O(\frac{1}{\sqrt{n}\nu^4})$. \cite{kargin2007berry} shows that the optimal rate is indeed of order  $\frac{1}{\sqrt{n}}$; but the term $\nu^{-4}$ is greater than what could be obtained by analytical methods (e.g see \cite{chistyakov2008limit,kargin2007berry}). However, those rely strongly on the fact that the following equation holds when $X$ and $Y$ are freely independent: $S_{X+Y}^{-1}(z)=S_{X}^{-1}(z)+S_{Y}^{-1}(z)+z^{-1}.$ Those methods are not adaptable for free mixing processes $(X_i)$.
\end{remark}
\subsection{Examples}\label{ex_random_matrices}
\noindent In this section, we explore a few illustrative examples  of dynamical systems $(X_g)$. We bound their free mixing coefficients and deduce that their asymptotic distribution is semi-circular.
\subsubsection{Examples of general operators}
\noindent Firstly we note that if $(X_i)$ is a freely independent sequence then its free mixing coefficients are null. However the reverse does not hold and null free mixing coefficients imply something weaker than freeness. Similarly \cite{kargin2007proof} proved a free central limit theorem under the assumption that $(X_i)$ respect the following conditions: 
\begin{itemize}
    \item [($H'_1$).]$\tau\big(X_k Y_1\big)=\tau\big(\big[X_k^2-\tau(X_k^2)\big]Y_1)=0$\quad for all $Y_1\in F_{\mathbb{N}\setminus \{k\}}(X)$
    \item [($H'_2$).]$\tau\big(X_kY_1X_k Y_2)=\tau(X_k^2)~\tau(Y_1)\tau(Y_2)$ \quad for all $Y_1,Y_2\in F_{\mathbb{N}\setminus \{k\}}(X)$.
\end{itemize}
Any process satisfying those conditions has null free mixing coefficients.
\begin{prop}
Let $(X_i)$ be a sequence of identically distributed self-adjoint free random variables satisfying conditions $H'_1$ and $H'_2$. Let $(\aleph[\cdot|\mathbb{Z}])$ denote the free mixing coefficients of $(X_i)$ then we have:
$\aleph[b|\mathbb{Z}]=0$ for all $b\ge 1$ and $(X_i)$ is ergodic.
Therefore $\frac{1}{\sqrt{n}}\sum_{i\le n} \big[X_i-\tau(X_i)\big]$ is asymptotically semi-circular with radius $\tau(X_1^2).$
\end{prop}

\noindent Another important example is quantum exchangeable processes: those are processes whose distribution is invariant under the coaction of quantum permutations.
~These processes play  in free probability an analogous role to the one exchangeable sequences have in classical probability. Indeed,  
\cite{kostler2009a} proved that $(X_i)$ is quantum exchangeable if and only if  conditionally on its tail algebra it is freely independent and identically distributed. 
This notably implies that quantum exchangeable sequences are exchangeable but the reverse does not hold. 
\begin{prop}\label{mbap}Choose $\mathbb{G}=\mathbb{Z}$ and let $(X_i)\in \mathcal{A}^{\mathbb{N}}$ be a quantum exchangeable sequence of free self-adjoint random variables. We have $$\aleph[b|\mathbb{S}(\mathbb{N})]=0,\qquad \forall b>0.$$
This implies that $\frac{1}{\sqrt{n}}\sum_{i\le n} \big[X_i-E(X_i)\big]$ is asymptotically semi-circular with radius $\eta: a\rightarrow E(X_1a X_1).$
\end{prop}

\subsubsection{Examples of random matrices}
\noindent An important class of examples are random matrices. Let $\mathcal{A}_m$ be the set of random matrices of size $m$ with essentially bounded entries. It forms a Von-Neumann algebra with $\tau_m(\cdot)=\frac{1}{m}\mathbb{E}(Tr(\cdot))$ as the normal faithful trace. Finally for any real-valued random variable $Y$ we write $\|Y\|_{L_p}$ its $L_p$ norm such as defined in classical probability. In this subsection we provide examples of free mixing coefficients for dynamical systems of random matrices.


\noindent We say that $(X^{\z,m})_{\z\in \Z^d}$ forms a stationary random field if its distribution is invariant under the following action of $\Z^d$: $\z'\cdot (X^{\z,m}):= (X^{\z+\z',m}).$ Its dependence is measured through the following strong mixing coefficients: $$\alpha_m(i):= \sup_{(\{\z_1,\z_2\},\tilde Z)\in \mathcal{C}[i]}\quad\sup_{\substack{A\in \sigma(X^{\z_1,m},X^{\z_2,m})\\~\\B\in \sigma(X^{\tilde z, m}, ~\tilde z\in \tilde Z)}}\Big|P(A,B)-P(A) P(B)\Big|.$$  Let $(k_m)$ be an increasing sequence of integers that satisfies $k_m\le m$ for all $m\in \mathbb{N}$. We prove, under mixing and moment conditions, that the asymptotic distribution of $\sum_{\z\in \dbracket{k_m}^d}X^{\z,m}$ is asymptotically semi-circular.
\noindent \textbf{Random matrices with patterned entries }
\\\noindent In this section we consider patterned random matrices. Those are random matrices with centered Gaussian entries that are not necessarily identically distributed or independent. The distributions of the entries is described through a sequence of deterministic matrices. 

\noindent To make this precise, let $N_m>0$ be an integer and let  $(A^{s,m})_{s\le N_m}$ be a sequence of self-adjoint deterministic $m\times m$ matrices. We assume that $(A^{s,m})_{s\le N_m}$ are orthogonal to each other meaning that $Tr(A^{s,m}A^{s',m})=0$ for all different choices of $s\ne s'\le N_m$. Let $\Big((Z_{s,\mathbf{z}})_{s\le N_m}\Big)_{\mathbf{z}\in \mathbb{Z}^r}$ be a stationary random field of standard Gaussian vectors meaning that for all $\z\in\mathbb{Z}^r$ we have  $\big(Z_{1,\mathbf{z}},\dots, Z_{N_m,\mathbf{z}}\big)\sim N(0,Id)$. 
The random matrices $(X^{\mathbf{z},m})$ we consider in this subsection have the following form
\begin{align}\label{seule}X^{\mathbf{z},m}=\sum_{s\le N_m}Z_{s,\mathbf{z}} A^{s,m}.\end{align}  Examples of such matrices include Gaussian random matrices with independent blocks, with independent diagonals, or sparse matrices etc\dots This class of patterned random matrices has been studied in the setting of concentration inequalities  \cite{bandeira2021matrix} and discussed in \cite{handel2017structured}. Using the same notations than in \cite{bandeira2021matrix} we write
\begin{align*}&
    \mathcal{V}_m(X^m)^2:=\sup_{\rm{Tr}(|M|^2)\le 1}\sum_{s\le N_m}\big|Tr(A^{s,m}M)\big|^2;
    \\&\sigma_m(X^m)^2:=\|\mathbb{E}\big((X^{\mathbf{1},m})^2\big)\|_{\infty}.
\end{align*}
We observe that under conditions on the size of $\mathcal{V}_m(X^m)$ and $\sigma_m(X^m)$ then we obtain that $\sum_{\z\in \dbracket{k_m}^d}X^{\z,m}$ is asymptotically semi-circular.
\begin{prop}\label{good_30}Let $(X^{\z,m})_{\z\in \Z^d}$ be a stationary random field of random matrices of size $m\times m$ defined as in \cref{seule}.  Write $(\alpha_m[i])$ the strong mixing coefficients of $(X^{\z,m})$. Then, there is a constant $C$ that does not depend on $m$ or $b$ such that
\begin{align*}&
   \aleph_m^{j,\lambda}[b|\mathbb{Z}^r]\le C\Big[\alpha_m[b]^{\frac{\epsilon}{2+\epsilon}}  +\frac{\sigma_m(X^m)^2\mathcal{V}_m(X^m)^2}{\sum_{s\le N_m}\|A^{s,m}\|^2_{2}}\sum_{l\ge 0}l^{r-1}\alpha_m[l]^{\frac{\epsilon}{2+\epsilon}}\Big],
   \\&\aleph_m^s[b|\mathbb{Z}^d]\le C\alpha_m[b]^{\frac{\epsilon}{2+\epsilon}}
\end{align*}
Moreover assume that $\sup_m l^{r-1}\alpha_m[l]^{\frac{\epsilon}{2+\epsilon}}\xrightarrow{l\rightarrow \infty} 0$ and if  $\frac{\sigma_m(X^m)\mathcal{V}_m(X^m)}{\sqrt{\sum_{s\le N_m}\|A^{s,m}\|^2_{2}}}\xrightarrow{m\rightarrow\infty}0$. Then for all $\gamma\in \mathbb{C}\setminus\mathbb{R}$ the following holds $$\big\|S_m(\gamma)-S_m^{sc}(\gamma)\|_1\rightarrow 0$$ where $S_m^{sc}$ is the Stieljes transform of a semi-circular operator with radius $\hat\eta_m:A\rightarrow \sum_{\z\in \Z^d}\frac{1}{m}\mathbb{E}(X^{1,m}AX^{\z,m})$. 
\end{prop}

\vspace{1mm}

\noindent Another important example is the class of jointly exchangeable arrays. Let $X:=(X_{i,j})$ be a random array, we say that it is jointly exchangeable if for all permutations $\pi\in \mathbb{S}(\mathbb{N})$ we have: $X\overset{d}{=}(X_{\pi(i),\pi(j)}).$  We write $\sigma(\mathbb{S}(\mathbb{N}))$ the sigma-field generated by events $A$ satisfying $\mathbb{I}(X\in A)=\mathbb{I}((X_{\pi(i),\pi(j)})\in A)$ for all permutations $\pi\in \mathbb{S}(\mathbb{N});$ and denote $\mathbb{E}(\cdot|\mathbb{S}(\mathbb{N}))$ for the conditional expectation knowing $\sigma(\mathbb{S}(\mathbb{N}))$ .

\noindent We say that a random matrix $Y$ of size $n\times n$ is jointly exchangeable if there is a random exchangeable array $X$ such that: $Y:=(X_{i,j})_{i,j\le n}.$ An important example  are adjacency matrices of exchangeable graphs \cite{kallenberg2006probabilistic}. We define $\mathcal{M}^{\rm{tail}}$ to be the algebra generated by random invariant arrays: $$\mathcal{M}^{\rm{tail}}:=\{a|(a_{i,j})=(a_{\pi(i),\pi(j)})~\forall \pi\in \mathbb{S}(\mathbb{N})\}.$$ Finally we write $\mathcal{M}^{\rm{tail}}_n:=\{y\in L_{\infty}\otimes \mathcal{M}_n(\mathbb{C})\big|~\exists a\in \A^{\rm{tail}}~ \rm{s.t}~y=(a_{i,j})_{i,j\le n} \}$ the set of matrices that can be obtained by truncating elements of $\mathcal{M}^{\rm{tail}}$.
\begin{prop}\label{good_3}Let $(X^{i,m})$ be a sequence of $\sigma(\mathbb{S}(\mathbb{N}))$-conditionally independent and identically distributed random matrices of size $m\times m$. We assume that they are self-adjoint and have  centered entries: $\mathbb{E}(X^{1,m}_{i,j}|\mathbb{S}(\mathbb{N}))=0$. Suppose that $X^{1,m}$ is jointly exchangeable and that the entries $(X^{1,m}_{i,j})$ admit a second moment. We denote by $(\aleph_m[i|\mathbb{S}(\mathbb{N})])$ the free mixing coefficients of $(\frac{1}{\sqrt{m}}X^{i,m})$. Suppose that the following conditions hold $\sup_{ m}\frac{\max_{i,j\le m}\|X^{1,m}_{i,j}\|_2}{\min_{i,j\le m}\|X^{1,m}_{i,j}\|_2}<\infty$ and $\sup_{ m}\max_{i,j\le m}\|X^{1,m}_{i,j}\|_2<\infty$ . Then there is a constant $C$ independent of $m$ such that 
$$\aleph^s_m[b|\Z]=0,\qquad \aleph^{j,\lambda}_m[b|\Z]\le  \frac{C}{\sqrt{m}},\quad \forall b\ge1.$$
 We denote $\eta_n$ the following mapping $\eta_n:a\rightarrow E_{m}(X^{1,m}a X^{1,m})$ then the distribution of $\frac{1}{{m}}\sum_{i\le m} X^{i,m}$ converges to  a semi-circular distribution with radius $\eta_m$.
\end{prop}

\section{Generalization to u-statistics}\label{u_st}
\subsection{Notations and definitions}
 In classical probability, u-statistics are a key quantity, and under general conditions their limiting distribution is well known to be Gaussian see e.g \cite{van2000asymptotic}. Let $(Y_i)_{i\in \Z}$ be a stationary process with observations taking value in a Borel space $\mathcal{X}$. Let $k\in \mathbb{N}$ be an integer and let $h:\mathcal{X}^k\rightarrow\mathbb{R}$ be a measurable function. We call the statistics $s_n:=\frac{1}{n^k}\sum_{i_1,\dots,i_k\in \dbracket{n}}h(Y_{i_1},\dots,Y_{i_k})$ a u-statistics. For example, if we take $k=2$ then the following quantity is a u-statistics $\frac{1}{n^2}\sum_{i,j\le n}h(Y_i,Y_j)$. Under general moment conditions and mixing conditions \cite{austern2018limit} we know that $s_n$ is asymptotically Gaussian and converges at a rate of $n^{-1/2}.$ We remark that if we define the process $Z:=(Z_{\z_1,\dots,\z_k})_{\z_1,\dots,\z_k\in\Z}$ as $Z_{\z}=Z_{\z_1,\dots,\z_k}:=h(Y_{\z_1},\dots,Y_{\z_k})$ then $s_n$ can be re-expressed as the following empirical average $s_n=\frac{1}{n^k}\sum_{\z\in \dbracket{n}^k} Z_{\z}$.
~While the distribution of $Z$ is not invariant under the natural action of $\mathbb{Z}^k$:~~$(j_1,\dots,j_k)\cdot Z:=(Z_{i_1+j_1,\dots,i_k+j_k})$; it is however invariant under the induced joint action (or diagonal action): $(j,\dots,j)\cdot Z$. In other words $s_n$ can be seen as an empirical average of a jointly invariant process $Z$.  We use this insight to extend the definition of u-statistics to non-commutative processes.


\vspace{2.7mm}


\noindent \textbf{Non-commutative u-statistics} Let $\G$ be a l.c.s.c.H amenable group with F\o{}lner sequence $(A_n)$ and let $\A$ be a Von-Neumann algebra with normal faithful and semi-definite trace $\tau$. Let $k\in \mathbb{N}$ be an integer, we study the asymptotics of processes indexed by $\G^k$. We remark that the sequence $(A_n^k)$ is a F\o{}lner sequence of the amenable group $\G^k$. Let $(X_{\g})_{\g\in\G^k}$ be a sequence of elements of $\A_{\tau}$ indexed by $\G^k$, our goal is to study the asymptotic of the average $$\bar{X_n}:=\frac{1}{|A_n|^k}\int_{A_n^k}X_{\g}d|\g|.$$ We observe that the form of $\bar{X_n}$ reassembles the one of a u-statistics and indeed, under conditions on the distribution of $(X_{\g})_{\g\in\G^k}$, we will call $\bar{X_n}$ a non-commutative u-statistics. We denote by $\g:=(\g_1,\dots,\g_k)$ the elements of $\G^k$ and for all subset $G\subset \G$ we write $$\mathcal{D}_k(H):=\{\g|\g\in G^k,~\g_{i}=\g_{j}~\forall i,j \le k\}$$ and call $\mathcal{D}_k(\G)$ the ``diagonal" of $\G^k$. For example if $\G=\Z$ is the group of integers and $k=2$ then the diagonal $\mathcal{D}(\Z^2)$ is the set $\{(i,i)~i\in \Z\}$.
We say that $(X_{\g})$ is jointly invariant with respect to $\G^k$ if it satisfies: 

\begin{itemize} \item [$(H_3)$] For all  $\g\in \mathcal{D}_k(\G)$ and all  $\g_1,\dots,\g_d\in \G^k$, we have $$(X_{\g_1},\dots,X_{\g_d})\overset{d}{=} (X_{\g\g_1},\dots,X_{\g\g_d}).$$
\end{itemize} We remark that if $k=1$ then $(X_{\g})$ is jointly invariant if and only if it is an invariant dynamical system. 
In general, if $(X_{\g})$ is jointly invariant with respect to $\G^k$ then we call $\bar{X^n}$ a non-commutative a u-statistics. The following proposition demonstrates that the classical notion of u-statistics can be embedded into this framework.
\begin{prop}\label{luna}
Let $\Phi:\A^k\rightarrow \A$ be a continuous function.
Choose $\G$ to be a l.c.s.c.H~group defining an action $(K'_g)$ on $\A$. Let $Y\in \A$ be self adjoint, we write $Y_g:=K'_g(Y)$ and denote $X_{\g}:= \Phi(Y_{\g_1},\dots,Y_{\g_k})$ for all $\g\in \G^k$. 
 \\\noindent The net $(X_{\g})$ is jointly invariant with respect to $\G^k$ and  the following is a u-statistic $$\frac{1}{|A_n^k|}\int_{A_n^k}\Phi(Y_{\g_1},\dots,Y_{\g_{k}})d|\g|=\frac{1}{|A_n^k|}\int_{A_n^k}X_{\g} d|\g|.$$  
\end{prop}
\noindent \textbf{$\mathcal{F}^d_G(\cdot)$ and tail algebra with respect to the diagonal action}

\noindent Playing a similar role than the tail-algebra in the the empirical case is the tail-algebra with respect to the diagonal action. To define it, we introduce $\mathcal{F}^d_G(X)\subset \mathcal{A}_{\tau}$ the unital-subalgebra of $\tau$-measurable operators generated by $\{X_{\g}; \g\in G\}$ where $ G\subset \G^k$. For all $\g\in \G^k$ we denote:
 $$\Ftu(\g):= \bigcap_{i\in \mathbb{N}}\mathcal{F}^d_{\mathcal{D}_k(\G\setminus A_i)\g }(X).$$  
 We remark that if $k=1$ then this corresponds to the notion of tail-algebra defined in section 2: $\Ftu(\g)=\F^{\rm{tail}}(X)$ for all $\g\in \G^k$.
 However, when $k>1$ we know that $\mathcal{D}_k(\G)\g \subsetneqq \G^k$ therefore the algebra $\Ftu(\g)$ will, in general, be different for different group elements $\g\in \G^k.$
 We define the tail sub-algebra of $(X_{\g})$ as the smallest algebra generated by $\bigcup_{\g\in \mathbb{G}^k} \mathcal{F}^{d,~\rm{tail}}(\g)$ and denote it by $\F_{\mathcal{D}_k(\G)}^{\rm{tail}}(X)$. We set $E_{\mathcal{D}}(\cdot)$ to be the non-commutative conditional expectation on $\F_{\mathcal{D}_k(\G)}^{\rm{tail}}(X)$. 


\vspace{3mm}

\noindent\textbf{Free mixing coefficients for u-statistics}

\noindent To prove that the limit of $S_n$ is semi-circular, we define a new notion of free-mixing and explain how it relates to the previous one defined in \cref{suis_suisse}. 
In classical probability, if a stationary process $(X_i)$ is strongly-mixing then $h(X_{i},X_{j})$ and $h(X_{k},X_l)$ become increasingly independent as the  distance between $\{i,j\}$ and $\{k,l\}$ grows to infinity.  We generalize this notion to the non-commutative setting. Let ${d}_k(\cdot,\cdot)$  be the pseudo distance on $\G^k$ defined as: $${d}_k(\g,\g'):=\min_{i,j\le k}d(\g_i,\g'_j).$$ We denote $\overline{B}_{k}(\g, b)$ the induced ball of radius $b$ around $\g$ and $\bar{d}_k(\cdot,\cdot)$ the induced Hausdroff pseudo-distance. 
We write 
\begin{equation*}\begin{split}&\mathcal{C}_{k}[b]:=\big\{( G_1,{G}_2)\big| G_1, G_2\subset \G^k,~~\bar{d}_k\big( G_1,  G_2\big) \ge b,\quad \rm{and}~ \rm{card}(G_1)\le 2\big\}.
\end{split}\end{equation*}{For a complex number $\gamma\in \mathbb{C}\setminus \mathbb{R}$ we denote $R_{\gamma}:\A_{\tau}\rightarrow\A$ the following function $R_{\gamma}:A\rightarrow\rm{Im}(\gamma)[A-\gamma1_{\A}]^{-1}$.  We write $\mathcal{H}_{\lambda}:=\{R_{\gamma} | ~\rm{Im}(\gamma)>0\}.$} For $K\subset\G^k$ and all $f\in\mathcal{H}_{\lambda}$ we write $f^{K}:= f\Big(\frac{1}{\sqrt{|K|}}\int_K X_{\g}d|\g|\Big)$. For an operator $a\in L_1(\A,{\tau})$ we write $\overline{a}:=a-E_{\mathcal{D}}(a).$ Finally for a process $(X_{\g})$ we use the shorthand  $$X^{\mathcal{N}}_{\g}:=\frac{X_{\g}-E_{\mathcal{D}}(X_{\g})}{\|X_{\g}\|_2}.$$
~ We define the \emph{free global mixing coefficients} of $(X_{\g})$ as $\aleph^{*}[\cdot|\G^k]:=(\aleph^{*,j,\lambda}[\cdot|\G^k], \aleph^{*,s}[\cdot|\G^k])$ where we have:
{\begin{equation*}\begin{split}&\label{suis_suisse_2}
 \aleph^{*,s}[b|\G^k]:=\sup_{\substack{\big(\{\g\},\tilde G\cup \{\g'\}\big)\in \mathcal{C}_{k}[b]}}\sup_{\substack{Y_{1:3}\in \mathcal{F}^{d}_G(X)\\\max_{i\le m}\|Y_i\|_{\infty}\le 1}}\Big\|E_{\mathcal{D}}\Big(Y_1 ~X^{\mathcal{N}}_{\g} ~Y_2~X^{\mathcal{N}}_{\g'}~Y_3 \Big)\Big\|_1
 \\&\aleph^{*,j,\lambda}[b|\G^k]:=\sup_{\substack{\big(\{\g,\g'\},\tilde G\big)\in \mathcal{C}_{k}[b]}}\quad\sup_{f\in \mathcal{H}_{\lambda}}\max\begin{cases}\Big\|E_{\mathcal{D}}\Big(f^GX_{\g}^{\mathcal{N}} ~\overline{f^G}~X^{\mathcal{N}}_{\g'} ~f^G\Big)\Big\|_1\\~\\\Big\|E_{\mathcal{D}}\Big(f^G\overline{X_{\g}^{\mathcal{N}} ~E(f^G) ~X^{\mathcal{N}}_{\g'} }~f^G\Big)\Big\|_1\end{cases}
\end{split}\end{equation*}} 
\noindent  We observe that $(\aleph^{*}[\cdot|\G^k])$ is very similar to the free mixing coefficients defined in \cref{suis_suisse} where the group $\G$ has been replaced by $\G^k$ and where the metric has been replaced by a pseudo distance. 
Indeed when $k=1$ then the two type of free mixing coefficients concur $$\aleph^*[\cdot|\G]=\aleph[\cdot|\G]$$

\subsection{Main results for u-statistics}
Let $(\A^n)$ be a sequence of Von Neumann algebras with normal faithful semi-definite trace $(\tau_n)$. Define $(\mathbb{G}_n)$ to be a sequence of l.c.s.c.H~amenable groups with F\o{}lner sequence $(A_{i,n})$. Let metric $d_n(\cdot,\cdot)$ be a left-invariant metric on $\mathbb{G}_n$ and denote by $B^n(g,b)$ the ball of radius $b$ around the element $g\in \G_n$. We shorthand $B^n(b):=B^n(e,b)$. Let $(k_n)$ be a sequence of integers and choose $(X_{\g}^n)$ to be a sequence of self-adjoint free random variables indexed by $\G_n^{k_n}$. We denote $E_{\mathcal{D}_n}$ the non-commutative conditional expectation on $\F^{\rm{tail}}_{\mathcal{D}(\G_n^{k_n})}(X^n)$and denote $(\aleph^{*}_n[b|\mathbb{G}^{k_n}_n])$ the free global mixing coefficients of $(X_{\g}^n)$.
~We define $\eta_n$ to be the following completely positive map $$\eta_n:a\rightarrow \frac{1}{|A_{n,n}|^{2k_n-1}}\int_{A_{n,n}^{2k_n}}E_{\mathcal{D}_n}\Big(X^n_{\g}a X^n_{\g'}\Big)d|\g|d|\g'|.$$ 
 Our goal is to study the asymptotic of {$$W_n:=\frac{1}{{|A_{n,n}|}^{k_n-\frac{1}{2}}}\int_{A_{n,n}^{k_n}}X^n_{\g}-E_{\mathcal{D}_n}(X^n_{\g})d|\g|$$} and we write $\mu_n$ the distribution of $W_n$. We prove that ${\mu}_n$ converges to a semi-circular law.
 Finally we write $S_n(\cdot)$ the operator-valued Stieljes transform of $W_n$, and $S_n^{sc}(\cdot)$ the Stiejles transform of the semi-circular operator $Y^{sc,\eta_n}$ with radius $\eta_n(\cdot)$.
\begin{theorem}\label{theorem_3}  Let $(k_n)$ be a sequence of integers; and let $(X^n_{\g})$  be a triangular array of self-adjoint free random variables. Suppose that \begin{itemize}\item[i.] $\sup_{n\in\mathbb{N}}\sup_{\g\in \G_n^{k_n}}\|X^n_{\g}\|_3<\infty$  \item[ii.] $\sum_{b\ge 0}\Big|\Big[B^n_{k+1}\setminus B^n_k\Big]\bigcap A_{n,n}\Big|~\aleph^{*,s}_n[b|\mathbb{G}^{k_n}_n] <\infty$
\end{itemize}  Then  there is a constant {$K=O(\Big[\sup_{\g\in \G^{k_n}}\Big\| {X_{\g}^n}\Big\|_{3}^3+\|Y^{sc,\eta_n}\|_3^3\Big]\vee 1\Big)$} such that for all sequences $(b_n)$ the following holds

\begin{equation*}\begin{split}&\Big\|S_n(\gamma_{x,\nu})-S_n^{sc}(\gamma_{x,\nu})\Big\|_{1}
\le K\Big[\frac{k_n^4}{\sqrt{|A_{n,n}|}\nu^4}
+\frac{k_n^2 \Big[\mathcal{R}^s_{n}[b_n]+|B^n_{b_n}|\aleph^{*,j,\nu}_n[b_n|\G^{k_n}_n]\Big]}{\nu^3}
\Big].
\end{split}\end{equation*}
where $\mathcal{R}^s_n[b]:=\sum_{k\ge b }\Big|\Big[B^n_{k+1}\setminus B^n_k\Big]\bigcap A_{n,n}\Big|~\aleph^{*,s }_n[k|\mathbb{G}^{k_n}_n]$.
\end{theorem}
\vspace{10mm}

\noindent We note that the candidate radius $\eta_n$ is more complex than one would expect as it involves all pairs $(\g,\g')$ in  $A_{n,n}^{2k_n}$. As $(X^n_g)$ is not invariant under the full action of $\G_n^{k_n}$, $\eta_n$ cannot be directly simplified. This contrasts with classical probability where the exact form of the asymptotic variance of u-statistics is known. For example if $(X_i)$ is a stationary sequence and $h$ is symmetric in its coordinates, under mixing conditions we know that $\sum_{i,j\le n}h(X_i,X_j)$ is asymptotically normal with variance: $\sigma^2:=4\sum_k\rm{cov}\big(\bar{h}(X_1,\cdot),~\bar{h}(X_k,\cdot)\big), $ where we wrote $\bar{h}(X_1,\cdot):=\lim_{l\rightarrow \infty} \frac{1}{l}\sum_{k\in \dbracket{ l}}h(X_1,X_k)$. To bridge this gap, we prove that $\eta_n$ can indeed be further simplified if some additional mixing conditions are also imposed. To make those more intuitive we remember that if $(X_i)$ is a stationary sequence then the distribution of $h(X_{0},X_{j})$ and $h(X_{0},X_l)$ become more similar as $j$ and $l$ get further away from $0$. In this paper, we call this behavior \emph{marginal mixing}. We generalize it to the non-commutative setting. 
 In this goal, we define the following set \begin{equation*}\begin{split}&\mathcal{C}_{n}^{i}[b]:=\Big\{(\g^1,\g^2,\g^3)\in \G_n^{k_n}\Big|\overline{d}\big(\{\g^1_i,\g^3_i\}, [\g^1,\g^2]^{\setminus i}\big)\ge b~\&~\g^1_j=\g^3_j,~\forall i\ne j\Big\},\end{split}\end{equation*}  where we have defined $[\g^1,\g^2]^{\setminus i}:=\{\g^1_l|~l\ne i\}\bigcup \{\g^2_l|~l\le k_n\}\subset \G_n$. 
We call  $(\aleph^{\rm{m}}_n[b|\G_n^{k_n}])$ the free marginal mixing coefficients of $(X_{\g})$ and define them as as
 
 \vspace{1.5mm}
 
 {\begin{equation*}\begin{split}\aleph^{\rm{m}}_n[b|\G_n^{k_n}]:= \hspace{-2mm}\sup_{\substack{\big(\g^1,\g^2,\g^3\big)\in \mathcal{C}_{n}^i[b]}}\quad\sup_{\substack{Y\in \F^{\rm{tail}}_{\mathcal{D}(\G_n^{k_n})}(X^n)\\\|Y\|_{\infty}\le 1}} \quad \begin{cases}\Big\|E_{\mathcal{D}_n}\Big(\big[{X^{n,\mathcal{N}}_{\g^1}-X^{n,\mathcal{N}}_{\g^3}}\big] ~Y~{X^{n,\mathcal{N}}_{\g^2}} \Big)\Big\|_1\\ \Big\|E_{\mathcal{D}_n}\Big({X^{n,\mathcal{N}}_{\g^2}}  ~Y~\big[{X^{n,\mathcal{N}}_{\g^1}-X^{n,\mathcal{N}}_{\g^3}}\big] \Big)\Big\|_1\end{cases}
\end{split}\end{equation*}} 

\vspace{2mm}

\noindent If the free marginal mixing coefficients decay fast enough then $\eta_n$ converges to the following radius: $$\tilde\eta_n:a\rightarrow \sum_{i,j\le k_n}\int_{A_{n,n}}E_n(\bar{X}^n_{i,e} a \bar{X}^n_{j,g})d|g|,$$ where $\bar{X}^n_{i,g}:=\lim_{p\rightarrow \infty} \frac{1}{A_{p,n}^{k_n-1}}\int_{\mathcal{I}_p(i,g)}X^n_{\g}d|\g|$ and $\mathcal{I}_p(i,g)=\{\theta|\theta\in A_{p,n}^{k_n}~\rm{s.t}~\theta_i=g\}.$ 
Denote by $S_n(\cdot)$ the operator-valued Stieljes transform of $W_n$, and $S_n^{sc}(\cdot)$ the Stiejles transform of the operator valued semi-circular operator $Y^{sc,\eta_n*}$ with radius $\tilde\eta_n(\cdot)$.

\begin{theorem}\label{theorem_4}Let $(k_n)$ be a sequence of integers; and let $(X^n_{\g})$  be a triangular array of self-adjoint free random variables that are jointly invariant with respect  to $\G_n^{k_n}$. Suppose that \begin{itemize}\item[i.] $X^n_{\g}\in L_3(\A_{\tau}^n,\tau)$ and  $E_{\mathcal{D}_n}(X^n_{\g})=0$ for all $\g\in \G^{k_n}_n$ \item[ii.] $\sum_{b\ge 0}\Big|\Big[B^n_{k+1}\setminus B^n_k\Big]\bigcap A_{n,n}\Big|~\aleph^{*,s}_n[b|\mathbb{G}^{k_n}_n] <\infty$
\item[iii.] $\mathcal{T}_n^m:=\sum_{k\ge 0 }\Big|\Big[B^n_{k+1}\setminus B^n_k\Big]\bigcap A_{n,n}\Big|~\aleph^{m}_n[k|\mathbb{G}^{k_n}_n]=o(|A_n|).$
\end{itemize}Then there is a constant {$K=O(\Big[\sup_{\g\in \G^{k_n}}\Big\| {X_{\g}^n}\Big\|_{3}^3+\|Y^{sc,\eta_n}\|_3^3\Big]\vee 1\Big)$} such that for all sequences $(b_n)$ the following holds

\begin{equation*}\begin{split}&\Big\|S_n(\gamma_{x,\nu})-S_n^{sc}(\gamma_{x,\nu})\Big\|_{1}
\\&\le  K\Big[\frac{|A_{n,n}\triangle B^n_{b_n}A_{n,n}|+\mathcal{T}^m_n}{\nu^3|A_{n,n}|}
+\frac{k_n^4|B^n_{2b_n}|^2}{\sqrt{|A_{n,n}|}\nu^4}
+\frac{k_n^2}{\nu^3}\Big[\mathcal{R}^s_{n}[b_n]+|B^n_{b_n}|\aleph_n^{*,j,\nu}[b_n|\mathbb{G}^{k_n}_n]\Big]
\Big].
\end{split}\end{equation*}

\noindent where we set $\mathcal{R}^s_n[b]:=\sum_{k\ge b }\Big|\Big[B^n_{k+1}\setminus B^n_k\Big]\bigcap A_{n,n}\Big|~\aleph^{*,s}_n[k|\mathbb{G}^{k_n}_n]$.
\end{theorem}

\subsection{Examples}\label{ex_u_stat}In this section we present examples of u-statistics, bound their free mixing coefficients and prove that they are asymptotically semi-circular.
\subsection{Examples for general operators}
\noindent We denote $\mathcal{M}(\A)$ the algebra of arrays $x:=(x_{i,j})_{i,j\in \Z}$ with entries in $\A$. An important subclass is the set of arrays $X=(X_{i,j})$ whose distribution is invariant under the joint action of $\Z$: $(z,z)\cdot X:=(X_{z+i,z+j})$ and is quantum exchangeable on disjoint coordinates. To make this precise, for every $k>0$ we say that $\big(\{\z^j_1,\dots,\z^j_k\}\big):=(I_j)\subset \prod (\Z^2)^k$ is a sequence of coordinate-disjoint sets if for all $j_1\ne j_2$ the following holds $$\{\z^{j_1}_{1,1},\z^{j_1}_{1,2},\dots,\z^{j_1}_{k,1},\z^{j_1}_{k,2}\}\bigcap\{\z^{j_2}_{1,1},\z^{j_2}_{1,2},\dots,\z^{j_2}_{k,1},\z^{j_2}_{k,2}\}=\emptyset.$$ For all 
$I\subset (\Z^2)^k$, we write $\sigma_{I}$ the unital algebra generated by $(X_{\z})_{\z\in I}$. We say that $X$ is quantum exchangeable on disjoint coordinates if for all $k\ge 1$, all sequences $(I_j)\subset (\Z^2)^k$ and all sequences of free random variables$(Y_j)\in \prod_j \sigma_{I_j}$ the sequence $(Y_j)$ is quantum-exchangeable.

We study the free mixing coefficients~of such an element $X$.
\begin{prop}\label{prop_800}
Take $\G=\Z$ and let $X\in \mathcal{M}(\A)$ be invariant under the joint action of the $\Z$ and quantum exchangeable on disjoint coordinates.  Write $(\aleph^*[b|\Z^2])$ and $(\aleph^m[b|\Z^2])$ respectively for the marginal and global free mixing coefficients of $(Z_{\z})$. We have: $$\aleph^{*,s}[b|\Z^2]=\aleph^{*,j,\lambda}[b|\Z^2]=0,\qquad \aleph^m[b|\Z^2]=0,\qquad \forall b>0,~\lambda>0.$$
Therefore $\frac{1}{n^{\frac{3}{2}}}\sum_{\z\in \dbracket{n}^2} Z_{\z}  $ is asymptotically semi-circular with a radius satisfying $\eta(a)=4 E_{\Z^2}(X_{1,2} a X_{1,3})$.
\end{prop}

\noindent Secondly we bound the free global mixing coefficients of the classical u-statistics presented in \cref{luna}. Those will depend on the properties of the free mixing properties of the underlying process $(Y_{g})$ through a slightly more complex notion of free mixing than the one defined in \cref{suis_suisse}. 
In this goal, for $a\in \A_{\tau}$ we write $\overline{a}:=a-E_{\mathcal{D}}(a)$ and we define \begin{equation*}\begin{split}&\aleph^{'}[b|\G]:=\sup_{\substack{\big(G_1\cup G_2,\tilde G\big)\in \mathcal{C}[b]\\~\\|G_1|,|G_2|\le k}}\quad \sup_{\substack{(Y_1,Y_2,Y_3)\in \mathcal{F}_{\tilde G}(Y)\\~\\|Y_1\|_{\infty},\|Y_2\|_{\infty}, \|Y_3\|_{\infty}\le 1\\~\\(Z_1,Z_2)\in \mathcal{F}_{G_1}(Y)\times  \mathcal{F}_{G_2}(Y)\\~\\|Z_1\|_{2},\|Z_2\|_{2}\le 1}}\max\begin{cases}\Big\|E\Big(Y_1{Z_1} \overline{Y_2}{Z_2} Y_3\Big)\Big\|_1\\\Big\|E\Big(Y_1\overline{Z_1 E(Y_2)Z_2} Y_3\Big)\Big\|_1\end{cases}.
\end{split}\end{equation*}

\noindent For simplicity we suppose that $\Phi\in \mathbb{C}\big<x_1,\dots,x_k,x^*_1,\dots,x^*_k\big>$ is a polynomial. A similar result holds for any function that can be successfully approximated by polynomials. 

\begin{prop}\label{u_stat_1} Let $(\A,\tau,*)$ be a Von-Neumann algebra.
Let $\Phi:\A^k\rightarrow \A$ be a polynomial.  Let $(K'_g)$ be a net of *-automorphisms from $\A$ into itself. We choose $Y\in \A$ to be a self-adjoint free random variable and write $Y_g:=K'_g(Y)$. Define $Z:=\Phi(Y_e,\dots,Y_e)$ and have $Z_{\g}:=\Phi(Y_{\g_1},\dots,Y_{\g_k}).$
~Let $\aleph^{*}[\cdot|\mathbb{G}^k]$ denote the free global mixing coefficients of $\Big(Z_{\g}\Big)$. We have: \begin{equation}\label{nels}\aleph^{\rm{*,s}}[b|\G^k]\le \sup_{\g\in \G^k}\frac{\|Z_{\g}\|_{\infty}^2}{\|Z_{\g}\|^2_2} ~~\aleph^{'}[b|\G],\qquad \aleph^{\rm{*,j,\lambda}}[b|\G^k]\le \sup_{\g\in \G^k}\frac{\|Z_{\g}\|_{\infty}^2}{\|Z_{\g}\|^2_2} ~~\aleph^{'}[b|\G] ,\quad \forall b\in \mathbb{N}.\end{equation} 
Therefore if $(Y_g)$ is a freely independent sequence then $\aleph^{\rm{*,s}}[b|\G^k]=\aleph^{\rm{*,j}}[b|\G^k]=0$ for all $b>0$.
  
\end{prop}

\subsection{Examples of u-statistics for random matrices with patterned entries}

\noindent
In this subsection we consider $u-$statistics that are functions of patterned matrices. In this goal let $N_m>0$ be an integer and let  $(A^{s,m})_{s\le N_m}$ be a sequence of self-adjoint deterministic $m\times m$ matrices. We assume that $(A^{s,m})_{s\le N_m}$ are orthogonal to each other: $Tr(A^{s,m}A^{s',m})=0$ for all different choices of $s\ne s'\le N_m$. Let $\Big((Z_{s,i})_{s\le N_m}\Big)_{i\in \mathbb{Z}}$ be a stationary sequence of standard Gaussian vectors meaning that for all $i\in\mathbb{Z}$ we have that $\big(Z_{1,i},\dots, Z_{N_m,i}\big)\sim N(0,Id)$ is a standard Gaussian vector. 
We define $(Y^{i,m})$ to be random matrices that have the following form
\begin{align}\label{21} Y^{i,m}=\sum_{s\le N_m}Z_{s,i} A^{s,m}.\end{align} In \cref{good_30} we showed that the empirical average of $(Y^{i,m})$ is asymptotically semi-circular. In this subsection, we prove that in general u-statistics of $(Y^{i,m})$ are also asymptotically semi-circular. Using the same notations than in \cite{bandeira2021matrix} we write
\begin{align*}&\mathcal{V}_m(Y^m)^2:=\sup_{\rm{Tr}(|M|^2)\le 1}\sum_{s\le N_m}\big|Tr(A^{s,m}M)\big|^2;
\\&\sigma_m(Y^m)^2:=\|\mathbb{E}\big((Y^{{1},m})^2\big)\|_{\infty}.
\end{align*}
Moreover for all $p\ge 1$ we also write 
\begin{align*} &s_{m,p}:= \max_{i,j\le m}\Big(\sum_{s\le N_m}|A^{s,m}_{i,j}|\Big)^{p} \wedge 1\\& \tilde s_{m,p}:=\sup_{1\le l\le p}\Big(\frac{1}{m^{1-\epsilon/2}}\sum_{i,j}\Big[\Big(\sum_{s\le N_m}(A^{s,m}_{i,j})^2\Big)^{l(1+\frac{\epsilon}{2})}+\sum_{s\le N_m}(A^{s,m}_{i,j})^{l(2+\epsilon)}\Big]\Big)^{\frac{2}{2+\epsilon}}.\end{align*}
\noindent Let $p\in \mathbb{N}$, we will consider polynomials of degree $p$. Choose $\Phi_m:M_m(\mathbb{C})\times M_m(\mathbb{C})\rightarrow M_m(\mathbb{C})$ to be a polynomial of degree $p$ that is symmetric: $\Phi_m(A,B)=\Phi_m(B,A)$ and is such that $\Phi_m(A,B)^*=\Phi_m(A^*,B^*)$ for all $A,B\in M_m(\mathbb{C})$. We define $$X^{(i,j),m}:=\Phi_m(Y^{i,m},Y^{j,m}).$$

\noindent In this section, we prove that $\sum_{i,j\le m} X^{(i,j),m}$ is asymptotically semi-circular.
\begin{prop}\label{marre_2}
Let $(Y^{i,m})$ be a stationary sequence of matrices of size $m\times m$ that are defined as described in \cref{21}. Let $(\Phi_m)$ be a sequence of symmetric polynomials of degree $p$ that are such that $\Phi_m(A,B)^*=\Phi_m(A^*,B^*)$ for all $A,B\in M_m(\mathbb{C})$.
 We write for all $\z\in \Z^2$  $$X^{\z,m}:= \Phi_m(Y^{\z_1,m},Y^{\z_2,m}).$$ We suppose that for all $i,j\le m$ we have $\mathbb{E}(X^{\z,m}_{i,j})=0$. Define $(\alpha_m(\cdot))$ to be the $\alpha$-mixing coefficients of $(Y^{i,m})$. Moreover we denote by $(\aleph^*_m[\cdot|\Z^2])$ the free global mixing coefficients of  $(X^{\z,m})$. Suppose that there is an $\epsilon>0$ such that $$\sup_n\sum_{b\ge 0} \alpha_n(b)^{\frac{\epsilon}{2+\epsilon}}<\infty.$$ Then the following holds
 \begin{align*}&
     \aleph_n^{*,j,\lambda}[b|\Z^2]\lesssim \frac{1}{\inf_{\z\in \Z^2}\|X^{\z,m}\|_2^2}\Big[\alpha_m[b]^{\frac{\epsilon}{2+\epsilon}}\tilde s_{m,p}+\sigma_m(Y^m)^2\mathcal{V}_m(Y^m)^2s_{m,p-1}^4(1+\log(N_m))^{2(p-1)}\Big]
     \\&     \aleph_n^{*,s}[b|\Z^2]\lesssim \frac{1}{\inf_{\z\in \Z^2}\|X^{\z,m}\|_2^2}\alpha_m[b]^{\frac{\epsilon}{2+\epsilon}}\tilde s_{m,p}.
 \end{align*}
\end{prop}
\bibliographystyle{plain}

\bibliography{references}

\begin{thebibliography}{10}

\bibitem{adamczak2016circular}
Rados{\l}aw Adamczak, Djalil Chafa{\"\i}, and Pawe{\l} Wolff.
\newblock Circular law for random matrices with exchangeable entries.
\newblock {\em Random Structures \& Algorithms}, 48(3):454--479, 2016.

\bibitem{anderson2014asymptotically}
Greg~W Anderson and Brendan Farrell.
\newblock Asymptotically liberating sequences of random unitary matrices.
\newblock {\em Advances in Mathematics}, 255:381--413, 2014.

\bibitem{arnold1971wigner}
Ludwig Arnold.
\newblock On wigner's semicircle law for the eigenvalues of random matrices.
\newblock {\em Zeitschrift f{\"u}r Wahrscheinlichkeitstheorie und verwandte
  Gebiete}, 19(3):191--198, 1971.

\bibitem{austern2018limit}
Morgane Austern and Peter Orbanz.
\newblock Limit theorems for distributions invariant under groups of
  transformations.
\newblock {\em The Annals of Statistics}, 50(4):1960--1991, 2022.

\bibitem{bandeira2021matrix}
Afonso~S Bandeira, March~T Boedihardjo, and Ramon van Handel.
\newblock Matrix concentration inequalities and free probability.
\newblock {\em arXiv preprint arXiv:2108.06312}, 2021.

\bibitem{banna2018operator}
Marwa Banna and Guillaume C{\'e}bron.
\newblock Operator-valued matrices with free or exchangeable entries.
\newblock {\em arXiv preprint arXiv:1811.05373}, 2018.

\bibitem{bercovici1996law}
Hari Bercovici and Vittorino Pata.
\newblock The law of large numbers for free identically distributed random
  variables.
\newblock {\em The Annals of Probability}, pages 453--465, 1996.

\bibitem{bercovici1993free}
Hari Bercovici and Dan Voiculescu.
\newblock Free convolution of measures with unbounded support.
\newblock {\em Indiana University Mathematics Journal}, 42(3):733--773, 1993.

\bibitem{bolthausen1982central}
Erwin Bolthausen.
\newblock On the central limit theorem for stationary mixing random fields.
\newblock {\em The Annals of Probability}, pages 1047--1050, 1982.

\bibitem{chistyakov2008limit}
Gennadii~P Chistyakov, Friedrich G{\"o}tze, et~al.
\newblock Limit theorems in free probability theory. i.
\newblock {\em The Annals of Probability}, 36(1):54--90, 2008.

\bibitem{conze1978ergodic}
Jean-Pierre Conze and N~Dang-Ngoc.
\newblock Ergodic theorems for noncommutative dynamical systems.
\newblock {\em Inventiones mathematicae}, 46(1):1--15, 1978.

\bibitem{handel2017structured}
Ramon~van Handel.
\newblock Structured random matrices.
\newblock {\em Convexity and concentration}, pages 107--156, 2017.

\bibitem{hochstattler2016semicircle}
Winfried Hochst{\"a}ttler, Werner Kirsch, and Simone Warzel.
\newblock Semicircle law for a matrix ensemble with dependent entries.
\newblock {\em Journal of Theoretical Probability}, 29(3):1047--1068, 2016.

\bibitem{kallenberg2006probabilistic}
Olav Kallenberg.
\newblock {\em Probabilistic symmetries and invariance principles}.
\newblock Springer Science \& Business Media, 2006.

\bibitem{kargin2007berry}
Vladislav Kargin.
\newblock Berry--esseen for free random variables.
\newblock {\em Journal of Theoretical Probability}, 20(2):381--395, 2007.

\bibitem{kargin2007proof}
Vladislav Kargin et~al.
\newblock A proof of a non-commutative central limit theorem by the lindeberg
  method.
\newblock {\em Electronic Communications in Probability}, 12:36--50, 2007.

\bibitem{kostler2009a}
Claus {Kostler} and Roland {Speicher}.
\newblock A noncommutative de finetti theorem: Invariance under quantum
  permutations is equivalent to freeness with amalgamation.
\newblock {\em Communications in Mathematical Physics}, 291(2):473--490, 2009.

\bibitem{lance1976ergodic}
E~Christopher Lance.
\newblock Ergodic theorems for convex sets and operator algebras.
\newblock {\em Inventiones mathematicae}, 37(3):201--214, 1976.

\bibitem{ledoux2001concentration}
Michel Ledoux.
\newblock {\em The concentration of measure phenomenon}.
\newblock Number~89. American Mathematical Soc., 2001.

\bibitem{lindenstrauss1999pointwise}
Elon Lindenstrauss.
\newblock Pointwise theorems for amenable groups.
\newblock {\em Electronic Research Announcements of the American Mathematical
  Society}, 5(12):82--90, 1999.

\bibitem{mai2013operator}
Tobias Mai and Roland Speicher.
\newblock Operator-valued and multivariate free berry-esseen theorems.
\newblock In {\em Limit theorems in probability, statistics and number theory},
  pages 113--140. Springer, 2013.

\bibitem{nica2006lectures}
Alexandru Nica and Roland Speicher.
\newblock {\em Lectures on the combinatorics of free probability}, volume~13.
\newblock Cambridge University Press, 2006.

\bibitem{popa2007freeness}
Mihai Popa.
\newblock Freeness with amalgamation, limit theorems and s-transform in
  non-commutative probability spaces of type b.
\newblock {\em Colloquium Mathematicum}, 120, 10 2007.

\bibitem{speicher1990new}
Roland Speicher.
\newblock A new example of ‘independence’and ‘white noise’.
\newblock {\em Probability theory and related fields}, 84(2):141--159, 1990.

\bibitem{umegaki1954conditional}
Hisaharu Umegaki.
\newblock Conditional expectation in an operator algebra.
\newblock {\em Tohoku Mathematical Journal, Second Series}, 6(2-3):177--181,
  1954.

\bibitem{umegaki1956conditional}
Hisaharu Umegaki.
\newblock Conditional expectation in an operator algebra, ii.
\newblock {\em Tohoku Mathematical Journal, Second Series}, 8(1):86--100, 1956.

\bibitem{van2000asymptotic}
Aad~W Van~der Vaart.
\newblock {\em Asymptotic statistics}, volume~3.
\newblock Cambridge university press, 2000.

\bibitem{voiculescu1986addition}
Dan Voiculescu.
\newblock Addition of certain non-commuting random variables.
\newblock {\em Journal of functional analysis}, 66(3):323--346, 1986.

\bibitem{wigner1967random}
Eugene~P Wigner.
\newblock Random matrices in physics.
\newblock {\em SIAM review}, 9(1):1--23, 1967.

\end{thebibliography}
\appendix

\section{Proof of \cref{theorem_1}}

We prove the following proposition that directly implies  \cref{theorem_1}.
\begin{prop}\label{bartlett}Let $(X^n_g)$ be a triangular array of self-adjoint free random variables invariant under $\G_n$. Suppose that $(X^n_g)$ satisfies all the condition of \cref{theorem_1}. Let $S_n(\cdot)$ denote the (operator-valued) Stiejles transform of $\hat{\mu}_n$, and $S_n^{sc}(\cdot)$ the Stiejles transform of the operator-valued semi-circular operator $Y^{sc,\eta_n}$ with radius $\eta_n(\cdot)$. 
The following upper bound holds
\begin{equation*}\begin{split}&\Big\|S_n(\gamma_{x,\nu})-S_n^{sc}(\gamma_{x,\nu})\Big\|_{1}\\&\le\frac{\Big\|{X_{e}^n}\Big\|^2_{2} }{\nu^3}\Big[3\mathcal{R}^s_{n}[b_n]+\frac{|A_n\triangle B^n_{b_n}A_n|}{|A_{n,n}|}+2|B^n_{b_n}|\aleph^{j,\nu}_n[b_n|\mathbb{G}_n]\Big]
\\&~+\frac{9}{\sqrt{|A_{n,n}|}\nu^4}\Big[|B^n_{2b_n}|^2\Big\| {X_{e}^n}\Big\|_{3}^3+|B^n_0|^2\Big\| {Y^{{sc},\eta_n}}\Big\|_{3}^3\Big]
+\|Y^{{sc},\eta_n}\|_2^2\frac{|A_n\triangle B^n_{b_n}A_n|}{\nu^3|A_{n,n}|}.
\end{split}\end{equation*}
where $\gamma_{x,\nu}=x+i\nu$.
\end{prop}
\begin{proof}
Our goal is to upper- bound $\Big\|S_n(\gamma_{x,\nu})-S_n^{sc}(\gamma_{x,\nu})\Big\|_{1}$. To do this we interpolate between those two quantities by defining a function  
$g:[0,1]\rightarrow \F^{\rm{tail}}(X)$ that satisfies: $$\|g(1)-g(0)\|_{1}=\Big\|S_n(\gamma_{x,\nu})-S_n^{sc}(\gamma_{x,\nu})\Big\|_{1};$$ and subsequently prove that $g$ is differentiable and had bounded derivatives.


\noindent We first note  that we can suppose  without loss of generality that $E_n(X^n_g)=0$ for all $g\in \G_n$ and do so. For all operators $Y$ we write  the resolvent as $R(Y,\gamma)=[Y-\gamma \mathbf{1}_{\mathcal{A}}]^{-1}$; and define the following averages $$ W_{g,b}^n:=\frac{1}{\sqrt{|A_{n,n}|}}\int_{A_{n,n}\setminus B^n(g,b)} X_{g'}^nd|g'|\qquad b>0,~g\in \A^n.$$  Let $(Y_g^n)$ be free copies of $Y^{sc,\eta_n}$; and write $Y^{n}:=\frac{1}{\sqrt{|A_{n,n}|}}\int_{A_{n,n}}Y_g^n|d|g|.$ We remark that $Y^{{sc},\eta_n}$ has the same distribution than $Y^n$.  For all $t\in [0,1]$ we define the following interpolating processes \begin{equation*}\begin{split}&W_n(t):=\sqrt{t}W_n+\sqrt{1-t} Y^{n},\qquad W_{g,b}^n(t):=\sqrt{t}W^n_{g,b}+\sqrt{1-t} Y^{n}
\\&\qquad \qquad \qquad W_g^{sc,\eta_n}(t):=\sqrt{t}W_n+ \frac{\sqrt{1-t}}{\sqrt{|A_{n,n}|}}\int_{A_{n,n}\setminus B^n(g,0)}Y_g^nd|g|\end{split}\end{equation*} We note that $W_n(0)=Y^n$ and that $W_n(1)=W_n$. We also remark that if group $\mathbb{G}_n$ is discrete then $A_{n,n}\setminus B^n(g,0)=\{g\}$ which implies that $W_g^{sc,\eta_n}(t)\ne W_n(t)$.
\\\noindent Finally for simplicity  we use the following shorthand notations: \begin{align}&R_n^t:=R(W_n(t),\gamma_{x,\nu}),\qquad\quad R_b^{g,t}:=R(W^n_{g,b}(t),\gamma_{x,\nu})\\R_{g,t}^{sc,\eta_n}&:=R(W_g^{sc,\eta_n}(t)),\qquad S^b_{g,t}:=E_n\big(R_b^{g,t}\big),\qquad S^0_{n,t}:=E_n\big(R_n^{t}\big).\end{align}
We define the function $g:[0,1]\rightarrow \F^{\rm{tail}}(X)$  as $g(t)=E_n(R_n^t)$. As $W_n(t)$ interpolates $Y^n$ and $W_n$ we have: $$g(0)=S_n^{sc}(\gamma_{x,\nu}),\qquad g(1)=S_n(\gamma_{x,\nu}).$$
\vspace{2mm}

\noindent We prove that $g$ is Gateaux-differentiable. For all $\epsilon>0$ and all  $t\in [0,1]$ such that $t+\epsilon\in [0,1]$ we have:
\begin{equation*}\begin{split}g(t+\epsilon)-g(t)&=E_n\Big(R_n^{t+\epsilon}-R_n^t\Big)
\\&\overset{(a)}{=}E_n\Big(R_n^{t+\epsilon}\Big[W_n(t)-W_n(t+\epsilon)\Big]R_n^t\Big)
\\&{=}E_n\Big(R_n^{t+\epsilon}\Big[\sqrt{t}-\sqrt{t+\epsilon}\Big]W_n R_n^t\Big)
\\&~+E_n\Big(R_n^{t+\epsilon}\Big[\sqrt{1-t}-\sqrt{1-(t+\epsilon)}\Big]Y^{{sc},\eta_n} R_n^t\Big), 
\end{split}\end{equation*} where to get (a) we exploited the fact that for all $Y,Y^*\in \mathcal{A}_{\tau}$ the following identity holds:\begin{equation}\label{ca_ira}R(Y,\gamma_{x,\nu})-R(Y^*,\gamma_{x,\nu})=R(Y,\gamma_{x,\nu})[Y^*-Y]R(Y^*,\gamma_{x,\nu}).\end{equation} This implies that $g$ is Gateaux-differentiable and satisfies  $$g'(t)=-E_n\Big(R_n^t\Big[\frac{W_n}{2\sqrt{t}}-\frac{Y^{{sc},\eta_n}}{2\sqrt{1-t}}\Big] R_n^t\Big).$$ 
Therefore to upper-bound $\|S_n(\gamma_{x,\nu})-S_n^{sc}(\gamma_{x,\nu})\|_{1}$ it is sufficient to bound $\Big\|\int_0^1 g'(t)dt\Big\|_{1}$.

\noindent To do so, we first remark that  for all $t\in (0,1)$ we have
\begin{equation*}\begin{split}E_n\Big[R_n^t\frac{W_n}{\sqrt{t}}R_n^t\Big]
&\overset{(a)}{=}\frac{1}{\sqrt{|A_{n,n}|}}\int_{A_{n,n}}E_n\Big[R_n^t\frac{X_{g}^n}{\sqrt{t}}R_n^t\Big]d|g|
\\&\overset{}{=}\frac{1}{\sqrt{|A_{n,n}|}}\int_{A_{n,n}}E_n\Big[R_n^t\frac{X_{g}^n}{\sqrt{t}}\big[R_n^t-R^{g,t}_{b_n}\big]\Big]d|g|
\\&\quad +\frac{1}{\sqrt{|A_{n,n}|}}\int_{A_{n,n}}E_n\Big[\big[R_n^t-R^{g,t}_{b_n}\big]\frac{X_{g}^n}{\sqrt{t}}R^{g,t}_{b_n}\Big]d|g|\\&\quad + \frac{1}{\sqrt{|A_{n,n}|}}\int_{A_{n,n}}E_n\Big[R^{g,t}_{b_n}\frac{X_{g}^n}{\sqrt{t}}R^{g,t}_{b_n}\Big]d|g|
\\&\overset{(b)}{=}\frac{-1}{\sqrt{|A_{n,n}|}}\int_{A_{n,n}}E_n\Big[R_n^t\big[W_n-W_{g,b_n}^n\big]R^{g,t}_{b_n}X_g^nR_n^t\Big]d|g|
\\&\quad - 
\frac{1}{\sqrt{|A_{n,n}|}}\int_{A_{n,n}}E_n\Big[R^{g,t}_{b_n}X_{g}^nR_n^t\big[W_n-W_{g,b_n}^n\big]R^{g,t}_{b_n}\Big]d|g|
\\&\quad+ \frac{1}{\sqrt{|A_{n,n}|}}\int_{A_{n,n}}E_n\Big[R^{g,t}_{b_n}\frac{X_{g}^n}{\sqrt{t}}R^{g,t}_{b_n}\Big]d|g|\\&=a_1^{\gamma_{x,\nu}}+a_2^{\gamma_{x,\nu}}+a_3^{\gamma_{x,\nu}}
\end{split}\end{equation*}
where (a) comes from the linearity of the functional $E_n$ and where to get (b) we used \cref{ca_ira}.
~The rest of proof  consist in: (i) proving that $\big\|a_3^{\gamma_{x,\nu}}\big\|_1\rightarrow 0$ and (ii) in re-expressing $a_1^{\gamma_{x,\nu}}$ and $a_2^{\gamma_{x,\nu}}$. 

\noindent We start by proving that $\big\|a_3^{\gamma_{x,\nu}}\big\|_1\rightarrow 0$. Using the definition of the free mixing coefficients $(\aleph^s_n[\cdot|\mathbb{G}_n])$ we have:
\begin{equation*}\begin{split}&\Big\|E_n\Big[R^{g,t}_{b_n}\frac{X_{g}^n}{\sqrt{t}}R^{g,t}_{b_n}\Big]\Big\|_{1}
\\&\le\sum_{b\ge b_n} \Big\|E_n\Big[R^{g,t}_{b}\frac{X_{g}^n}{\sqrt{t}}\big[R^{g,t}_{b+1}-R^{g,t}_{b}\big]\Big]\Big\|_{1}
\\&\quad+\sum_{b\ge b_n} \Big\|E_n\Big[\big[R^{g,t}_{b+1}-R^{g,t}_{b}\big]\frac{X_{g}^n}{\sqrt{t}}R^{g,t}_{b+1}\Big]\Big\|_{1}
\\&\overset{(a)}{\le} \sum_{b\ge b_n} \Big\|E_n\Big[R^{g,t}_{b}{X_{g}^n}R^{g,t}_{b+1}\big[W_{g,b+1}^n-W_{g,b}^n\big]R^{g,t}_{b}\big]\Big]\Big
\|_{1}
\\&\quad +\sum_{b\ge b_n} \Big\|E_n\Big[R^{g,t}_{b+1} \big[W_{g,b+1}^n-W_{g,b}^n\big]R^{g,t}_{b}{X_{g}^n}R^{g,t}_{b+1}\Big]\Big
\|_{1}\\&\overset{}{\le} \frac{1}{\sqrt{|A_{n,n}|}}\sum_{b\ge b_n}\int_{B^{n}(b+1,g)\setminus B^n(b,g)} \Big\|E_n\Big[R^{g,t}_{b}{X_{g}^n}R^{g,t}_{b+1}X_{g'}^nR^{g,t}_{b}\big]\Big]\Big
\|_{1}d|g'|
\\&\quad +\frac{1}{\sqrt{|A_{n,n}|}}\sum_{b\ge b_n}\int_{B^{n}(b+1,g)\setminus B^n(b,g)} \Big\|E_n\Big[R^{g,t}_{b+1} X_{g'}^nR^{g,t}_{b}{X_{g}^n}R^{g,t}_{b+1}\Big]\Big
\|_{1} d|g'|
\end{split}\end{equation*}where (a) comes from the triangular inequality. 

Therefore we obtain that\begin{equation*}\begin{split}&\Big\|E_n\Big[R^{g,t}_{b_n}\frac{X_{g}^n}{\sqrt{t}}R^{g,t}_{b_n}\Big]\Big\|_{1}
\\&\overset{(a)}{\le}\frac{2}{\sqrt{|A_{n,n}|}} \Big\|{X_{e}^n}\Big\|^2_{2}\sum_{b\ge b_n}\big\|R^n_{e,b}(t)\big\|_{\infty}^2\big\|R^n_{e,b+1}(t)\big\|_{\infty} |B^n_{b+1}\setminus B^n_b|\aleph^s_n[b|\mathbb{G}_n].
\end{split}\end{equation*} and where to get (a) we used the definition of $\aleph_n^s[\cdot|\G_n]$ in \cref{suis_suisse} coupled with the fact that $E_n({X_{e}^n})=0$ and \cref{anti_hero}. Therefore by taking the average over $g\in A_{n,n}$ we obtain that 
\begin{equation}\begin{split}\label{eq*_1}\big|a_3^{\gamma_{x,\nu}}|&\le \frac{2\Big\|{X_{e}^n}\Big\|^2_{2} \mathcal{R}^s_{n}[b_n]}{\nu^3}\rightarrow 0.
\end{split}\end{equation}

\noindent The next step consists in re-expressing $a_1^{\gamma_{x,\nu}}$ and $a_2^{\gamma_{x,\nu}}$ in terms of $R^{g,t}_{2b_n}$ and $X_g^n$.  
Using \cref{ca_ira} for all $g\in A_{n,n}$ we have
\begin{equation}\begin{split}&\label{eq_1}
\Big\|E_n\Big[\big[R_n^t-R^{g,t}_{2b_n}\big]\big[W_n-W_{g,b_n}^n\big]R^{g,t}_{b_n}X_g^nR_n^t\Big]\Big\|_{1}
\\&\le \Big\|E_n\Big[ R_n^t\big[W_n-W_{g,2b_n}\big]R^{g,t}_{2b_n}\big[W_n-W_{g,b_n}^n\big]R^{g,t}_{b_n} X_g^n R_n^t\Big]\Big\|_{1}
\\&\overset{(a)}{\le}\frac{ {|B^n_{2b_n}|^2}}{|A_{n,n}|\nu^4}\Big\| {X_{e}^n}\Big\|_{3}^3 
\end{split}\end{equation}where (a) comes from the H\"{o}lder-inequality combined with the fact that $\big|B(g,2b_n)^n\big|\le |B^n_{2b_n}|$.
Similarly,  the following also holds:
\begin{equation}\begin{split}&\label{eq_2}
\Big\|E_n\Big[R^{g,t}_{2b_n}\big[W_n-W_{g,b_n}^n\big]R^{g,t}_{b_n}X_g^n \big[R_n^t-R^{g,t}_{2b_n}\big]\Big]\Big\|_{1}
{\le}\frac{ {|B^n_{2b_n}|^2}}{|A_{n,n}|\nu^4}\Big\| {X_{e}^n}\Big\|_{3}^3; 
\end{split}\end{equation}
 as well as
\begin{equation}\begin{split}&\label{eq_3}
\Big\|E_n\Big[R^{g,t}_{2b_n}\big[W_n-W_{g,b_n}^n\big]\big[R^{g,t}_{b_n}-R^{g,t}_{2b_n}\big] X_g^nR^{g,t}_{2b_n}\Big]\Big\|_{1}
{\le}\frac{ {|B^n_{2b_n}|^2}}{|A_{n,n}|\nu^4}\Big\| {X_{e}^n}\Big\|_{3}^3. 
\end{split}\end{equation}

\noindent Therefore by averaging over $g\in A_{n,n}$ the results of \cref{eq_1}, \cref{eq_2} and \cref{eq_3} we obtain that: \begin{equation}\begin{split}&\label{eq*_2}\Big|a_1^{\gamma_{x,\nu}}+\frac{1}{{|A_{n,n}|}}\int_{A_{n,n}}\int_{ B(g,b_n)}\hspace{-5mm} E_n\Big[R^{g,t}_{2b_n}X_{g'}^nR^{g,t}_{2b_n}X_g^nR^{g,t}_{2b_n}\Big]d|g'|d|g| \Big| \\&\le \frac{3|B^n_{2b_n}|^2}{\sqrt{|A_{n,n}|}\nu^4}\Big\| {X_{e}^n}\Big\|_{3}^3;
\end{split}\end{equation} and similarly we obtain that:
\begin{equation}\begin{split}&\label{eq*_3}\Big|a_2^{\gamma_{x,\nu}}+\frac{1}{{|A_{n,n}|}}\int_{A_{n,n}}\int_{ B(g,b_n)}\hspace{-5mm}E_n\Big[R^{g,t}_{2b_n}X_g^nR^{g,t}_{2b_n}X^n_{g'}R^{g,t}_{2b_n}\Big]d|g'|d|g|\Big| \\&\le \frac{3|B^n_{2b_n}|^2}{\sqrt{|A_{n,n}|}\nu^4}\Big\| {X_{e}^n}\Big\|_{3}^3.
\end{split}\end{equation}

Therefore combined together \cref{eq*_1}, \cref{eq*_2} and \cref{eq*_3} imply that 
\begin{equation*}\begin{split}&\Big\|g(1)-g(0)\Big\|_{1}\\&\le\int_0^1\Big\|\frac{1}{{|A_{n,n}|}}\int_{A_{n,n}}\int_{ B(g,b_n)}\hspace{-5mm}E_n\Big[R^{g,t}_{2b_n}X_g^nR^{g,t}_{2b_n}X^n_{g'}R^{g,t}_{2b_n}\Big]d|g'|d|g|
\\&\qquad \qquad-E_n\Big[R_n^t\frac{Y^{{sc},\eta_n}}{2\sqrt{1-t}}R_n^t\Big]~dt\Big\|_{1}
+\frac{6|B^n_{2b_n}|^2}{\sqrt{|A_{n,n}|}\nu^4}\Big\| {X_{e}^n}\Big\|_{3}^3+\frac{2\Big\|{X_{e}^n}\Big\|^2_{2} \mathcal{R}^s_{n}[b_n]}{\nu^3}.
\end{split}\end{equation*}
Moreover by exploiting the free independence of $(Y_g^n)$ we obtain for all $t\in (0,1)$ that:
\begin{equation*}\begin{split}&\Big\|\frac{1}{{|A_{n,n}|}}\int_{A_{n,n}}E_n\Big[R^{sc,\eta_n}_{g,t}Y_g^nR^{sc,\eta_n}_{g,t}Y_g^nR^{sc,\eta_n}_{g,t}\Big]d|g|
-E_n\Big[R_n^t\frac{Y^{{sc},\eta_n}}{2\sqrt{1-t}}R_n^t\Big]~dt\Big\|_{1}
\\&\le\frac{6|B^n_0|^2}{\sqrt{|A_{n,n}|}\nu^4}\Big\| {Y^{{sc},\eta_n}}\Big\|_{3}^3.
\end{split}\end{equation*}
We remark that if the group $\mathbb{G}_n$ is discrete then $|B^n_0|=1$.

\noindent Moreover, we observe that to finish upper bounding $\|g(1)-g(0)\|_{1}$ we need to compare 
$$\frac{1}{{|A_{n,n}|}}\int_{A_{n,n}}\int_{ B(g,b_n)}E_n\Big[R^{g,t}_{2b_n}X_g^nR^{g,t}_{2b_n}X^n_{g'}R^{g,t}_{2b_n}\Big]d|g'|d|g|$$ $$\rm{with}\qquad\frac{1}{{|A_{n,n}|}}\int_{A_{n,n}}E_n\Big[R^{sc,\eta_n}_{g,t}Y_g^nR^{sc,\eta_n}_{g,t}Y_g^nR^{sc,\eta_n}_{g,t}\Big]d|g|.$$

\noindent To do so we introduce the following notation for all $g,g'\in \mathbb{G}_n$ $$\eta_{g,g'}:a\rightarrow E_n\Big[X_g^n aX^n_{g'}\Big].$$ Using the triangular inequality we obtain that
\begin{equation*}\begin{split}&\Big\|\frac{1}{|A_{n,n}|}\int_{A_{n,n}}\int_{B^n(g,b_n)}\hspace{-5mm}E_n\Big[R^{g,t}_{2b_n}X_g^nR^{g,t}_{2b_n}X^n_{g'}R^{g,t}_{2b_n}\Big]d|g'|d|g|
 -E_n\big(R_n^t\eta_n(S^0_{n,t})R_n^t\big)\Big\|_{1}
\\&\le \Big\|\frac{1}{{|A_{n,n}|}}\int_{A_{n,n}}\int_{ B^n(g,b_n)}E_n\Big[R^{g,t}_{2b_n}X_g^nR^{g,t}_{2b_n}X^n_{g'}R^{g,t}_{2b_n}\Big]-E_n\Big(R_{g,t}^{2b_n}\eta_{g,g'}(S_{g,t}^{2b_n})R_{g,t}^{2b_n}\Big)d|g'|d|g|  \Big\|_{1}
\\&\quad+ \Big\|\frac{1}{{|A_{n,n}|}}\int_{A_{n,n}}\int_{ B^n(g,b_n)} E_n\Big(R_{g,t}^{2b_n}\eta_{g,g'}(S_{g,t}^{2b_n})R_{g,t}^{2b_n}\Big)
- E_n\big(R_{n}^{t}\eta_{g,g'}(S_{n,t}^0)R_{n}^{t}\big)d|g'|d|g| \Big\|_{1}
\\&\quad+ \Big\|\frac{1}{{|A_{n,n}|}}\int_{A_{n,n}}\int_{ B^n(g,b_n)} E_n\big(R_{n,t}^{0}\eta_{g,g'}(S_{n,t}^0)R_{n,t}^{0}\big)d|g'|d|g| 
 -E_n\big(R_{n}^{t}\eta_n(S^0_{n,t})R_{n}^{t}\big) \Big\|_{1}
\\&\le (c_1^{\gamma_{x,\nu}})+(c_2^{\gamma_{x,\nu}})+(c_3^{\gamma_{x,\nu}})
\end{split}\end{equation*}
We bound each term successively. The first term is bounded using the definition of $(\aleph^{j,\nu}_n[\cdot|\G_n])$, and the later terms are bounded using  the triangular inequality and \cref{ca_ira}. We focus first on the term $c_1^{\gamma_{x,\nu}}$. By definition of $\aleph^{j}_n[\cdot|\G_n]$, we obtain that
\begin{equation}\begin{split}&\label{e}
\big|(c_1^{\gamma_{x,\nu}})\big|
\\&\le \Big\|\frac{1}{{|A_{n,n}|}}\int_{A_{n,n}}\int_{ B^n(g,b_n)}E_n\Big[R^{g,t}_{2b_n}X_g^nR^{g,t}_{2b_n}X^n_{g'}R^{g,t}_{2b_n}\Big]-E_n\Big(R_{g,t}^{2b_n}\eta_{g,g'}(S_{g,t}^{2b_n})R_{g,t}^{2b_n}\Big)d|g'|d|g|  \Big\|_{1}
\\&\le \Big\|\frac{1}{{|A_{n,n}|}}\int_{A_{n,n}}\int_{ B^n(g,b_n)}E_n\Big[R^{g,t}_{2b_n}X_g^n\overline{R^{g,t}_{2b_n}}X^n_{g'}R^{g,t}_{2b_n}\Big]d|g'|d|g|  \Big\|_{1}
\\&\quad+\Big\|\frac{1}{{|A_{n,n}|}}\int_{A_{n,n}}\int_{ B^n(g,b_n)}E_n\Big[R^{g,t}_{2b_n}X_g^nS^{g,t}_{2b_n}X^n_{g'}R^{g,t}_{2b_n}\Big]-E_n\Big(R_{g,t}^{2b_n}\eta_{g,g'}(S_{g,t}^{2b_n})R_{g,t}^{2b_n}\Big)d|g'|d|g|  \Big\|_{1}\\&\le \Big\|\frac{1}{{|A_{n,n}|}}\int_{A_{n,n}}\int_{ B^n(g,b_n)}E_n\Big[R^{g,t}_{2b_n}X_g^n\overline{R^{g,t}_{2b_n}}X^n_{g'}R^{g,t}_{2b_n}\Big]d|g'|d|g|  \Big\|_{1}
\\&\quad+\Big\|\frac{1}{{|A_{n,n}|}}\int_{A_{n,n}}\int_{ B^n(g,b_n)}E_n\Big[R^{g,t}_{2b_n}\overline{X_g^nS^{g,t}_{2b_n}X^n_{g'}}R^{g,t}_{2b_n}\Big]d|g'|d|g|  \Big\|_{1}
\\&\le 2\Big\|X_{e}^n\Big\|^2_{2}\big\|R^{g,t}_{2b_n}\big\|_{\infty}^3\frac{1}{|A_{n,n}|}\int_{A_{n,n}}\int_{B^n(g,b_n)}\aleph^{j,\nu}_n[\bar{d}(\{g,g'\}, A_{n,n}\setminus B^n(g,2b_n))|\mathbb{G}_n]d|g'|d|g|
\\&\le\frac{2\Big\|X_{e}^n\Big\|^2_{2}|B^n_{b_n}|\aleph^{j,\nu}_n[b_n|\mathbb{G}_n]}{\nu^3}.
\end{split}\end{equation} 
To bound $(c_2^{\gamma_{x,\nu}})$ we use the triangular inequality and obtain that 
\begin{equation*}\begin{split}&
\Big\|E_n\Big(R_{g,t}^{2b_n}\eta_{g,g'}(S_{g,t}^{2b_n})
R_{g,t}^{2b_n}\Big) -E_n\Big(R_{n}^{t}\eta_{g,g'}(S_{n,t}^0)
R_{n}^{t}\Big)\Big\|_{1}
\\&\overset{}{\le} 
\Big\|E_n\Big(\Big[R_{g,t}^{2b_n}-R_n^t\Big]\eta_{g,g'}(S_{g,t}^{2b_n})
R_{g,t}^{2b_n}\Big)\Big\|_{1}
+\Big\|E_n\Big[R_{n}^{t}\Big[\eta_{g,g'}(S_{g,t}^{2b_n})-\eta_{g,g'}(S_{n,t}^0)\Big]
R_{g,t}^{2b_n}\Big]\Big\|_{1}
\\&\quad +\Big\|E_n\Big(R_{n}^{t}\eta_{g,g'}(S_{n,t}^0)
\Big[R_{g,t}^{2b_n}-R_{n}^{t}\Big]\Big) \Big\|_{1}
\\&\overset{(a)}{\le} \Big\|R_n^{g,2b_n}(t)\big[W_n-W_{g,2b_n}^n(t)\big]R_n^t\eta_{g,g'}(S_{g,t}^{2b_n})
R_{g,t}^{2b_n} \Big\|_{1}
\\&\quad+\Big\|R_{n}^{t}E_n\Big(X_gE_n\Big[R_n^{g,2b_n}(t)\big[W_n-W_{g,2b_n}^n(t)\big]R_n^t\Big]X_{g'}\Big)R_{g,t}^{2b_n}\Big\|_{1}
\\&\quad +\Big\|R_n^t\eta_{g,g'}(S_{n,t}^0)
R_n^{g,2b_n}(t)\big[W_n-W_{g,2b_n}^n(t)\big]R_n^t\Big\|_{1}
\\&\overset{(b)}{\le}\frac{3|B^n_{2b_n}|}{\sqrt{|A_{n,n}|}\nu^4}\Big\|X_{e}^n\Big\|_{3}^3.
\end{split}\end{equation*} where to get (a) we used \cref{ca_ira}; and  (b) was obtained by using the Cauchy-Schwarz inequality. Therefore by averaging over $g\in A_{n,n}$ and $g'\in B^n(g,b_n)$ we obtain that
\begin{equation}\begin{split}\label{e_2}
(c_2^{\gamma_{x,\nu}})
&\le\frac{3|B^n_{2b_n}|^2}{\sqrt{|A_{n,n}|}\nu^4}\Big\|X_{e}^n\Big\|_{3}^3.\end{split}\end{equation} 

\noindent Finally by the triangular inequality we remark that 
\begin{equation}\begin{split}\label{e_3}&|(c_3^{\gamma_{x,\nu}})|\\&\le
 \Big\|\frac{1}{{|A_{n,n}|}}\int_{A_{n,n}}\int_{ B^n(g,b_n)} E_n\Big(R_{n}^{t}\eta_{g,g'}(S_{n,t}^0)R_n^t\Big)d|g'|d|g| 
 -E_n\Big(R_n^t\eta_n(S^0_{n,t})R_n^t\Big)\Big\|_{1}
\\&\le \Big\|\frac{1}{|A_{n,n}|}\int_{A_{n,n}}\int_{ B^n(g,b_n)}E_n\Big(R_{n}^{t}\eta_{g,g'}(S_{n,t}^0)R_n^t\Big)d|g'|d|g|-\int_{ B^n(e,b_n)}\hspace{-2mm}E_n\Big(R_{n}^{t}\eta_{g,e}(S_{n,t}^0)R_n^t\Big)d|g|\Big\|_{1}
\\&\quad+\Big\|\int_{ B^n(e,b_n)}\hspace{-2mm}E_n\big(R_{n}^{t}\eta_{e,g}(S_{n,t}^0)R_n^t\big)d|g|-\int_{A_{n,n}}E_n\big(R_{n}^{t}\eta_{e,g}(S_{n,t}^0)R_n^t\big)d|g|\Big\|_{1}
\\&\overset{(a)}{\le}\big\|X_{e}^n\big\|_{2}^2 \frac{|A_n\triangle B^n_{b_n}A_n|}{\nu^3|A_{n,n}|}+\frac{\big\|X_{e}^n\big\|_{2}^2}{\nu^3}  \mathcal{R}^s_n[b_n].
\end{split}\end{equation}  where to obtain (a) we exploited the Cauchy-Schwarz inequality. 
Therefore by combining \cref{e}, \cref{e_2} and \cref{e_3} we obtain that
\begin{equation}\begin{split}&\label{e_5}\Big\|g(1)-g(0)\Big\|_{1}\\&\le\int_0^1\Big\|S_{n,t}^{0}\eta_n(S^0_{n,t})S_{n,t}^{0}-\frac{1}{{|A_{n,n}|}}\int_{A_{n,n}}E_n\Big[R^{sc,\eta_n}_{g,t}Y_g^nR^{sc,\eta_n}_{g,t}Y_g^nR^{sc,\eta_n}_{g,t}\Big]d|g|\Big\|_{1}~dt
\\&+ \frac{\Big\|{X_{e}^n}\Big\|^2_{2} }{\nu^3}\Big[3\mathcal{R}^s_{n}[b_n]+\frac{|A_n\triangle B^n_{b_n}A_n|}{|A_{n,n}|}+2|B^n_{b_n}|\aleph^{j,\nu}_n[b_n|\mathbb{G}_n]\Big]
\\&~+\frac{1}{\sqrt{|A_{n,n}|}\nu^4}\Big[9|B^n_{2b_n}|^2\Big\| {X_{e}^n}\Big\|_{3}^3+6|B^n_0|^2\Big\| {Y^{{sc},\eta_n}}\Big\|_{3}^3\Big].
\end{split}\end{equation}
Moreover following similar arguments we prove that 
\begin{equation}\begin{split}\label{e_4}&\int_0^1\Big\|E_n\big(R_{n}^{t}\eta_n(S^0_{n,t})R_{n}^{t}\big)-\frac{1}{{|A_{n,n}|}}\int_{A_{n,n}}E_n\Big[R^{sc,\eta_n}_{g,t}Y_g^nR^{sc,\eta_n}_{g,t}Y_g^nR^{sc,\eta_n}_{g,t}\Big]d|g|\Big\|_{1}~dt
\\&\le\frac{3\|Y^{{sc},\eta_n}\|_3^3}{\sqrt{|A_{n,n}|}\nu^4}+\|Y^{{sc},\eta_n}\|_2^2\frac{|A_n\triangle B^n_{b_n}A_n|}{\nu^3|A_{n,n}|}
.
\end{split}\end{equation}

Therefore by combining \cref{e_5} and \cref{e_4} we finally obtain that:
\begin{equation*}\begin{split}&\Big\|g(1)-g(0)\Big\|_{1}\\&\le \frac{\Big\|{X_{e}^n}\Big\|^2_{2} }{\nu^3}\Big[3\mathcal{R}^s_{n}[b_n]+\frac{|A_n\triangle B^n_{b_n}A_n|}{|A_{n,n}|}+2|B^n_{b_n}|\aleph^{j,\nu}_n[b_n|\mathbb{G}_n]\Big]
\\&~+\frac{9}{\sqrt{|A_{n,n}|}\nu^4}\Big[|B^n_{2b_n}|^2\Big\| {X_{e}^n}\Big\|_{3}^3+|B^n_0|^2\Big\| {Y^{{sc},\eta_n}}\Big\|_{3}^3\Big]
+\|Y^{{sc},\eta_n}\|_2^2\frac{|A_n\triangle B^n_{b_n}A_n|}{\nu^3|A_{n,n}|}.
\end{split}\end{equation*}
~
\end{proof}
\section{Proof of \cref{theorem_3} and \cref{theorem_4}}
We first prove an intermediary result that we use to deduce \cref{theorem_3} and \cref{theorem_4}. In this goal given a completely positive function: $\eta:\A\rightarrow \F^{\rm{tail}}_{\mathcal{D}(\G_n^{k_n})}(X^n)$ we write $$\Big\|\eta(\cdot)-\eta_n(\cdot)\big|_{\mathcal{A}^{\rm{tail}}_{\mathcal{D}(\G_n^{k_n})}}\Big\|_{\rm{op}}:=\sup_{\substack{a\in\mathcal{A}^{\rm{tail}}_{\mathcal{D}(\G_n^{k_n})}\\\|a\|_{\infty}\le 1}}\Big\|\eta(a)-\eta_n(a)\Big\|_{1}.$$
\begin{prop}\label{bartlett_2}Let $(X^n_{\g})$ be a triangle array of free random variables satisfying all the conditions of \cref{theorem_3}. Let $(\eta_n')$ be a sequence of completely positive maps. Let $S_n(\cdot)$ denote the (operator-valued) Stiejles transform of \\$\frac{1}{|A_{n,n}|^{k_n-\frac{1}{2}}}\int_{\A_{n,n}^{k_n}}X_{\g}d|\g|$, and $S_n^{sc}(\cdot)$ the Stiejles transform of the operator-valued semi-circular operator $Y^{sc,\eta'_n}$ with radius $\eta'_n(\cdot)$. 
The following upper bound holds
\begin{equation*}\begin{split}&\Big\|S_n(\gamma_{x,\nu})-S_n^{sc}(\gamma_{x,\nu})\Big\|_{1}
\\&\le
\frac{{12|B_{0}|^2}\Big\|Y^{sc,\eta'_n}\Big\|_{3}^3+{18k_n^4|B^n_{2b_n}|^2}\sup_{\g\in \G^{k_n}}\Big\| {X_{\g}^n}\Big\|_{3}^3}{\sqrt{|A_{n,n}|}\nu^4}
+\frac{3k_n^2\sup_{\g\in \G^{k_n}}\Big\|{X_{\g}^n}\Big\|^2_{2} \mathcal{R}^s_{n}[b_n]}{\nu^3}
\\&~+\frac{1}{\nu^3} \Big\|\eta_n'(\cdot)-\eta_n(\cdot)\big|_{\mathcal{A}^{\rm{tail}}_{\mathcal{D}(\G_n^{k_n})}}\Big\|_{\rm{op}}+\frac{ 2k_n^2\sup_{\g\in \G^{k_n}}\Big\|X_{\g}^n\Big\|^2_{2}|B^n_{b_n}|\aleph_n^{*,j,\nu}[b_n|\mathbb{G}^{k_n}_n]}{\nu^3}.
\end{split}\end{equation*}
where $\gamma_{x,\nu}=x+i\nu$.
\end{prop}
\begin{proof} 
The proof of \cref{bartlett_2}, as the proof of \cref{bartlett}, adapts the Linderberg method; and to do so it creates an operator $W_n(t)$ that interpolates between $W_n$ and $Y^{{sc},\eta'_n}$.

\noindent In this goal we introduce some notations. For all operators $W$ and $\gamma\in \mathbb{C}\setminus\mathbb{R}$ we write $R(W,\gamma)=[W-\gamma \mathbf{1}_{\mathcal{A}}]^{-1}$. Moreover we define
$$W_n:=\frac{1}{|A_{n,n}|^{k_n-\frac{1}{2}}}\int_{A_{n,n}^{k_n}} X_{\g}^nd|\g|.$$  Let $Y^n$ be a free operator-valued semi-circular operator with radius $\eta'_n$. Let $(Y^n_{g})$ be free copies of $Y^n$. We write $Y^{{sc},\eta'_n}:=\frac{1}{\sqrt{|A_{n,n}|}}\int_{A_{n,n}}Y^n_{g}|d|g|$. We note that $Y^{{sc},\eta'_n}$ is defined as an empirical average of $(Y^n_{g})$ when $g$ varies over $A_{n,n}$ not $A_{n,n}^{k_n}$.  For all $t\in [0,1]$, we define the following interpolating processes:
\begin{equation*}\begin{split}&W_n(t):=\sqrt{t}W_n+\sqrt{1-t} Y^{{sc},\eta'_n}
\\&W_{g}^{sc,\eta_n}(t):=\sqrt{t}W_n+ \frac{\sqrt{1-t}}{\sqrt{|A_{n,n}|}}\int_{A_{n,n}\setminus B^n(g,0)}Y^n_gd|g|.\end{split}\end{equation*}
Moreover to be able to handle the fact $(X^n_{\g})$ is not free we introduce some slightly modified $W_n(t).$ For all $\g,\g'\in \G^{k_n}$ and all $b>0$ we write \begin{align*}&W_{\g,b}^n(t):=W_n(t)- \frac{1}{|A_{n,n}|^{k_n-\frac{1}{2}}}\int_{ \overline{B}_{k_n}(\g,b)^n} X_{\g'}^nd|\g'|\\&W_{\g,\g',2b_n}^n(t):=W_n(t)-\frac{\sqrt{t}}{|A_{n,n}|^{k_n-\frac{1}{2}}}\int_{\overline{B}_{k_n}(\g,2b_n)^n\bigcup \overline{B}_{k_n}(\g',2b_n)^n}X^n_{\g''}d|\g''|.\end{align*}

\noindent As in the proof of \cref{bartlett}, $W_n(t)$ interpolates between $W_n$ and $Y^{{sc},\eta'_n}$.  For simplicity  we use the following shorthand notations: $$R_n^t:=R(W_n(t),\gamma_{x,\nu}),\qquad\quad R^{\g,t}_b:=R(W^n_{\g,b}(t),\gamma_{x,\nu}) ,\qquad \quad R^{\g,\g',t}_{b}:=R(W_{\g,\g',b}^n(t),\gamma_{x,\nu}).$$%
Denote $g:t\rightarrow E_{\mathcal{D}_n}(R_n^t)$, we remark that $$g(1)-g(0)=S_n(\gamma_{x,\nu})-S^{sc}_n(\gamma_{x,\nu}).$$Therefore the objective is to upper-bound $\|g(1)-g(0)\|_{1}.$ In this goal, similarly than in the proof of \cref{bartlett}, we notice that $g$ is Gateaux-differentiable and that its derivative respects $$g'(t)=-E_{\mathcal{D}_n}\Big(R_n^t\Big[\frac{W_n}{2\sqrt{t}}-\frac{Y^{{sc},\eta'_n}}{2\sqrt{1-t}}\Big] R_n^t\Big).$$ 

\noindent The key of the proof consists in upper-bounding $|g'(t)|$. Firstly we note for all $t\in (0,1)$ that we have:
\begin{equation*}\begin{split}&E_{\mathcal{D}_n}\Big[R_n^t\frac{W_n}{\sqrt{t}}R_n^t\Big]
\\&\overset{(a)}{=}\frac{1}{|A_{n,n}|^{k_n-\frac{1}{2}}}\int_{A_{n,n}^{k_n}}E_{\mathcal{D}_n}\Big[R_n^t\frac{X_{\g}^n}{\sqrt{t}}R_n^t\Big]d|\g|
\\&\overset{(b)}{=}\frac{-1}{|A_{n,n}|^{k_n-\frac{1}{2}}}\int_{A_{n,n}^{k_n}}E_{\mathcal{D}_n}\Big[R_n^t\big[W_n-W_{\g,b_n}^n\big]R^{\g,t}_{b_n}X_{\g}^nR_n^t\Big]d|\g|
\\&\quad - 
\frac{1}{|A_{n,n}|^{k_n-\frac{1}{2}}}\int_{A_{n,n}^{k_n}}E_{\mathcal{D}_n}\Big[R^{\g,t}_{b_n}X_{\g}^nR_n^t\big[W_n-W_{\g,b_n}^n\big]R^{\g,t}_{b_n}\Big]d|\g|
\\&\quad+ \frac{1}{|A_{n,n}|^{k_n-\frac{1}{2}}}\int_{A_{n,n}^{k_n}}E_{\mathcal{D}_n}\Big[R^{\g,t}_{b_n}\frac{X_{\g}^n}{\sqrt{t}}R^{\g,t}_{b_n}\Big]d|\g|\\&=a_1^{\gamma_{x,\nu}}+a_2^{\gamma_{x,\nu}}+a_3^{\gamma_{x,\nu}}
\end{split}\end{equation*}
where (a) comes from the linearity of the functional $E_{\mathcal{D}_n}$ and (b) from \cref{ca_ira}.
~The rest of proof consists in proving that: (i) $\big|a_3^{\gamma_{x,\nu}}\big|\rightarrow 0$ and (ii) on re-expressing $a_1^{\gamma_{x,\nu}}$ and $a_2^{\gamma_{x,\nu}}$. 

\noindent We start by proving that $\big|a_3^{\gamma_{x,\nu}}\big|\rightarrow 0$. Using the definition of the free global mixing coefficients $(\aleph^*_n[\cdot|\mathbb{G}^{k_n}_n])$ we have:
\begin{equation*}\begin{split}&\Big\|E_{\mathcal{D}_n}\Big[R^{\g,t}_{b_n}\frac{{X_{\g}^n}}{\sqrt{t}}R^{\g,t}_{b_n}\Big]\Big\|_{1}
\\&\overset{(a)}{\le}\sum_{b\ge b_n} \Big\|E_{\mathcal{D}_n}\Big[R^{\g,t}_{b}\frac{X_{\g}^n}{\sqrt{t}}\big[R^{\g,t}_{b+1}-R^{\g,t}_{b}\big]\Big]\Big\|_{1}
\\&\quad+\sum_{b\ge b_n} \Big\|E_{\mathcal{D}_n}\Big[\big[R^{\g,t}_{b+1}-R^{\g,t}_{b}\big]\frac{X_{\g}^n}{\sqrt{t}}R^{\g,t}_{b+1}\Big]\Big\|_{1}
\\&\overset{(b)}{\le} \sum_{b\ge b_n} \Big\|E_{\mathcal{D}_n}\Big[R^{\g,t}_{b}{X_{\g}^n}R^{\g,t}_{b+1}\big[W_{\g,b+1}^n-W_{\g,b}^n\big]R^{\g,t}_{b}\big]\Big]\Big
\|_{1}
\\&\quad +\sum_{b\ge b_n} \Big\|E_{\mathcal{D}_n}\Big[R^{\g,t}_{b+1} \big[W_{\g,b+1}^n-W_{\g,b}^n\big]R^{\g,t}_{b}{X_{\g}^n}R^{\g,t}_{b+1}\Big]\Big
\|_{1}
\\&\overset{(c)}{\le}\frac{2k_n^2}{\sqrt{|A_{n,n}|}\nu^3}\sup_{\g\in \G^{k_n}} \Big\|{X_{\g}^n}\Big\|^2_{2}\sum_{b\ge b_n}|B_{b+1}^n\setminus B^n_b|\aleph^{*,s}_n[b|\G^{k_n}_n].
\end{split}\end{equation*} where (a) comes from the triangular inequality and (b) from \cref{ca_ira}. While  (c) was obtained using successively the definition of $(\aleph^{*s}_n[\cdot|\G^{k_n}_n])$, the fact that $E_{\mathcal{D}_n}({X_{\g}^n})=0$ as well as the fact that $|\overline{B}^n_{k_n}(\g,b+1)\setminus \overline{B}^n_{k_n}(\g,b)|\le k_n^2|B_{b+1}^n\setminus B^n_b||A_{n,n}|^{k_n-1}$. 
Therefore by taking the average over $\g\in A_{n,n}^{k_n}$ we obtain that 
\begin{equation}\begin{split}\label{u*_1}\big|a_3^{\gamma_{x,\nu}}|&\le \frac{2k_n^2}{\nu^3}\sup_{\g \in \G^{k_n}}\Big\|{X_{\g}^n}\Big\|^2_{2} \mathcal{R}^s_{n}[b_n]
\end{split}\end{equation}

\noindent The next step consists in re-expressing $a_1^{\gamma_{x,\nu}}$ as an alternative quantity that is easier to bound. Using \cref{ca_ira}  we remark for all $\g,\g'\in A_{n,n}^{k_n}$ that
\begin{equation}\begin{split}&\label{u_1}
\Big\|E_{\mathcal{D}_n}\Big[\big[R_n^t-R^{\g,\g',t}_{2b_n}\big]~X^n_{\g'}~R^{\g,t}_{b_n}~X_{\g}^n~R_n^t\Big]\Big\|_{1}
\\&\le \Big\|E_{\mathcal{D}_n}\Big[ R_n^t~\big[W_n(t)-W_{\g,\g',2b_n}(t)\big]~R_{\g,2b_n}^n(t)~X_{\g'}^n~R^{\g,t}_{b_n}~ X_{\g}^n~ R_n^t\Big]\Big\|_{1}
\\&\overset{(a)}{\le}\frac{ 2k_n^2{|B^n_{2b_n}|}}{\nu^4\sqrt{|A_{n,n}|}}\sup_{\g\in \G^{k_n}}\Big\| {X_{\g}^n}\Big\|_{3}^3 
\end{split}\end{equation}where (a) comes from the H\"{o}lder-inequality coupled with the fact that \\$\Big|\overline{B}^n_{k_n}(\g',2b_n)\bigcup \overline{B}^n_{k_n}(\g,2b_n)\Big|\le2 k_n^2|B^n_{2b_n}||A_{n,n}|^{k_n-1}.$ Similarly we can prove that:
\begin{equation}\begin{split}&\label{u_2}
\Big\|E_{\mathcal{D}_n}\Big[R^{\g,\g',t}_{2b_n}X^n_{\g'}R^{\g,t}_{b_n}X_{\g}^n\times \big[R^{\g,\g',t}_{2b_n}-R_n^t\big]\Big]\Big\|_{1}
\le \frac{2 k_n^2{|B^n_{2b_n}|}}{\nu^4\sqrt{|A_{n,n}|}}\sup_{\g\in \G^{k_n}}\Big\| {X_{\g}^n}\Big\|_{3}^3; 
\end{split}\end{equation}
as well as:
\begin{equation}\begin{split}&\label{u_3}
\Big\|E_{\mathcal{D}_n}\Big[R^{\g,\g',t}_{2b_n}X^n_{\g'}\Big[R^{\g,t}_{b_n}-R^{\g,\g',t}_{2b_n}\Big] X_{\g}^n R^{\g,\g',t}_{2b_n}\Big]\Big\|_{1}
\le \frac{ 2k_n^2{|B^n_{2b_n}|}}{\nu^4\sqrt{|A_{n,n}|}}\sup_{\g\in \G^{k_n}}\Big\| {X_{\g}^n}\Big\|_{3}^3. 
\end{split}\end{equation}

Therefore by combining \cref{u_1}, \cref{u_2} and \cref{u_3} and averaging over $\g\in A_{n,n}^{k_n}$ and $\g'\in \overline{B}_{k_n}(\g,b_n)$ we obtain that \begin{equation}\begin{split}&\label{u*_2}\Big|a_1^{\gamma_{x,\nu}}+\frac{\sqrt{t}}{{|A_{n,n}|^{2k_n-1}}}\int_{A_{n,n}^{k_n}}\int_{ \overline{B}^n_{k_n}(\g,b_n)} E_{\mathcal{D}_n}\Big[R^{\g,\g',t}_{2b_n}~X_{\g'}^n~R^{\g,\g',t}_{2b_n}~X_{\g}^n~R^{\g,\g',t}_{2b_n}\Big]d|\g'|d|\g| \Big| \\&\overset{(a)}{\le} \frac{6k_n^4|B^n_{2b_n}|^2}{\sqrt{|A_{n,n}|}\nu^4}\sup_{\g\in\G^{k_n}}\Big\| {X_{\g}^n}\Big\|_{3}^3
\end{split}\end{equation}  where to get $(a)$ we used the fact that: $\overline{B}^n_{k_n}(\g,2b_n)\le k_n^2|B^n_{b_n}| |A_{n,n}|^{k_n-1}.$ Similarly  we have:
\begin{equation}\begin{split}&\label{u*_3}\Big|a_2^{\gamma_{x,\nu}}+\frac{\sqrt{t}}{{|A_{n,n}|^{2k_n-1}}}\int_{A_{n,n}^{k_n}}\int_{ \overline{B}^n_{k_n}(\g,b_n)} E_{\mathcal{D}_n}\Big[R^{\g,\g',t}_{2b_n}X_{\g'}^nR^{\g,\g',t}_{2b_n}X_{\g}^nR^{\g,\g',t}_{2b_n}\Big]d|\g'|d|\g| \Big|\\&\le \frac{6k_n^4|B^n_{2b_n}|^2}{\sqrt{|A_{n,n}|}\nu^4}\sup_{\g\in \G^{k_n}}\Big\| {X_{\g}^n}\Big\|_{3}^3
\end{split}\end{equation}

Therefore using \cref{u*_1}, \cref{u*_2} and \cref{u*_3} we get that:
\begin{equation}\begin{split}&\label{nulle_3}\Big\|g(1)-g(0)\Big\|_{1}\\&\le\int_0^1\Big\|\frac{1}{{|A_{n,n}|^{2k_n-1}}}\int_{A_{n,n}^{k_n}}\int_{ \overline{B}^n_{k_n}(\g,b_n)}E_{\mathcal{D}_n}\Big[R^{\g,\g',t}_{2b_n}X_{\g}^nR^{\g,\g',t}_{2b_n}X^n_{\g'}R^{\g,\g',t}_{2b_n}\Big]d|\g'|d|\g|\\&\qquad\quad-E_{\mathcal{D}_n}\Big[R_n^t\frac{Y^{{sc},\eta'_n}}{2\sqrt{1-t}}R_n^t\Big]\Big\|_{1}~dt
+\frac{12k_n^4|B^n_{2b_n}|^2}{\sqrt{|A_{n,n}|}\nu^4}\sup_{\g\in \G^{k_n}}\Big\| {X_{\g}^n}\Big\|_{3}^3\\&\quad+\frac{2k_n^2\sup_{\g\in \G^{k_n}}\Big\|{X_{\g}^n}\Big\|^2_{2} \mathcal{R}^s_{n}[b_n]}{\nu^3}.
\end{split}\end{equation}

\noindent Moreover if we define the following shorthand notation  $R^{sc,\eta_n}_{g,t}:=R(W^{n,sc}_g(t),\gamma_{x,\nu})$ then by exploiting the free independence of $(Y^n_g)$  we obtain for all $t\in (0,1)$ that:
\begin{equation}\begin{split}&\label{nulle_2}\Big\|\frac{1}{{|A_{n,n}|}}\int_{A_{n,n}}E_{\mathcal{D}_n}\Big[R^{sc,\eta_n}_{g,t}Y^n_gR^{sc,\eta_n}_{g,t}Y^n_gR^{sc,\eta_n}_{g,t}\Big]d|g|
-E_{\mathcal{D}_n}\Big[R_n^t\frac{Y^{{sc},\eta'_n}}{2\sqrt{1-t}}R_n^t\Big]~dt\Big\|_{1}
\\&\le\frac{6|B_0|^2}{\sqrt{|A_{n,n}|}\nu^4}\Big\| {Y^{{sc},\eta'_n}}\Big\|_{3}^3.
\end{split}\end{equation}
By combining \cref{nulle_2} and \cref{nulle_3} we therefore have:
\begin{equation}\begin{split}&\label{nulle_4}\Big\|g(1)-g(0)\Big\|_{1}\\&\le\int_0^1\Big\|\frac{1}{{|A_{n,n}|^{2k_n-1}}}\int_{A_{n,n}^{k_n}}\int_{ \overline{B}^n_{k_n}(\g,b_n)}\hspace{-3mm}E_{\mathcal{D}_n}\Big[R^{\g,\g',t}_{2b_n}X_{\g}^nR^{\g,\g',t}_{2b_n}X^n_{\g'}R^{\g,\g',t}_{2b_n}\Big]d|\g'|d|\g|\\&\quad-\frac{1}{{|A_{n,n}|}}\int_{A_{n,n}}E_{\mathcal{D}_n}\Big[R^{sc,\eta_n}_{g,t}Y^n_gR^{sc,\eta_n}_{g,t}Y^n_gR^{sc,\eta_n}_{g,t}\Big]d|g|\Big\|_{1}dt
+\frac{12k_n^4|B^n_{2b_n}|^2}{\sqrt{|A_{n,n}|}\nu^4}\sup_{\g\in \G^{k_n}}\Big\| {X_{\g}^n}\Big\|_{3}^3\\&\quad+\frac{2k_n^2\sup_{\g\in \G^{k_n}}\Big\|{X_{\g}^n}\Big\|^2_{2} \mathcal{R}^s_{n}[b_n]}{\nu^3}+\frac{6|B_0|^2}{\sqrt{|A_{n,n}|}\nu^4}\Big\| {Y^{{sc},\eta'_n}}\Big\|_{3}^3
\end{split}\end{equation}
Therefore to get the desired result it is sufficient to compare \begin{equation*}
    \begin{split}&
      \frac{1}{|A_{n,n}|^{2k_n-1}} \int_{A_{n,n}^{k_n}}\int_{ \overline{B}^n_{k_n}(\g,b_n)}\hspace{-3mm}E_{\mathcal{D}_n}\Big[R^{\g,\g',t}_{2b_n}X_{\g}^nR^{\g,\g',t}_{2b_n}X^n_{\g'}R^{\g,\g',t}_{2b_n}\Big]d|\g'|d|\g|
        \\&\rm{with}\qquad\frac{1}{{|A_{n,n}|}}\int_{A_{n,n}}E_{\mathcal{D}_n}\Big[R^{sc,\eta_n}_{g,t}Y^n_gR^{sc,\eta_n}_{g,t}Y^n_gR^{sc,\eta_n}_{g,t}\Big]d|g|.
    \end{split}
\end{equation*}In this goal we introduce some shorthand notations. For all $t\in [0,1]$ we denote: \begin{equation*}\begin{split}&\eta_{\g,\g'}: a\rightarrow E_{\mathcal{D}_n}\Big[X_{\g}^n a X^n_{\g'}\Big],\quad S_{\g,\g'}^{2b_n}:=E_{\mathcal{D}_n}\Big[R(W_{\g,\g'}^{2b_n}(t),\gamma_{x,\nu})\Big],\quad S^0_{n}:= E_{\mathcal{D}_n}[R_n^t].\end{split}\end{equation*} Using the triangular inequality we remark that:
\begin{equation*}\begin{split}&\Big\|\frac{1}{{|A_{n,n}|^{2k_n-1}}}\int_{A_{n,n}^{k_n}}\int_{\overline{B}^n_{k_n}(\g,b_n)}\hspace{-3mm}E_{\mathcal{D}_n}\Big[R^{\g,\g',t}_{2b_n}X_{\g}^nR^{\g,\g',t}_{2b_n}X^n_{\g'}R^{\g,\g',t}_{2b_n}\Big]d|\g'|d|\g| -E_{\mathcal{D}_n}\Big[R_{n}^{t}\eta'_n(S^0_{n,t})R_n^t\Big] \Big\|_{1}
\\&\le \Big\|\frac{1}{{|A_{n,n}|^{2k_n-1}}}\int_{A_{n,n}^{k_n}}\int_{\overline{B}^n_{k_n}(\g,b_n)}E_{\mathcal{D}_n}\Big[R^{\g,\g',t}_{2b_n}X_{\g}^nR^{\g,\g',t}_{2b_n}X^n_{\g'}R^{\g,\g',t}_{2b_n}\Big]\\&\qquad \qquad \qquad\qquad \qquad \qquad \qquad -E_{\mathcal{D}_n}\Big[R_{\g,\g'}^{2b_n}\eta_{\g,\g'}(S_{\g,\g'}^{2b_n})R_{\g,\g'}^{2b_n}\Big]d|\g'|d|\g|  \Big\|_{1}
\\&\quad+ \Big\|\frac{1}{{|A_{n,n}|^{2k_n-1}}}\int_{A_{n,n}^{k_n}}\int_{\overline{B}^n_{k_n}(\g,b_n)}\hspace{-2mm} E_{\mathcal{D}_n}\Big[R_{\g,\g'}^{2b_n}\eta_{\g,\g'}(S_{\g,\g'}^{2b_n})R_{\g,\g'}^{2b_n}\Big]
- E_{\mathcal{D}_n}\Big[R_{n}^{t}\eta_{\g,\g'}(S^0_{n,t})R_{n}^{t}\Big]d|\g|d|\g'|\Big\|_{1}
\\&\quad+ \Big\|\frac{1}{{|A_{n,n}|^{2k_n-1}}}\int_{A_{n,n}^{k_n}}\int_{ \overline{B}^n_{k_n}(\g,b_n)}  E_{\mathcal{D}_n}\Big[R_{n}^{t}\eta_{\g,\g'}(S^0_{n,t})R_{n}^{t}\Big]d|\g'|d|\g| 
- E_{\mathcal{D}_n}\Big[R_{n}^{t}\eta'_{n}(S^0_{n,t})R_{n}^{t}\Big] \Big\|_{1}
\\&\le (c_1^{\gamma_{x,\nu}})+(c_2^{\gamma_{x,\nu}})+(c_3^{\gamma_{x,\nu}})
\end{split}\end{equation*}
We bound each term successively. The first term is bounded using the definition of $(\aleph^{*,j,\nu}_n[\cdot|\G_n])$, and the later terms are a consequence of the triangular inequality and \cref{ca_ira}. We focus first on the term $c_1^{\gamma_{x,\nu}}$. Using the definition of the free global mixing coefficients, we obtain that
\begin{equation*}\begin{split}&
\big|(c_1^{\gamma_{x,\nu}})\big|
\\&\le  \Big\|\frac{1}{{|A_{n,n}|^{2k_n-1}}}\int_{A_{n,n}^{k_n}}\int_{\overline{B}^n_{k_n}(\g,b_n)}E_{\mathcal{D}_n}\Big[R^{\g,\g',t}_{2b_n}X_{\g}^nR^{\g,\g',t}_{2b_n}X^n_{\g'}R^{\g,\g',t}_{2b_n}\Big]\\&\qquad \qquad \qquad\qquad \qquad \qquad \qquad -E_{\mathcal{D}_n}\Big[R_{\g,\g'}^{2b_n}\eta_{\g,\g'}(S_{\g,\g'}^{2b_n})R_{\g,\g'}^{2b_n}\Big]d|\g'|d|\g|  \Big\|_{1}
\\&\le  \Big\|\frac{1}{{|A_{n,n}|^{2k_n-1}}}\int_{A_{n,n}^{k_n}}\int_{\overline{B}^n_{k_n}(\g,b_n)}E_{\mathcal{D}_n}\Big[R^{\g,\g',t}_{2b_n}X_{\g}^n\overline{R^{\g,\g',t}_{2b_n}}X^n_{\g'}R^{\g,\g',t}_{2b_n}\Big]d|\g'|d|\g|  \Big\|_{1}\\&\quad+\Big\|\frac{1}{{|A_{n,n}|^{2k_n-1}}}\int_{A_{n,n}^{k_n}}\int_{\overline{B}^n_{k_n}(\g,b_n)}E_{\mathcal{D}_n}\Big[R^{\g,\g',t}_{2b_n}\overline{X_{\g}^nE_{\mathcal{D}_n}(R^{\g,\g',t}_{2b_n})X^n_{\g'}}R^{\g,\g',t}_{2b_n}\Big]d|\g'|d|\g|  \Big\|_{1}
\\&\le\frac{2 \sup_{\g\in \G^{k_n}}\Big\|X_{\g}^n\Big\|^2_{2}}{\nu^3|A_{n,n}|^{2k_n-1}}\int_{A_{n,n}^{k_n}}\int_{\overline{B}^n_{k_n}(\g,b_n)}\aleph_n^{*,j,\nu}[\bar{d}(\{\g,\g'\}, A_{n,n}^{k_n}\setminus \overline{B}_{k_n}(\g,2b_n)^n)|\mathbb{G}^{k_n}_n]d|\g'|d|\g|
\\&\overset{}{\le}\frac{2 k_n^2\sup_{\g\in \G^{k_n}}\Big\|X_{\g}^n\Big\|^2_{2}|B^n_{b_n}|\aleph_n^{*,j,\nu}[b_n|\mathbb{G}^{k_n}_n]}{\nu^3};
\end{split}\end{equation*} 
The next step consists in bounding $c_2^{\gamma_{x,\nu}}$. Using the triangular inequality we obtain that:
\begin{equation*}\begin{split}&
\Big\|E_{\mathcal{D}_n}\big(R_{\g,\g'}^{2b_n}\eta_{\g,\g'}(S_{\g,\g'}^{2b_n})
R_{\g,\g'}^{2b_n}\big) -E_{\mathcal{D}_n}\big(R_{n}^{t}\eta_{\g,\g'}(S^0_{n,t})
R_{n}^{t}\big)\Big\|_{1}
\\&\overset{}{\le} 
\Big\|\Big[R_{\g,\g'}^{2b_n}-R_n^t\Big]\eta_{\g,\g'}(S_{\g,\g'}^{2b_n})
R_{\g,\g'}^{2b_n}\Big\|_{1}
+\Big\|R_{n}^{t}\Big[\eta_{\g,\g'}(S_{\g,\g'}^{2b_n})-\eta_{\g,\g'}(S^0_{n,t})\Big]
R_{\g,\g'}^{2b_n}\Big\|_{1}
\\&\quad +\Big\|R_n^t\eta_{\g,\g'}(S^0_{n,t})
\Big[R_{\g,\g'}^{2b_n}-R_{n}^{t}\Big] \Big\|_{1}
\\&\overset{(a)}{\le}\frac{6k_n^2|B^n_{2b_n}|}{\sqrt{|A_{n,n}|}\nu^4}\sup_{\g\in \G^{k_n}}\Big\|X_{\g}^n\Big\|_{3}^3.
\end{split}\end{equation*} where to get (a) we used \cref{ca_ira}and Cauchy-Schwarz inequality. Therefore by averaging over $\g\in A_{n,n}^{k_n}$ and $\g'\in \overline{B}_{k_n}(\g',b_n)$ this implies that
\begin{equation*}\begin{split}
(c_2^{\gamma_{x,\nu}})
&\overset{(a)}{\le}\frac{6k_n^4|B^n_{2b_n}|^2}{\sqrt{|A_{n,n}|}\nu^4}\sup_{\g\in G^{k_n}}\Big\|X_{\g}^n\Big\|_{3}^3.\end{split}\end{equation*} where to get (a) we used the fact that $|\overline{B}_{k_n}(\g,2b_n)^n|\le k_n^2 |B^n_{2b_n}| |A_{n,n}|^{k_n-1}.$
 The next step is to bound $(c_3^{\gamma_{x,\nu}})$. Using the triangular inequality we have
 \begin{equation*}\begin{split}&(c_{3}^{\gamma_{x,\nu}})\\&\le
 \Big\|\frac{1}{{|A_{n,n}|^{2k_n-1}}}\int_{A_{n,n}^{k_n}\times \overline{B}^n_{k_n}(\g,b_n)} E_{\mathcal{D}_n}\Big(R_n^t\eta_{\g,\g'}(S^0_{n,t})R_n^t\Big)d|\g'|d|\g|  -E_{\mathcal{D}_n}\Big(R_n^t\eta'_n(S^0_{n,t})R_n^t\Big)\Big\|_{1}
\\&\le \Big\|\frac{1}{{|A_{n,n}|^{2k_n-1}}}\int_{A_{n,n}^{k_n}}\int_{A_{n,n}^{k_n}\setminus \overline{B}^n_{k_n}(\g,b_n)} E_{\mathcal{D}_n}\Big(R_n^t\eta_{\g,\g'}(S^0_{n,t})S_{n,t}^{0}R_n^t\Big)d|\g'|d|\g|  \Big\|_{1}
\\& \quad +\frac{1}{\nu^3} \Big\|\eta_n'(\cdot)-\eta_n(\cdot)\big|_{\mathcal{A}^{\rm{tail}}_{\mathcal{D}(\G_n^{k_n})}}\Big\|_{\rm{op}}
\\&\overset{(a)}{\le} \frac{k_n^2\sup_{\g\in \G^{k_n}}\big\|X_{\g}^n\big\|_{2}^2}{\nu^3}  \mathcal{R}^s_n[b_n] +\frac{1}{\nu^3} \Big\|\eta_n'(\cdot)-\eta_n(\cdot)\big|_{\mathcal{A}^{\rm{tail}}_{\mathcal{D}(\G_n^{k_n})}}\Big\|_{\rm{op}};
\end{split}\end{equation*}  where to obtain (a) we exploited the Cauchy-Schwarz inequality.
~This implies that
\begin{equation*}\begin{split}&\Big\|g(1)-g(0)\Big\|_{1}\\&\le\int_0^1\Big\|S_{n,t}^{0}\eta'_n(S^0_{n,t})S_{n,t}^{0}-\frac{1}{{|A_{n,n}|}}\int_{A_{n,n}}E_{\mathcal{D}_n}\Big[R^{sc,\eta_n}_{g,t}Y^n_gR^{sc,\eta_n}_{g,t}Y^n_gR^{sc,\eta_n}_{g,t}\Big]d|g|\Big\|_{1}dt
\\&~+\frac{12k_n^4|B^n_{2b_n}|^2}{\sqrt{|A_{n,n}|}\nu^4}\sup_{\g\in \G^{k_n}}\Big\| {X_{\g}^n}\Big\|_{3}^3+\frac{3k_n^2\sup_{\g\in \G^{k_n}}\Big\|{X_{\g}^n}\Big\|^2_{2} \mathcal{R}^s_{n}[b_n]}{\nu^3}+\frac{6|B_0|^2}{\sqrt{|A_{n,n}|}\nu^4}\Big\| {Y^{{sc},\eta'_n}}\Big\|_{3}^3
\\&+\frac{6k_n^4|B^n_{2b_n}|^2}{\sqrt{|A_{n,n}|}\nu^4}\sup_{\g\in \G^{k_n}}\Big\|X_{\g}^n\Big\|_{3}^3+\frac{2 k_n^2\sup_{\g\in \G^{k_n}}\Big\|X_{\g}^n\Big\|^2_{2}|B^n_{b_n}|\aleph_n^{*,j,\nu}[b_n|\mathbb{G}^{k_n}_n]}{\nu^3}
\\&+\frac{1}{\nu^3} \Big\|\eta_n'(\cdot)-\eta_n(\cdot)\big|_{\mathcal{A}^{\rm{tail}}_{\mathcal{D}(\G_n^{k_n})}}\Big\|_{\rm{op}}.
\end{split}\end{equation*}

Moreover similarly we can prove that 
\begin{equation*}\begin{split}&\int_0^1\Big\|S_{n,t}^{0}\eta'_n(S^0_{n,t})S_{n,t}^{0}-\frac{1}{{|A_{n,n}|}}\int_{A_{n,n}}E_{\mathcal{D}_n}\Big[R^{sc,\eta_n}_{g,t}Y^n_gR^{sc,\eta_n}_{g,t}Y^n_gR^{sc,\eta_n}_{g,t}\Big]d|g|\Big\|_{1}dt
\\&\le\frac{6|B_{0}|^2}{\sqrt{|A_{n,n}|}\nu^4}\Big\|Y^{{sc},\eta'_n}\Big\|_{3}^3.
\end{split}\end{equation*}
Therefore we finally obtain that:
\begin{equation*}\begin{split}&\Big\|g(1)-g(0)\Big\|_{1}\\&\le 
\frac{1}{\sqrt{|A_{n,n}|}\nu^4}\Big[{12|B_{0}|^2}\Big\|Y^{{sc},\eta'_n}\Big\|_{3}^3+{18k_n^4|B^n_{2b_n}|^2}\sup_{\g\in \G^{k_n}}\Big\| {X_{\g}^n}\Big\|_{3}^3\Big]+
\\&~+\frac{1}{\nu^3} \Big\|\eta_n'(\cdot)-\eta_n(\cdot)\big|_{\mathcal{A}^{\rm{tail}}_{\mathcal{D}(\G_n^{k_n})}}\Big\|_{\rm{op}}+\frac{ k_n^2\sup_{\g\in \G^{k_n}}\Big\|X_{\g}^n\Big\|^2_{2}}{\nu^3}\Big[2|B^n_{b_n}|\aleph_n^{*,j,\nu}[b_n|\mathbb{G}^{k_n}_n]+ \mathcal{R}_n^s[b_n]\Big].
\end{split}\end{equation*}
\end{proof}
\noindent \emph{Proof of \cref{theorem_3}}: We prove that \cref{bartlett} implies that \cref{theorem_3} holds.
\begin{proof}
Taking $\eta'_n(\cdot):=\eta_n(\cdot)$ gives the desired result.
\end{proof}
We prove that \cref{bartlett} implies that \cref{theorem_4} holds.
\begin{prop} 
Let $(X^n_{\g})$ be a triangle array of free random variables satisfying all the conditions of \cref{theorem_4}. Let $S_n(\cdot)$ denote the (operator-valued) Stiejles transform of \\$\frac{1}{|A_{n,n}|^{k_n-\frac{1}{2}}}\int_{\A_{n,n}^{k_n}}X_{\g}d|\g|$, and $S_n^{sc}(\cdot)$ the Stiejles transform of the operator valued semi-circular operator $Y^{{sc},\eta_n}$ with radius $\tilde \eta_n(\cdot)$. 
The following upper bound holds
\begin{equation*}\begin{split}&\Big\|S_n(\gamma_{x,\nu})-S_n^{sc}(\gamma_{x,\nu})\Big\|_{1}
\\&\le
\frac{1}{|A_{n,n}|\nu^3}\Big[k_n^4|B^n_{b_n}|\sup_{\g \in \G^{k_n}}\|X_{\g}^n\|_2^2\mathcal{T}^m_n+2|A_{n,n}\triangle B^n_{b_n}A_{n,n}|\Big[\sup_{\g \in \G^{k_n}}\|X_{\g}^n\|_2^2+ \|Y^{{sc},\eta_n}\|_2^2\Big]\Big]
\\&+\frac{1}{\sqrt{|A_{n,n}|}\nu^4}\Big[{12|B_{0}|^2}\Big\|Y^{{sc},\eta_n}\Big\|_{3}^3+{18k_n^4|B^n_{2b_n}|^2}\sup_{\g\in \G^{k_n}}\Big\| {X_{\g}^n}\Big\|_{3}^3\Big]
\frac{5k_n^2\sup_{\g\in \G^{k_n}}\Big\|{X_{\g}^n}\Big\|^2_{2} \mathcal{R}^s_{n}[b_n]}{\nu^4}
\\&~+\frac{ k_n^2\sup_{\g\in \G^{k_n}}\Big\|X_{\g}^n\Big\|^2_{2}|B^n_{b_n}|\aleph_n^{*,j,\nu}[b_n|\mathbb{G}^{k_n}_n]}{\nu^3}.
\end{split}\end{equation*}
where $\gamma_{x,\nu}=x+i\nu$.
\end{prop}

\begin{proof}  We choose $\eta'_n(\cdot)=\tilde \eta_n(\cdot)$. As $(X^n_{\g})$ satisfies all the conditions of \cref{theorem_4} it also satisfies the conditions of \cref{theorem_3}. Therefore using \cref{bartlett_2} we know that 
\begin{equation*}\begin{split}&\Big\|S_n(\gamma_{x,\nu})-S_n^{sc}(\gamma_{x,\nu})\Big\|_{1}
\\&\le
\frac{{12|B^n_{0}|^2}\Big\|Y^{sc,\eta_n}\Big\|_{3}^3+{18k_n^4|B^n_{2b_n}|^2}\sup_{\g\in \G^{k_n}}\Big\| {X_{\g}^n}\Big\|_{3}^3}{\sqrt{|A_{n,n}|}\nu^4}+\frac{3k_n^2\sup_{\g\in \G^{k_n}}\Big\|{X_{\g}^n}\Big\|^2_{2} \mathcal{R}^s_{n}[b_n]}{\nu^3}
\\&~+\frac{ k_n^2\sup_{\g\in \G^{k_n}}\Big\|X_{\g}^n\Big\|^2_{2}|B^n_{b_n}|\aleph_n^{*,j,\nu}[b_n|\mathbb{G}^{k_n}_n]}{\nu^3}+\frac{1}{\nu^3} \Big\|\tilde \eta_n(\cdot)-\eta_n(\cdot)\big|_{\mathcal{A}^{\rm{tail}}_{\mathcal{D}(\G_n^{k_n})}}\Big\|_{\rm{op}}.
\end{split}\end{equation*}
The key point of this proof is to upper-bound $ \Big\|\eta_n(\cdot)-\tilde \eta_n(\cdot)\big|_{\mathcal{A}^{\rm{tail}}_{\mathcal{D}(\G_n^{k_n})}}\Big\|_{\rm{op}}$.   We denote $\eta_{\g,\g'}:a\rightarrow E_{\mathcal{D}_n}\big[X^n_{\g}a X^{n}_{\g'}\big]$. Let $a\in \F^{\rm{tail}}_{\mathcal{D}(\G_n^{k_n})}(X^n)$ be an operator verifying $\|a\|_{\infty}\le 1$, we remark that $$\eta_n(a)=\frac{1}{|A_{n,n}^{2k_n-1}|}\int_{A_{n,n}^{2k_n}}\eta_{\g,\g'}(a)
d|\g|d|\g'|.$$
 To prove the desired result we firstly propose a simpler form for $\eta_{\g,\g'}(\cdot)$ when $\g'\in \overline{B}^n_{b_n}(\g)$.  We write: $$\mathcal{I}_{b_n}(\g,\g'):=\{(i,j)| d(\g_i,\g'_j)\le b_n\}.$$ Let $\g^1,\g^2,\g^3,\g^4\in \G^{k_n}_n$ be elements such that $\{(i,j)\}\subset\mathcal{I}_{b_n}(\g^1,\g^2)$ and such that ${\g^1_i}=\g^{3}_i$ and $\g^2_j=\g^4_j$. We want to compare $\eta_{\g^1,\g^2}(\cdot)$ with $\eta_{\g^3,\g^4}(\cdot)$. In this goal, we build $(\g^{1,l})$ and $(\g^{2,l})$ to be interpolating sequences between $(\g^1,\g^2)$ and $(\g^3,\g^4)$. We define them as $$\g^{1,l}_m:= \begin{cases}\g^{3}_m ~\rm{if}~ m\le l\\ \g^{1}_m ~\rm{otherwise}\end{cases}\qquad \g^{2,l}_m:= \begin{cases}\g^{4}_m ~\rm{if}~ m\le l\\ \g^{2}_m ~\rm{otherwise}\end{cases}.$$ By the triangular inequality and the definition of free marginal mixing coefficients, the following holds
 \begin{equation*}\begin{split}&\Big\|\eta_{\g^1,\g^2}(a)-\eta_{\g^3,\g^4}(a)\Big\|_{1}
 \\&\le \sum_{l=1}^{k_n}\Big\|\eta_{\g^{1,l},\g^{2,l}}(a)-\eta_{\g^{1,l-1},\g^{2,l-1}}(a)\Big\|_{1}
 \\&\le \sup_{\g \in \G^{k_n}}\|X_{\g}^n\|_2^2\sum_{l=1}^{k_n} \aleph^m_n\Big[\bar{d}(\{\g^1_l,\g^3_l\}, [\g^1,\g^2]^{\setminus l}\bigcup [\g^3,\g^4]^{\setminus l}) \Big|\G^{k_n}\Big]
 \\&\qquad \qquad \qquad \qquad \qquad +\aleph^m_n\Big[\bar{d}(\{\g^2_l,\g^4_l\}, [\g^2,\g^1]^{\setminus l}\bigcup [\g^4,\g^3]^{\setminus l}) \Big|\G^{k_n}\Big]
\end{split}\end{equation*}Therefore using the definition of $\bar{X}^n_{i,\g^1_{i}}$ and $\bar{X}^n_{j,\g^2_{j}}$ we obtain that:
 \begin{equation*}\begin{split}&\Big\|\eta_{\g_1,\g_2}(a)-E_{\mathcal{D}_n}\Big(\bar{X}^n_{i,\g^1_i}a\bar{X}^n_{i,\g^2_j}\Big)\Big\|_{1}
\\&\le 2\sup_{\g \in \G^{k_n}}\|X_{\g}^n\|_2^2\sum_{l=1}^{k_n} \aleph^m_n\Big[\bar{d}(\g^1_l, [\g^1,\g^2]^{\setminus l}) \Big|\G^{k_n}\Big]
 +\aleph^m_n\Big[\bar{d}(\g^2_l, [\g^2,\g^1]^{\setminus l})\Big|\G^{k_n}\Big].
\end{split}\end{equation*}
Which in term implies that if we write: $\mathcal{J}_{b_n,n}^{i,j}:=\{(\g,\g')|\g,\g'\in A_{n,n}^{k_n}~\rm{s.t}~\{i,j\}\subset\mathcal{I}_{b_n}(\g,\g')\}$ then we have
\begin{equation*}\begin{split}&
\Big\|\frac{1}{|A_{n,n}|^{2k_n-1}}\int_{\mathcal{J}_{b_n,n}^{i,j}}\eta_{\g,\g'}(a)-E_{\mathcal{D}_n}\Big(\bar{X}^n_{i,\g_i}a\bar{X}^n_{j,\g'_j}\Big)d|\g'|d|\g| \Big\|_{1}
\\&\le  \frac{4k_n^4|B^n_{b_n}|\sup_{\g \in \G^{k_n}}\|X_{\g}^n\|_2^2}{|A_{n,n}|}\mathcal{T}^m_{n}.
\end{split}\end{equation*} 
Therefore using the triangular inequality we obtain that:
\begin{equation*}\begin{split}&
 \Big\|\frac{1}{{|A_{n,n}|^{2k_n-1}}}\int_{A_{n,n}^{k_n}}\int_{ \overline{B}^n_{k_n}(\g,b_n)} \eta_{\g,\g'}(a)d|\g'|d|\g|  -\eta_n(a) \Big\|_{1}
\\&\le\sum_{i,j\le k_n}\Big\|\frac{1}{|A_{n,n}|^{2k_n-1}}\int_{\mathcal{J}_{b_n,n}^{i,j}}\eta_{\g,\g'}(a)-E_{\mathcal{D}_n}\Big(\bar{X}^n_{i,\g_i}a\bar{X}^n_{j,\g'_j}\Big)d|\g'|d|\g| \Big\|_{1}
\\&+\sum_{i,j\le k_n}\Big\|\frac{1}{|A_{n,n}|}\int_{A_{n,n}\times B^n(g,b_n)}E_{\mathcal{D}_n}\Big(\bar{X}^n_{i,g}a\bar{X}^n_{j,g'}\Big)d|g'|d|g| -\int_{ B^n(e,b_n)}E_{\mathcal{D}_n}\Big(\bar{X}^n_{i,e}a\bar{X}^n_{j,g}\Big)d|g|\Big\|_{1}
\\&+\Big\|\int_{ A_{n,n}\setminus B^n(e,b_n)}E_{\mathcal{D}_n}\Big(\bar{X}^n_{i,e}a\bar{X}^n_{j,g}\Big)d|g|\Big\|_{1}
\\&\overset{(a)}{\le} \frac{4k_n^4|B^n_{b_n}|\sup_{\g \in \G^{k_n}}\|X_{\g}^n\|_2^2}{|A_{n,n}|}\mathcal{T}^m_n+\sup_{\g\in \G^{k_n}}\big\|X_{\g}^n\big\|_{2}^2 \frac{|A_{n,n}\triangle B_{b_n}A_{n,n}|}{|A_{n,n}|}\\&\quad+{k_n^2\sup_{\g\in \G^{k_n}}\big\|X_{\g}^n\big\|_{2}^2}  \mathcal{R}^s_n[b_n].
\end{split}\end{equation*}  where to obtain (a) we exploited the Cauchy-Schwarz inequality.
~Moreover using the definition of $\aleph^{*,s}[\cdot|\G^{k_n}_{n}]$ we obtain that\begin{equation*}\begin{split}&
 \Big\|\frac{1}{{|A_{n,n}|^{2k_n-1}}}\int_{A_{n,n}^{k_n}}\int_{A_{n,n}^{k_n}\setminus \overline{B}^n_{k_n}(\g,b_n)} \eta_{\g,\g'}(a)d|\g'|d|\g| \Big\|_{1}
 \\&\le \sum_{b\ge b_n}\aleph^{*,s}_n[b|\G_n^{k_n}] ~\frac{\sup_{\g\in \G_n^{k_n}}\|X_{\g}^n\|_2^2}{{|A_{n,n}|^{2k_n-1}}}\int_{A_{n,n}^{k_n}}\int_{A_{n,n}^{k_n}\setminus \overline{B}^n_{k_n}(\g,b_n)}\hspace{-3mm} \mathbb{I}\big(d(\g,\g')\in [b,b+1)\big)d|\g'|d|\g|
\\&\le k_n^2 \sup_{\g\in \G_n^{k_n}}\|X_{\g}^n\|_2^2 \mathcal{R}^s_{n}[b_n].
\end{split}\end{equation*}
As $a\in \A^{\rm{tail}}_{\mathcal{D}_n(\G_n^{k_n})}$ is arbitrary we have:
\begin{equation*}\begin{split}&\Big\|\eta_n(\cdot)-\eta_n(\cdot)\Big|_{\A^{\rm{tail}}_{\mathcal{D}_n(\G_n^{k_n})}}\Big\|_{\rm{op}}\\&\le 
\frac{\sup_{\g \in \G^{k_n}}\|X_{\g}^n\|_2^2}{|A_{n,n}|\nu^3}\Big[k_n^4|B^n_{b_n}|\sup_{\g \in \G^{k_n}}\|X_{\g}^n\|_2^2\mathcal{T}^m_n+|A_{n,n}\triangle B_{b_n}A_{n,n}|\Big]\\&\quad+
\frac{2k_n^2\sup_{\g\in \G^{k_n}}\Big\|{X_{\g}^n}\Big\|^2_{2} \mathcal{R}^s_{n}[b_n]}{\nu^4}.
\end{split}\end{equation*}

Therefore we finally obtain that:
\begin{equation*}\begin{split}&\Big\|g(1)-g(0)\Big\|_{1}\\&\le 
\frac{1}{|A_{n,n}|\nu^3}\Big[k_n^4|B^n_{b_n}|\sup_{\g \in \G^{k_n}}\|X_{\g}^n\|_2^2\mathcal{T}^m_n+2|A_{n,n}\triangle B_{b_n}A_{n,n}|\Big[\sup_{\g \in \G^{k_n}}\|X_{\g}^n\|_2^2+ \|Y^{sc,\eta_n}\|_2^2\Big]\Big]
\\&+\frac{{12|B_{0}|^2}\Big\|Y^{sc,\eta_n}\Big\|_{3}^3+{18k_n^4|B^n_{2b_n}|^2}\sup_{\g\in \G^{k_n}}\Big\| {X_{\g}^n}\Big\|_{3}^3}{\sqrt{|A_{n,n}|}\nu^4}
+\frac{5k_n^2\sup_{\g\in \G^{k_n}}\Big\|{X_{\g}^n}\Big\|^2_{2} \mathcal{R}^s_{n}[b_n]}{\nu^4}
\\&~+\frac{ k_n^2\sup_{\g\in \G^{k_n}}\Big\|X_{\g}^n\Big\|^2_{2}|B^n_{b_n}|\aleph_n^{*,j,\nu}[b_n|\mathbb{G}^{k_n}_n]}{\nu^3}.
\end{split}\end{equation*}
\end{proof}


\section{Preliminary lemmas for \cref{theorem_1,theorem_3,theorem_4}}
\begin{lemma}\label{anti_hero}
Choose $g,g'\in G$ and  $G\subset \G_n$ to be such that $\bar{d}_n\big(\{g\}, ~G\cup \{g'\}\big)\ge b$.  For all $Y_1,Y_2,Y_3\in \mathcal{F}_G(X^n)$ the following holds
\begin{align*}
    \Big\|E_n\Big[Y_1 X_{g'}^nY_2{X_{g}^n}Y_3\Big]\Big
\|_{1}\le \aleph^s_n[b|\G_n]
\end{align*}
\end{lemma}
\begin{proof}Choose $g,g'\in G$ and  $G\subset \G$ to be such that $\bar{d}_n\big(\{g\}, ~G\cup \{g'\}\big)\ge b$.  Let $Y_1,Y_2,Y_3\in \mathcal{F}_G(X^n)$. Firstly we remark that by definition of the noncommutative conditional expectation we have
    \begin{align*}&
\Big\|E_n\Big[Y_1 X_{g'}^nY_2X_g^nY_3\Big]\Big
\|_{1}
\\&=\sup_{\substack{a\in \mathcal{F}^{\rm{tail}}(X^n)\\\|a\|_{\infty\le 1} }}\tau\Big(aY_1X_{g'}^nY_2{X_{g}^n}Y_3\Big).
    \end{align*}Using the fact that $\tau$ is tracial we observe that  for all $a\in  \mathcal{F}^{\rm{tail}}(X^n)$ such that $\|a\|_{\infty}\le 1$ then we have  \begin{align*}&
\tau\Big(aY_1X_{g'}^nY_2{X_{g}^n}Y_3\Big)\\&=\tau\Big(Y_3^*{X_{g}^n}Y_2^*aY_1X_{g'}^n\Big)
\\&\overset{(a)}{\le} \aleph^s_n[b|\G_n]
    \end{align*}where to obtain (a) we used the fact that $\|Y_3\|_{\infty}, \|Y_2^*aY_1\|_{\infty}\le 1$. 
\end{proof}
\begin{lemma}\label{anti_hero_2}
Choose $\g,\g'\in G^{k_n}$ and  $G\subset \G$ to be such that $\bar{d}_n\big(\{\g\}, ~G\cup \{\g'\}\big)\ge b$.  For all $Y_1,Y_2,Y_3\in \mathcal{F}^d_G(X^n)$ the following holds
\begin{align*}
    \Big\|E_n\Big[Y_1 X_{\g'}^nY_2{X_{\g}^n}Y_3\Big]\Big
\|_{1}\le \aleph^{*s}_n[b|\G_n]
\end{align*}
\end{lemma}
\begin{proof}The proof follows in the same fashion than the proof of \cref{anti_hero}.
\end{proof}\subsection{Preliminary results to  \cref{marre_2} and \cref{good_30}}
In this section we present a known result ( see e.g \cite{austern2018limit}), that will be used in the proof of \cref{marre_2} and \cref{good_30}.

\vspace{2mm}

\noindent 
Let $Y:=(Y_{\z})_{\z\in \Z^d}$ be a stationary random field with entries taking value in a Borel space $\mathcal{Y}$. Write $(\alpha[b])$ the strong mixing coefficients of $Y$. Let $f:\mathcal{Y}^{k_1}\times \mathcal{Y}^{k_2}\rightarrow \mathbb{R}$ be a measurable function. 
For all finite subset $ Z:=\{\z_1,\dots,\z_{|Z|}\}\subset \Z^d$ we write $ Z\cdot Y:=(Y_{\z_1},\dots, Y_{\z_{|Z|}})$. We denote $\tilde Y:=(\tilde Y_{\z})$ an independent copy of $Y$.
\begin{lemma}\label{chanson}
Fix ${l\in\mathbb{N}}$. 
Select any subsets $Z_1,Z_2\subset\Z^d$ of respective size  $k $ and 2 that $\min_{\z\in Z_1,\z'\in Z_2 }\min_{j\le d}|\z_j-\z'_j|\ge l$. Then the following holds:
\begin{equation*}\big|\mathbb{E}(f(Z_1 \cdot Y,Z_2\cdot Y))- \mathbb{E}(f(Z_1 \cdot Y,Z_2 \cdot \tilde Y))\big| \le 4  C~\alpha(l)^{\frac{\epsilon}{2+\epsilon}},\end{equation*}
where $C=\big\|f(Z_1\cdot Y,Z_2\cdot Y)-f(Z_1 \cdot Y,Z_2 \cdot\tilde Y) \big\|_{L_{1+\frac{\epsilon}{2}} }$\;.
%

\end{lemma}
\begin{proof} 
Abbreviate
\begin{equation*}
  \Delta h(Y):=f(Z_1 \cdot Y,Z_2\cdot Y)-  f(Z_1 \cdot Y, Z_2 \cdot \tilde Y)\;.
\end{equation*}
We first consider the case ${\big\|\Delta h\big\|_{\infty}<\infty}$, and then the general case.\\
\emph{Case 1: ${\big\|\Delta h\big\|_{\infty}<\infty}$}.\\
Fix ${\delta>0}$. Then there is $N_{\delta}\in \mathbb{N}$ sets $(A_i,B_i)_{i \le N_{\delta}}$ and
coefficients ${c_1,\dots,c_{N_{\delta}}}$ with ${|c_i|\le \big\|\Delta h\big\|_{\infty}}$ such that the approximation
\begin{equation*}
  \Delta h^*(Y):=\sum_{i=1}^{N_{\delta}} c_i \mathbb{I}\big({Z_1 \cdot Y\!\in\! A_i}\big)
  \bigl(\mathbb{I}\big({Z_2 \cdot Y\!\in\! B_i}) - \mathbb{I}\big({Z_2 \cdot\tilde Y\!\in\! B_i}\big)\bigr)
\end{equation*}
satisfies ${\big\|\Delta h(Y) - \Delta h^*(Y)\big\|_{\infty} \le \delta}$.
Moreover we have,
\begin{equation*}\begin{split}\big|\mathbb{E}\Big(\Delta h^*(Y)\Big)\big|
&\le \sum_{i=1}^{N_{\delta}} |c_i| \big|\mathbb{E}\big[\mathbb{I}\big({Z_1 \cdot Y\!\in\! A_i}\big)
    \bigl(\mathbb{I}\big({Z_2 \cdot Y\!\in\! B_i}\big) - \mathbb{I}\big({Z_2 \cdot \tilde Y\!\in\! B_i}\big)\bigr)\big]\big|
    \\& \le 2  \big\|\Delta h\big\|_{\infty}~\alpha(l)\;,
\end{split}\end{equation*}
where the second inequality follows from the definition of the $\alpha$-mixing coefficients and by the triangle inequality.
Since $\delta$ may be arbitrarily small, we have
\begin{equation*}
  \big|\mathbb{E}\big[f(Z_1 \cdot Y,Z_2 \cdot Y)-f(Z_1\cdot Y,Z_2 \cdot\tilde Y])\big]\big|  \le  2 \big\| \Delta h\big\|_{\infty}~\alpha(l).
\end{equation*}

{\noindent\emph{Case 2: $\big\|\Delta h\big\|_{\infty}$ not bounded}.}
With no loss of generality, we can suppose that $\big\|\Delta h\big\|_{1+\frac{\epsilon}{2}}\le 1$.
For ${r\in\mathbb{R}}$, define ${\Delta h_r := \Delta h~\mathbb{I}\big(\Delta h\le r\big)}$
and ${\overline{\Delta h_r}:=\Delta h-\Delta h_r}$.
By H\"older's inequality we have,
\begin{equation*}
  \begin{split}
    \big|\mathbb{E}[f(Z_1 \cdot Y,Z_2 \cdot Y)-f(Z_1 \cdot Y,Z_2 \cdot \tilde Y)\big]\big|
    & \le \big|\mathbb{E}(\Delta h_r)\big|+ \big|\mathbb{E}(\overline{\Delta h_r})\big|
    \\& \le 2r \alpha(l)+ 2r^{-\frac{\epsilon}{2}}\;.
\end{split}
\end{equation*}
The result follows for ${r=\alpha(l)^{\frac{-2}{2+\epsilon}}}$.

\end{proof}
\section{Proof of \cref{eq_suis}}
\begin{proof}
We first establish that if $(K_g)$ defines a group action on $\mathcal{A}_{\tau}$ then for all $g,g_1,\dots,g_k\in \mathbb{G}$ the following holds ${(X_{g_1},\dots,X_{g_k})}\overset{d}{=}{(X_{gg_1},\dots,X_{gg_k})}$. In this goal we firstly remark that as $K_g(\cdot)$ is a *-automorphism if $(X_{g_1},\dots,X_{g_k})\in \A^k$ then for all $P\in \mathbb{C}<x_1,\dots,x_k,x^*_1,\dots,x^*_k>$  we have $$P(X_{gg_1},\dots,X_{gg_k},X^*_{gg_1},\dots,X^*_{gg_k})=K_g\Big(P(X_{g_1},\dots,X_{g_k},X^*_{g_1},\dots,X^*_{g_k})\Big).$$ Therefore we have $$\tau\Big(P(X_{gg_1},\dots,X_{gg_k},X^*_{gg_1},\dots,X^*_{gg_k})\Big)=\tau\Big(P(X_{g_1},\dots,X_{g_k},X^*_{g_1},\dots,X^*_{g_k})\Big).$$
Which implies the equality in distribution.
If $(X_{g_1},\dots,X_{g_k})$ does not belong to $\A^k$, then by definition of $\mathcal{A}_{\tau}$ we know that there is a sequence $(p^n_{1},\dots,p^n_{k})$ of projectors of $\mathcal{H}_{\lambda}$ such that:
\begin{itemize}
    \item $(X_{g_1} p_{1}^n,\dots, X_{g_k}p_{k}^n)\in \A^k$;
    \item $\max_{i\in \dbracket{k}}\tau(1-p_i^n)\rightarrow 0$.
\end{itemize}
As $X_{g_1}p_1^n,\dots,X_{g_k}p_k^n\in \A$, using the result we just proved the following holds
$$(X_{g_1} p_{1}^n,\dots, X_{g_k}p_{k}^n)\overset{d}{=}(K_g(X_{g_1} p_{1}^n),\dots, K_g(X_{g_k}p_{k}^n)).$$
As $K_g$ is a *-automorphism the following holds \begin{equation*}
    \begin{split}
K_g\Big(X_{g_1} p_{1}^n,\dots, X_{g_k}p_{k}^n\Big)&=\Big(K_g(X_{g_1}) K_g(p_{1}^n),\dots, K_g(X_{g_k})K_g(p_{k}^n)\Big)
\\&=\Big(X_{gg_1} K_g(p_{1}^n),\dots, X_{gg_k}K_g(p_{k}^n)\Big).\end{split}\end{equation*}

\noindent Using hypothesis $H_2$ we have $\tau(1-K_g(p_i^n))=\tau(1-p_i^n)$; which implies that $\max_{i\le k}\tau(1-K_g(p_i^n))\rightarrow 0$.
Moreover we observe that $K_g(p_i^n)$ is a projector as it satisfies $K_g(p_i^n)^2=K_g({p_i^n}\times p_i^n)=K_g(p_i^n).$
Therefore using the definition of multivariate distributions we obtain that the first claim of \cref{eq_suis} holds.

\noindent The next goal is to prove that the second part of \cref{eq_suis} holds. Let $\{g_i\}\subset\mathbb{G}$ be a dense countable  subset of $\mathbb{G}$. For all $g'\in \mathbb{G}$ there are sequences $(p_{g_i}^n)$ and $(p_{g_i}^{n'})$ of projectors of $\mathcal{H}_{\lambda}$ such that
\begin{itemize}
    \item $\tau(1-{p^{n'}_{g_i}}),\tau(1-p_{g_i}^n)\rightarrow 0$ for all $i\in \mathbb{N}$;
    \item $(Z_{g_1}p_{g_1}^n,\dots,Z_{g_k}p_{g_k}^n)\overset{d}{=}(Z_{g'g_1}p_{g_1}^{n'},\dots,Z_{g
    'g_k}p_{g_k}^{n'})$ for all $k\in \mathbb{N}$.
\end{itemize}
Then according to  \cite{nica2006lectures} theorem 4.10 we know that there is $K_{g'}$ a $*-$automorphism of $\B$ such that $K_{g'}(Z_{g_i}p_{g_i}^n)=Z_{g'g_i}p_{g_i}^{n'}$ for all $i\le n$ and all $n\in \mathbb{N}$. As $K_{g'}$  is a *-automorphism we have that it is a.e continuous: $$K_{g'}(Z_{g_i})=\lim_{n\rightarrow \infty} K_{g'}(Z_{g_i}p^n_{g_i})=\lim_{n\rightarrow \infty}Z_{g'
g_i}p^{n'}_{g_i} =Z_{g'g_i}.$$ Therefore we obtain that $\forall i\in \mathbb{N}$ we have $K_{g'}(Z_{g_i})=Z_{g'g_i}$. Using the density of the subset $\{g_i\}$ we get that $K_{g'}$ satisfies $H_1$ and is such that: $K_{g'}(Z_g)=Z_{g'g}$ for all $g\in\G$. This defines a net $(K_g)$ of *-automorphims that satisfy property $(H_2)$ on $\B$. 

\end{proof}
\section{Proof of \cref{luna}}
\begin{proof}
To prove the desired result we need to check that the net $(X_{\g})$ verifies condition $(H_3)$. For ease of notation we write $Z_{\g}:=(Y_{\g_1},\dots,Y_{\g_k})$ for all $\g\in \G^k$.
We note that for all $\g^1,\dots,\g^z\in \G^k$ we have: $$\Big(Z_{\g^1},\dots,Z_{\g^z}\Big)\overset{d}{=}\Big(Z_{\g'\g^1},\dots,Z_{\g'\g^z}\Big),\qquad \forall \g'\in \mathcal{D}_k(\G).$$Therefore if $\Phi\in \mathbb{C}\big<x_1,\dots,x_k,x_1^*,\dots,x^*_k>$  was a polynomial then using the definition of multivariate distributions we would obtain that $$\Big(X_{\g^1},\dots,X_{\g^z}\Big)\overset{d}{=}\Big(X_{\g'\g^1},\dots,X_{\g'\g^z}\Big),\qquad \forall \g'\in \mathcal{D}_k(\G).$$
In general, by the Stone-Weierstrass theorem as $\Phi(\cdot)$ is continuous there is a sequence of polynomials $(\Phi_n:\A^k\rightarrow \A)\in \mathbb{C}\big<x_1,\dots,x_k,x_1^*,\dots,x^*_k>^{\mathbb{N}}$ such that $\big\|\Phi_n(Y_{\g_1},\dots, Y_{\g_k})-X_{\g}\|_{\infty}\rightarrow 0$ for all $\g\in \G^k$. This implies that condition $(H_3)$ holds. 
\end{proof}

\section{Proof of \cref{u_stat_1}}
\begin{proof}Let $\g\in \G^k$ we shorthand $\mathcal{S}(\g):=\{\g_1,\dots,\g_k\}$.
As $\Phi$ is a polynomial we have $\Phi(Y_{\g_1},\dots,Y_{\g_k})\in \mathcal{F}_{\mathcal{S}(\g)}(Y)$. Moreover, using the definition of $\mathcal{F}^{\rm{tail}}_{\mathcal{D}(G^k)}$ one can also prove that $E_{\mathcal{D}}(\cdot)=E(\cdot)$. The results of \cref{u_stat_1} are then following directly from the definition of $\aleph'[\cdot|\G].$
\end{proof}
\section{Proof of \cref{good_intro},  and \cref{good_30}}

We note that \cref{good_intro} is a  consequence of \cref{good_30}. Therefore we only present the proof for the latter.




\subsection{Proof of \cref{good_30}}
\begin{proof} We remark that we can suppose without loss of generality that the entries of the random matrices $(X^{\z,m})$ have mean $0$ and do so in the following. The first step of the proof consists on establishing the form of $E_m$ (the non-commutative conditional expectation on ${\mathcal{F}_{\mathbb{Z}^r,m}^{\rm{tail}}}(X^m)$). We will then use this to upper-bound the free mixing coefficients.

\noindent We remark that as $(X^{\z,m})$ is a stationary random field then according to \cref{eq_suis} we can find $(K_{\z,m})$ a sequence of *-automorphisms of $\F_{\Z^r}(X^m)$ such that $X^{\z,m}=K_{\z,m}(X^{0,m})$ for all $\z\in \Z^r$. 
Let $K\subset \Z^r$ be an arbitrary subset and choose an element $F\in \F_{K}(X^m)$ then according to \cref{thm111} we know that: $$E_m(F)=\lim_{n\rightarrow \infty}\frac{1}{|\dbracket{n}^r|}\sum_{z\in \dbracket{n}^r} K_{\z}(F).$$ Moreover as $\alpha_m(i)\xrightarrow{i\rightarrow\infty}0$ then we know that $(X^{\mathbf{z},m})$ is an ergodic random field of matrices. Therefore by Linderstrauss pointwise ergodic theorem \cite{lindenstrauss1999pointwise}  we know that  $$\lim_{n\rightarrow \infty}\frac{1}{|\dbracket{n}^r|}\sum_{z\in \dbracket{n}^r} K_{\z}(F)=\big(\mathbb{E}(F_{i,j})\big).$$ Therefore for any element $F$ in $\F_{\Z^r}(X^m)$ the non commutative conditional expectation of $F$ is given by:
\begin{align}\label{hus40}E_m(F)=\big(\mathbb{E}(F_{i,j})\big)_{i,j\le m}.\end{align} 


\noindent Let $\lambda>0$ be a real. We first upper-bound $\aleph^{j,\lambda}[\cdot|\mathbb{Z}^d]$. In this goal, let $\z,\z'\in \mathbb{Z}^r$  and $K\subset\mathbb{Z}^r$ be such  that $$\min_{\z^*\in K }\min_{j\le d}\min\big(|\z_j-\z^*_j|,|\z'_j-\z^*_j|\big)\ge b.$$  Choose $\gamma\in \mathbb{C}$ with $\rm{Im}(\gamma)>\lambda$ and write $f:x\rightarrow \rm{Im}(\gamma)~[x-\gamma \rm{Id}]^{-1}$. We shorthand $f^K:=f\Big(\frac{1}{\sqrt{|K|}}\sum_{\z\in K}X^{\z,m}\Big)$ as well as $\overline{{f^K}}:=f^K-E_m(f^K).$ The proof will work in two stages: Firstly we show that $\Big\|E\Big(f^KX^{\z,m}\overline{f^K}X^{\z',m}f^K\Big)\Big\|_1$ is small and then we establish that $\Big\|E\Big(f^K\overline{X^{\z,m}E_m({f^K})X^{\z',m}}f^K\Big)\Big\|_1$ is also small.

\noindent Firstly we remark that by definition of the $\mathbb{L}_1$ norm we have
\begin{align*}&\|E_m(f^{K}X^{\mathbf{z},m}\overline{f^{K}}X^{\mathbf{z'},m}f^{K})\|_1
    \\&\le \sup_{\substack{a\in \mathcal{F}^{tail}_{\mathbb{Z}^r}(X^m)\\\|a\|_{\infty}\le 1}}\tau\Big(af^{K}X^{\mathbf{z},m}\overline{f^{K}}X^{\mathbf{z'},m}f^{K}\Big)
\end{align*} Choose $a\in \mathcal{F}^{tail}_{\mathbb{Z}^r}(X^m)$ such that $\|a\|_{\infty}\le 1$. We note that by \cref{hus40} that $a$ is a deterministic matrix. Let $(\tilde X^{\mathbf{z},m},\tilde X^{\mathbf{z}',m})$ be a copy of $( X^{\mathbf{z},m}, X^{\mathbf{z}',m})$ that is independent from $(X^{\mathbf{z},m})$. Then by definition of the strong-mixing coefficients we have:
\begin{equation}\begin{split}\label{hpp_19}&\Big|\tau\Big(af^{K}{X^{\z,m}}{}\overline{f^{K}}{X^{\z',m}}{}f^{K}\Big)-\tau\Big(af^{K}{\tilde X^{\z,m}}{}\overline{f^{K}} {\tilde X^{\z',m}}{} f^{K}\Big)\Big|
\\&\overset{(a)}{\le} \Big|\frac{1}{m}\mathbb{E}\Big(\rm{Tr}(af^{K}{X^{\z,m}}{}\overline{f^{K}}{X^{\z',m}}{}f^{K})\Big)-\frac{1}{m}\mathbb{E}\Big(\rm{Tr}\big(af^{K}{\tilde X^{\z,m}}{}\overline{f^{K}} {\tilde X^{\z',m}}{} f^{K}\big)\Big)\Big|
\\&\overset{(b)}{\le} \frac{8}{m}\alpha_m[b]^{\frac{\epsilon}{2+\epsilon}}\|\sqrt{\rm{Tr}((X^{\z,m})^2)}\|_{L_{2+\epsilon}}\|\sqrt{\rm{Tr}((X^{\z',m})^2)}\|_{L_{2+\epsilon}}
\\&\le \frac{8}{m}\alpha_m[b]^{\frac{\epsilon}{2+\epsilon}}\|\sqrt{\rm{Tr}((X^{\mathbf{1},m})^2)}\|^2_{L_{2+\epsilon}}
\\&\overset{(c)}{\le}  \frac{8}{m}\alpha_m[b]^{\frac{\epsilon}{2+\epsilon}}\|\sqrt{\sum_{s\le N_m}Tr((A^{s,m})^2)Z^2_{s,\mathbf{1}}}\|^2_{L_{2+\epsilon}}
\\&\overset{(d)}{\le} \frac{8}{m}\alpha_m[b]^{\frac{\epsilon}{2+\epsilon}}\big[\big(\sum_{s\le N_m}\rm{Tr}((A^{s,m})^2) \big)^{1+\frac{\epsilon}{2}}\|Z\|_{2+\epsilon}^{2+\epsilon}\big]^{\frac{2}{2+\epsilon}}
\\&\le 8\|Z\|_{2+\epsilon}^2\alpha_m[b]^{\frac{\epsilon}{2+\epsilon}} \big(\sum_{s\le N_m}\|A^{s,m}\|_2^2\big)
\end{split}\end{equation} 
where $Z$ designates a standard normal random variable and where (a) is a consequence of the equation $\tau_m(\cdot)=\frac{1}{m}\mathbb{E}(\rm{Tr}(\cdot))$; and where to get (b) we used \cref{chanson} and the fact that for any two matrices $A,B\in M_m(\mathbb{C})$ we have $|\frac{1}{m}\rm{Tr}(f^Kaf^KA\overline{f^K}B)|\le \|A\|_2\|B\|_2$. To obtain (c) we used the fact that the matrices $(A^{s,m})$ are orthogonal and where finally to obtain (d) we used the fact that the following inequality holds by Jensen inequality 
\begin{align*}&
    \big(\sum_{s\le N_m}\rm{Tr}((A^{s,m})^2) Z_{s,\mathbf{1}}^2\big)^{1+\frac{\epsilon}{2}}
    \\&\le \big(\sum_{s\le N_m}\rm{Tr}((A^{s,m})^2) \big)^{\frac{\epsilon}{2}}\big(\sum_{s\le N_m}\rm{Tr}((A^{s,m})^2) Z_{s,\mathbf{1}}^{2+\epsilon}\big)
\end{align*}
Moreover if we write $\rho^s_{\mathbf{z},\mathbf{z'}}=\rm{Cov}\big(Z_{s,\mathbf{z}},~Z_{s,\mathbf{z'}}\big)$ we observe that \begin{equation*}\begin{split}&\Big|\tau\Big(af^{K}{\tilde X^{\z,m}}\overline{f^{K}} {\tilde X^{\z',m}} f^{K}\Big)\Big|
\\&\le \Big|\tau\Big(f^{K}af^{K}{\tilde X^{\z,m}}\overline{f^{K}} {\tilde X^{\z',m}} \Big)\Big|
\\&\le\Big| \frac{1}{m}\sum_{i_{1:4}\le m}\mathbb{E}\Big((f^{K}af^{K})_{i_1,i_2}\tilde X^{\z,m}_{i_2,i_3}\overline f^{K}_{i_3,i_4}\tilde X^{\z',m}_{i_4,i_1} \Big)\Big|
\\&\overset{(a)}{\le} \Big|\frac{1}{m}\sum_{i_{1:4}\le m}\mathbb{E}\Big( X^{\z,m}_{i_2,i_3}X^{\z',m}_{i_4,i_1}\Big)\mathbb{E}\Big((f^{K}af^{K})_{i_1,i_2}\overline f^{K}_{i_3,i_4} \Big)\Big|
\\&\overset{(b)}{\le} \frac{1}{m}\Big|\sum_{i_{1:4}\le m}\mathbb{E}\Big( X^{\z,m}_{i_2,i_3}X^{\z',m}_{i_4,i_1}\Big)\rm{cov}\Big(f^{K}_{i_3,i_4},(f^{K}af^{K})_{i_1,i_2} \Big)\Big|
\\&\overset{}{\le} \frac{1}{m}\Big|\sum_{i_{1:4}\le m}\sum_{s\le N_m} A^{s,m}_{i_2,i_3}A^{s,m}_{i_4,i_1}\rho^s_{\mathbf{z},\mathbf{z'}}\rm{cov}\Big(f^{K}_{i_3,i_4},(f^{K}af^{K})_{i_1,i_2} \Big)\Big|
\end{split}\end{equation*} where to obtain (a) we used the independence of $(\tilde X^{{\z},m})$ and  $( X^{{\z},m})$, and where (b) is consequence from the fact that by definition of $\overline{f^K}$ we have
\begin{align*}&\mathbb{E}\Big((f^{K}af^{K})_{i_1,i_2}\overline f^{K}_{i_3,i_4} \Big)
\\&=\mathbb{E}\Big((f^{K}af^{K})_{i_1,i_2}f^{K}_{i_3,i_4} \Big)-\mathbb{E}\Big((f^{K}af^{K})_{i_1,i_2}\Big)\mathbb{E}\Big(f^{K}_{i_3,i_4} \Big)
   \\&= \rm{cov}\Big(f^{K}_{i_3,i_4},(f^{K}af^{K})_{i_1,i_2} \Big).
\end{align*}

\noindent To bound  $\rm{cov}\Big(f^{K}_{i_3,i_4},(f^{K}af^{K})_{i_1,i_2} \Big)$ we will exploit the heat equation. To do so we let $(Z_i^K)$ be a sequence of Gaussian vectors with diagonal variance-covariance matrix $\Sigma^K$ given by $\Sigma^K_{j,j}=\sigma_{K,j}^2$ where $\sigma_{K,j}^2=\rm{var}\big(\frac{1}{\sqrt{|K|}}\sum_{\mathbf{z}\in K}Z_{j,\mathbf{z}}\big).$ We remark that 
$$\frac{1}{\sqrt{|K|}}\sum_{\mathbf{z}\in K}X^{\mathbf{z},m}\overset{d}{=}\sum_{s\le N_m} Z_{s}^KA^{s,m}.$$ 
Moreover we define $$g_1:x\rightarrow [\sum_{s\le N_m} x_s A^{s,m}-\gamma Id]^{-1}$$ and $$g_2:x\rightarrow([\sum_{s\le N_m} x_s A^{s,m}-\gamma Id]^{-1}a[\sum_{s\le N_m} x_s A^{s,m}-\gamma Id]^{-1}.$$ We remark that the function $g_1$ and $g_2$ are differentiable. Indeed for all $x,x'\in \mathbb{R}^{N}$ and all $\epsilon>0$ we have \begin{align*}&[\sum_{s\le N_m} x_s A^{s,m}-\gamma Id]^{-1}-[\sum_{s\le N_m} (x_s+\epsilon x_s')A^{s,m}-\gamma Id]^{-1}\\&=\epsilon[\sum_{s\le N_m} x_s A^{s,m}-\gamma Id]^{-1} \sum_{s\le N_m} x_s'A^{s,m}[\sum_{s\le N_m}(x_s+\epsilon x'_s)A^{s,m}-\gamma Id]^{-1}\end{align*} and \begin{align*}&[\sum_{s\le N_m} x_s A^{s,m}-\gamma Id]^{-1}a[\sum_{s\le N_m} x_s A^{s,m}-\gamma Id]^{-1}-[\sum_{s\le N_m} (x_s+\epsilon x'_s)A^{s,m}-\gamma Id]^{-1}\\&\qquad \qquad a\times[\sum_{s\le N_m} (x_s+\epsilon x'_s)A^{s,m}-\gamma Id]^{-1}\\&=\epsilon [\sum_{s\le N_m} x_s A^{s,m}-\gamma Id]^{-1}\sum_{s\le N_m} x'_sA^{s,m}[\sum_{s\le N_m} (x_s+\epsilon x'_s)A^{s,m}-\gamma Id]^{-1}\\&\qquad \qquad \times a[\sum_{s\le N_m} x_s A^{s,m}-\gamma Id]^{-1}
\\&+\epsilon [\sum_{s\le N_m} (x_s+\epsilon x'_s)A^{s,m}-\gamma Id]^{-1}a[\sum_{s\le N_m} x_s A^{s,m}-\gamma Id]^{-1}\sum_{s\le N_m} x'_sA^{s,m}\\&\qquad [\sum_{s\le N_m} (x_s+\epsilon x'_s)A^{s,m}-\gamma Id]^{-1}.\end{align*}Therefore the gradient of $g_1$ is given by $$\nabla_{i} g_1(x)=g_1(x)A^{i,m}g_1(x).$$ Similarly we remark that the gradient of $g_2$ is given by $$\nabla_{i} g_2(x)=g_1(x)A^{i,m} g_2(x)+g_2(x)A^{i,m}g_1(x) .$$
We define $Z_{i,\mathbf{z}}(t)= tZ_{i,\mathbf{z}}+\sqrt{1-t^2}\tilde Z_{i,\mathbf{z}}$  where $(\tilde Z_{i,\mathbf{z}})$ is an independent copy of $(Z_{i,\mathbf{z}})$. For any subset $K\subset\mathbb{Z}^r$ we write
$$Z_{s}^K(t)=\frac{1}{\sqrt{|K|}}\sum_{\mathbf{z}\in K}Z_{s,\mathbf{z}}(t).$$ Using the heat equation \cite{ledoux2001concentration}[section 5.5] we know that\begin{align*}
 &\rm{cov}\Big(f^{K}_{i_3,i_4},(f^Kaf^{K})_{i_1,i_2}\Big)
 \\&=\rm{Im}(\gamma)^3\int_0^1\mathbb{E}\big(\big<\nabla g_1(Z^K)_{i_3,i_4},\Sigma^K~ g_2(Z^K(t))_{i_1,i_2}\big>\big) dt
  \\&=\rm{Im}(\gamma)^3\int_0^1\sum_{s'\le N_m}\sigma_{K,s'}^2\mathbb{E}\Big(\nabla_{s'}g_1(Z^K)_{i_3,i_4}\nabla_{s'}g_2(Z^K(t))_{i_1,i_2}\Big)dt
    \\&=\rm{Im}(\gamma)^3\int_0^1\sum_{s'\le N_m}\sigma_{K,s'}^2\mathbb{E}\Big(\big(g_1(Z^K)A^{s',m}g_1(Z^K)\big)_{i_3,i_4}\big(g_1(Z^K(t))A^{s',m}g_2(Z^K(t))\big)_{i_1,i_2}\Big)dt
    \\&\quad+ \rm{Im}(\gamma)^3\int_0^1\sum_{s'\le N_m}\sigma_{K,s'}^2\mathbb{E}\Big(\big(g_1(Z^K)A^{s',m}g_1(Z^K)\big)_{i_3,i_4}\big(g_2(Z^K(t))A^{s',m}g_1(Z^K(t))\big)_{i_1,i_2}\Big)dt\end{align*}
    
Therefore we directly obtain that \begin{equation*}\begin{split}&\Big|\tau\Big(af^{K}{\tilde X^{\z,m}}\overline{f^{K}} {\tilde X^{\z',m}} f^{K}\Big)\Big|
\\&\overset{}{\le} \frac{\rm{Im}(\gamma)^3}{m}\int_0^1\Big|\sum_{s,s'\le N_m}\sigma_{K,s'}^2\rho_{\mathbf{z},\mathbf{z'}}^s\mathbb{E}\Big(\rm{Tr}\big[g_1(Z^K(t))A^{s',m}g_2(Z^K(t)) A^{s,m}g_1(Z^K)A^{s',m}g_1(Z^K) A^{s,m}\big]\Big)\Big|dt
\\&+\frac{\rm{Im}(\gamma)^3}{m}\int_0^1\Big|\sum_{s,s'\le N_m}\sigma_{K,s'}^2\rho_{\mathbf{z},\mathbf{z'}}^s\mathbb{E}\big(\rm{Tr}\big[g_2(Z^K(t))A^{s',m}g_1(Z^K(t)) A^{s,m}g_1(Z^K)A^{s',m}g_1(Z^K) A^{s,m}\big]\Big)\Big|dt
\\&\overset{}{\le} \frac{\rm{Im}(\gamma)^3}{m}\int_0^1\mathbb{E}\Big(\Big|\rm{Tr}\big[\sum_{s,s'\le N_m}\sigma_{K,s'}^2\rho_{\mathbf{z},\mathbf{z'}}^sg_1(Z^K(t))A^{s',m}g_2(Z^K(t)) A^{s,m}g_1(Z^K)A^{s',m}g_1(Z^K) A^{s,m}\big]\Big|\Big)dt
\\&+\frac{\rm{Im}(\gamma)^3}{m}\int_0^1\mathbb{E}\big(\Big|\sum_{s,s'\le N_m}\sigma_{K,s'}^2\rho_{\mathbf{z},\mathbf{z'}}^s\rm{Tr}\big[g_2(Z^K(t))A^{s',m}g_1(Z^K(t)) A^{s,m}g_1(Z^K)A^{s',m}g_1(Z^K) A^{s,m}\big]\Big|\Big)dt
\\&\overset{(a)}{\le} 2\rm{Im}(\gamma)^{-1}\sup_{\substack{Y_1,Y_2,Y_3\\\|Y_1\|_{\infty},\|Y_2\|_{\infty},\|Y_3\|_{\infty}\le 1}}\Big\|\sum_{s,s'\le N_m}\sigma_{K,s'}^2\rho_{\z,\z'}^sA^{s,m}Y_1A^{s',m}Y_2A^{s,m}Y_3A^{s',m}\Big\|_1
\\&\overset{(b)}{\le}2\rm{Im}(\gamma)^{-1}\sup_{\substack{U_1,U_2,U_3\in U(m)}}\Big\|\sum_{s,s'\le N_m}\sigma_{K,s'}^2\rho_{\z,\z'}^sA^{s,m}U_1A^{s',m}U_2A^{s,m}U_3A^{s',m}\Big\|_1
\end{split}\end{equation*}where in (a) we remarked that $\|g_1\|_{\infty},\|g_2\|_{\infty}\le \rm{Im}(\gamma)^{-1}$ and where to obtain (b) we used the fact that every matrix $\|Y\|_{\infty}\le 1$ can be written as convex combination of unitary matrix.
We then remark that for any choice of $U_1,U_2,U_3$ we have
\begin{align*}&
    \Big\|\sum_{s,s'\le N_m}\sigma_{K,s'}^2\rho_{\z,\z'}^sA^{s,m}U_1A^{s',m}U_2A^{s,m}U_3A^{s',m}\Big\|_1
    \\&=\Big\|\mathbb{E}\Big(\tilde X^{\z,m}U_1S_KU_2\tilde X^{\z',m}U_3S_K\Big)\Big\|_1
\end{align*}

Therefore the value of $\Big\|\sum_{s,s'\le N_m}\sigma_{K,s'}^2\rho_{\z,\z'}^sA^{s,m}Y_1A^{s',m}Y_2A^{s,m}Y_3A^{s',m}\Big\|_1$ only depends on the distribution of $(\tilde X^{\z,m},\tilde X^{\z',m})$ and $S_K$. We remark that $S_K$ can be seen as a $m^2$ dimensional Gaussian vector with the variance-covariance matrix $\Sigma$. It follows that if we write $(C_i)$ the (unnormalized) orthogonal eigenvectors of $\Sigma$ and if $(Z_i)$ is a sequence of i.i.d standard normal random variables then we observe that $S_K\overset{d}{=}\sum_{i\le m^2}Z_iC_i$.
This directly implies that for all \begin{align*}&
    \Big\|\sum_{s,s'\le N_m}\sigma_{K,s'}^2\rho_{\z,\z'}^sA^{s,m}U_1A^{s',m}U_2A^{s,m}U_3A^{s',m}\Big\|_1
    \\&=\Big\|\mathbb{E}\Big(\tilde X^{\z,m}U_1S_KU_2\tilde X^{\z',m}U_3S_K\Big)\Big\|_1
    \\&= \Big\|\sum_{s\le N_m,s'\le m^2}\rho_{\z,\z'}^sA^{s,m}U_1C_{s'}U_2A^{s,m}U_3C_{s'}\Big\|_1.
\end{align*}
Now note that by definition we have $Tr(C_iC_j)=0$ if $i\ne j$ are different. Therefore we can note that for any matrix $Y$ we have $\sum_{i\le m^2}\frac{1}{\|C_i\|_{HS}^2}C_i^*YC_i=Tr(Y)Id$ (see lemma 4.8 in \cite{bandeira2021matrix}). Moreover we remark that \begin{align*}&
    \Big\|\sum_{s\le N_m,s'\le m^2}\rho_{\z,\z'}^sA^{s,m}U_1C_{s'}U_2A^{s,m}U_3C_{s'}\Big\|_1
    \\&\le \sup_{\|x\|,\|y\|\le 1}\Big|\sum_{s\le N_m,s'\le m^2}\rho_{\z,\z'}^s\Big<C_{s'}^*x,~U_3^*A^{s,m}U_2^*C_{s'}^*U_1^*A^{s,m}y\Big>\Big|
    \\&\le\sup_{\|x\|,\|y\|\le 1} \Big(\sum_{ s'\le m^2}\|C^*_{s'}x\|^2\Big)^{1/2}\Big(\sum_{ s'\le m^2}\Big\|\sum_{s\le N_m}\rho_{\z,\z'}^s~U_3^*A^{s,m}U_2^*C_{s'}^*U_1^*A^{s,m}y\Big\|^2\Big)^{1/2}.
\end{align*} 
Note that we have\begin{align*}&
    \sum_{s'\le m^2}\Big\|\sum_{s\le N_m}\rho_{\z,\z'}^s~U_3^*A^{s,m}U_2^*C_{s'}^*U_1^*A^{s,m}y\Big\|^2
    \\&=    \sum_{s'\le m^2}\sum_{s_1,s_2\le N_m}\rho_{\z,\z'}^{s_1}\rho_{\z,\z'}^{s_2}y^*U_3^*A^{s_1,m}U^*_2C_{s'}^*U_1^*A^{s_1,m}A^{s_2,m}U_1C_{s'}U_2A^{s_2,m}U_3y
    \\&\overset{}{\le } \max_i\|C_i\|_{HS}^2 \Big|\sum_{s_1,s_2\le N_m}\rho_{\z,\z'}^{s_1}\rho_{\z,\z'}^{s_2}y^*U_3^*A^{s_1,m}A^{s_2,m}U_3y Tr\Big(A^{s_1,m}A^{s_2,m}\Big)\Big|
       \\&\overset{(a)}{=} \max_i\|C_i\|_{HS}^2 \Big|\sum_{s \le N_m}(\rho_{\z,\z'}^{s})^2y^*U_3^*A^{s,m}A^{s,m}U_3y Tr\Big((A^{s,m})^2\Big)\Big|
\end{align*}

where to obtain $(b)$ we used the fact that the matrices $(A^{s,m})$ are orthogonal. Therefore we obtain that \begin{align*}&
\sup_{\|y\|\le 1} \Big|\sum_{s\le N_m}(\rho_{\z,\z'}^{s})^2y^*U_3^*A^{s,m}A^{s,m}U_3y Tr\Big((A^{s,m})^2\Big)\Big|
     \\&=\sup_{\|y\|\le 1} \sum_{s\le N_m}(\rho_{\z,\z'}^{s})^2y^*(A^{s,m})^2y Tr\Big((A^{s,m})^2\Big)
     \\&\le \sup_{s\le N_m}Tr((A^{s,m})^2)\sum_{s\le N_m}(\rho_{\z,\z'}^{s})^2\|A^{s,m}y\|_2^2 
     \\&\le\sup_{s\le N_m}Tr((A^{s,m})^2)\sum_{s\le N_m}\|A^{s,m}y\|_2^2 
     \\&\le \mathcal{V}_m(X^m)^2\sigma_m(X^m)^2.
\end{align*} where to get the last inequality we used the fact that by definition we have $$\mathcal{V}_m(X^m)^2=\sup_{Tr(|M|^2)\le 1}\sum_{i\le N_m}Tr(A^{i,m} M)^2,$$ and the fact that if we define $M_s=\frac{A^{s,m}}{\sqrt{Tr((A^{s,m})^2)}}$ then we remark that $Tr(M_s^2)=1$ and $\sum_{j\le N_m}Tr(A^{j,m}M_s)^2=Tr((A^{s,m})^2)$.  Moreover, we remark that we also have:

\begin{align*}\max_{i\le m^2}\|C_i\|_{\rm{HS}}^2&\le \sup_{Tr(|M|^2)\le 1}\sum_{s\le N_m}\sigma_{K,s}^2\big|Tr(A^{s,m}M)\big|^2\\&\le\max_{s\le N_m} \sigma_{K,s}^2\sup_{Tr(|M|^2)\le 1}\sum_{s\le N_m}\big|Tr(A^{s,m}M)\big|^2\\&= \max_{s\le N_m}\sigma_{K,s}^2\mathcal{V}_m(X^m)^2.\end{align*} Moreover according to \cref{chanson} for all $s\le N_m$ we have
  \begin{align*} \sigma_{K,s}^2\lesssim  \|Z\|_{2+\epsilon}^2\sum_{b\le \infty}rb^{r-1}\alpha_m[b]^{\frac{\epsilon}{2+\epsilon}}
\end{align*} 
This directly implies that\begin{align*}\max_{i\le m^2}\|C_i\|_{\rm{HS}}^2&\lesssim \mathcal{V}_m(X^m)^2\|Z\|_{2+\epsilon}^2\sum_{b\le \infty}rb^{r-1}\alpha_m[b]^{\frac{\epsilon}{2+\epsilon}}.\end{align*} 

\noindent Finally we also remark that  \begin{align*}
  \sup_{\|x\|\le 1}  \sum_{s'\le m^2}\|C_{s'}x\|^2&=\|E(S_K^2)\|_{\infty}
  \\&=\|\sum_{s\le N_m}\sigma_{K,s}^2 (A^{k,s})^2\|_{\infty}\lesssim \max_{s\le N_m}\sigma_{K,s}^2\|\sum_{s\le N_m} (A^{k,s})^2\|_{\infty}.\end{align*} 
  
 \noindent Moreover according to \cref{chanson} for all $s\le N_m$ we have
  \begin{align*} \sigma_{K,s}^2\lesssim  \|Z\|_{2+\epsilon}^2\sum_{b\le \infty}rb^{r-1}\alpha_m[b]^{\frac{\epsilon}{2+\epsilon}}
\end{align*}

\noindent Therefore we obtain that\begin{align*}
  \sup_{\|x\|\le 1}  \sum_{s'\le m^2}\|C_{s'}x\|^2\lesssim \|Z\|_{2+\epsilon}^2\sum_{b\lesssim \infty}rb^{r-1}\alpha_m[b]^{\frac{\epsilon}{2+\epsilon}} \sigma_m(X^m)^2.\end{align*} 
  Therefore we obtain that
\begin{align*}&
  \Big|\tau\Big(af^{K}{\tilde X^{\z,m}}\overline{f^{K}} {\tilde X^{\z',m}} f^{K}\Big)\Big|\\&\lesssim \sum_{b\ge 0}b^{r-1}\alpha_m[b]^{\frac{\epsilon}{2+\epsilon}}\sigma_m(X^m)^2\mathcal{V}_m(X^m)^2.
\end{align*}

Combined with \cref{hpp_19} we obtain that 
\begin{align}&
  \Big|\tau\Big(af^{K}{X^{\z,m}}\overline{f^{K}} { X^{\z',m}} f^{K}\Big)\Big|\\& \lesssim  \sum_{s\le N_m}\big\|A^{s,m}\big\|_2^2\alpha_m[b]^{\frac{\epsilon}{2+\epsilon}}  +\sum_{l\ge 0}l^{r-1}\alpha_m[l]^{\frac{\epsilon}{2+\epsilon}}\sigma_m(X^m)^2\mathcal{V}_m(X^m)^2.\nonumber
\end{align}

\noindent Now we bound $\Big\|E_m\Big(f^K\overline{X^{\z,m}E_m({f^K})X^{\z',m}}f^K\Big)\Big\|_1$. In this goal we write $$A^{\z,\z'}:=X^{\z,m}E_m({f^K})X^{\z',m}.$$ Firstly we remark that by definition of the $\mathbb{L}_1$ norm we have
\begin{align*}&
\|E_m\Big(f^K\overline{X^{\z,m}E_m({f^K})X^{\z',m}}f^K\Big)\|_1
    \\&\le \sup_{\substack{a\in \mathcal{F}^{tail}_{\mathbb{Z}^r}(X^m)\\\|a\|_{\infty}\le 1}}\tau\Big(af^{K}\overline{A^{\z,\z'}}f^K\Big).
\end{align*} Choose $a\in \mathcal{F}^{tail}_{\mathbb{Z}^r}(X^m)$ such that $\|a\|_{\infty}\le 1$. We note that by \cref{hus40} that $a$ is a deterministic matrix. Let $(\tilde X^{\mathbf{z},m},\tilde X^{\mathbf{z}',m})$ be a copy of $( X^{\mathbf{z},m}, X^{\mathbf{z}',m})$ that is independent from $(X^{\mathbf{z},m})$ and write $\tilde A^{\z,\z'}:=\tilde X^{\z,m}E_m({f^K})\tilde X^{\z',m}$. Then by definition of the strong-mixing coefficients and by \cref{chanson} we have:
\begin{equation}\begin{split}&\Big|\tau\Big(af^{K}A^{\z,\z'}f^K-af^{K}\tilde A^{\z,\z'}f^K\Big)\Big|
\\&\overset{(a)}{\le} \frac{1}{m}\Big|\sum_{i_{1:3}\le m} \rm{cov}\Big(A^{\z,\z'}_{i_2,i_3} ,~(af^K)_{i_{1:2}}f^K_{i_3,i_1}\Big)\Big|
\\&\overset{}{\le} \Big|\frac{1}{m}\mathbb{E}\Big(\rm{Tr}(af^{K}{X^{\z,m}}{}E_m({f^{K}}){X^{\z',m}}{}f^{K})\Big)-\frac{1}{m}\mathbb{E}\Big(\rm{Tr}\big(af^{K}{\tilde X^{\z,m}}{}E_m({f^{K}}) {\tilde X^{\z',m}}{} f^{K}\big)\Big)\Big|
\\&\overset{(b)}{\le} \frac{8}{m}\alpha_m[b]^{\frac{\epsilon}{2+\epsilon}}\|\sqrt{\rm{Tr}((X^{\z,m})^2)}\|_{L_{2+\epsilon}}\|\sqrt{\rm{Tr}((X^{\z',m})^2)}\|_{L_{2+\epsilon}}
\\&\le \frac{8}{m}\alpha_m[b]^{\frac{\epsilon}{2+\epsilon}}\|\sqrt{\rm{Tr}((X^{\mathbf{1},m})^2)}\|^2_{L_{2+\epsilon}}
\\&\overset{}{\le}  \frac{8}{m}\alpha_m[b]^{\frac{\epsilon}{2+\epsilon}}\|\sqrt{\sum_{s\le N_m}Tr((A^{s,m})^2)Z^2_{s,\mathbf{1}}}\|^2_{L_{2+\epsilon}}
\\&\overset{}{\le}   8\|Z\|_{2+\epsilon}^2 \alpha_m[b]^{\frac{\epsilon}{2+\epsilon}} \big(\sum_{s\le N_m}\|A^{s,m}\|_2^2\big)
\end{split}\end{equation} 
where $Z$ designates a standard normal random variable and where (a) is a consequence of the equation $\tau_m(\cdot)=\frac{1}{m}\mathbb{E}(\rm{Tr}(\cdot))$; and where to get (b) we used \cref{chanson} and the fact that $\|af^{K}\|_{\infty},\|f^{K}\|_{\infty}\le 1$. 
Moreover, we remark that
\begin{align*}&
    \Big|\tau\Big(af^{K}A^{\z,\z'}f^K-af^{K}\tilde A^{\z,\z'}f^K\Big)\Big|
    \\&=\tau\Big(af^{K}\overline{A^{\z,\z'}}f^K\Big)
\end{align*}
\noindent Finally we remark that\begin{align*}
\|X^{1,m}\|_2^2&=\frac{1}{m}\sum_{i,j\le m}\mathbb{E}((X^{1,m}_{i,j})^2)
\\&=\frac{1}{m}\sum_{s\le N_m}\sum_{i,j\le m}(A^{s,m}_{i,j})^2
\\&=\sum_{s\le N_m}\big\|A^{s,m}\big\|_2^2.
\end{align*}
This directly implies that\begin{align*}&
   \aleph_m^{j,\lambda}[b|\mathbb{Z}^r]\lesssim\|Z\|_{2+\epsilon}^2\alpha_m[b]^{\frac{\epsilon}{2+\epsilon}}  +\frac{1}{\sum_{s\le N_m}\|A^{s,m}\|^2_{2}}\sum_{l\ge 0}l^{r-1}\alpha_m[l]^{\frac{\epsilon}{2+\epsilon}}\sigma_m(X^m)^2\mathcal{V}_m(X^m)^2.
\end{align*}
Therefore we remark that there is a constant $C$ that is independent from $m$ and $b$ such that:
$$\aleph^{j,\lambda}_m[b|\Z^r]\le C\Big[\alpha_m[b]^{\frac{\epsilon}{2+\epsilon}}+\frac{\sigma_m(X^m)^2\mathcal{V}_m(X^m)^2}{\sum_{s\le N_m}\big\|A^{s,m}\big\|_2^2}\sum_{l\ge 0}l^{r-1}\alpha_m[l]^{\frac{\epsilon}{2+\epsilon}}\Big].$$
\noindent We now move on to bounding $\aleph^s_m[\cdot|\mathbb{Z}^r]$. Let $\z,\z'\in \mathbb{Z}^r$  and $K\subset\mathbb{Z}^r$ be such  that $$\min_{\z^*\in K\cup\{\z'\} }d(\z^*,\z)\ge b.$$  Choose $Y_1,Y_2,Y_3\in \mathcal{F}^{\rm{tail}}_{\Z^r}(X^m)$ be operators satisfying $\|Y_1\|_{\infty},\|Y_2\|_{\infty},\|Y_3\|_{\infty}\le 1$. Firstly we remark that by definition of the $\mathbb{L}_1$ norm we have
\begin{align*}&
    \|E_m(Y_1X^{\mathbf{z},m}Y_2X^{\mathbf{z'},m}Y_3)\|_1
    \\&\le \sup_{\substack{a\in \mathcal{F}^{tail}_{\mathbb{Z}^r}(X^m)\\\|a\|_{\infty}\le 1}}\tau\Big(aY_1X^{\mathbf{z},m}Y_2X^{\mathbf{z'},m}Y_3\Big)
\end{align*} Choose $a\in \mathcal{F}^{tail}_{\mathbb{Z}^r}(X^m)$ such that $\|a\|_{\infty}\le 1$. We note that by \cref{hus40} that $a$ is a deterministic matrix.

\noindent Let $\tilde X^{\mathbf{z},m}$ be a copy of $ X^{\mathbf{z},m}$ that is independent from $(X^{\mathbf{z},m})$. Then by definition of the strong-mixing coefficients and \cref{chanson} we have:
\begin{equation*}\begin{split}&\Big|\tau\Big(aY_1{X^{\z,m}}{}Y_2{X^{\z',m}}{}Y_3\Big)-\tau\Big(aY_1{\tilde X^{\z,m}}{}Y_2{ X^{\z',m}}{}Y_3\Big)\Big|
\\&\overset{}{\le} \frac{1}{m}\Big|\sum_{i_{1:4}\le m} \rm{cov}\Big(X^{\z,m}_{i_2,i_3} ,~X^{\z',m}_{i_4,i_1}(Y_3aY_1)_{i_1,i_2}(Y_2)_{i_3,i_4}\Big)\Big|
\\&\overset{}{\lesssim} \alpha_m[b]^{\frac{\epsilon}{2+\epsilon}}\|Z\|_{2+\epsilon}^2\sum_{s\le N_m}\Big\|A^{s,m}\Big\|_2^2
\end{split}\end{equation*}

Moreover we remark that $\tau\Big(aY_1{\tilde X^{\z,m}}{}Y_2 { X^{\z',m}}{} Y_3\Big)=0$. This directly implies that \begin{equation*}\begin{split}&\Big|E_m\Big(Y_1{X^{\z,m}}{}Y_2{X^{\z',m}}{}Y_3\Big)\Big|
\\&\lesssim \alpha_m[b]^{\frac{\epsilon}{2+\epsilon}}\|Z\|_{2+\epsilon}^2\sum_{s\le N_m}\Big\|A^{s,m}\Big\|_2^2.
\end{split}\end{equation*}
\end{proof}

\section{Proof of \cref{good_3}}
\begin{proof} 
 We remark that we can suppose without loss of generality that the entries of the random matrices $(X^{i,m})$ have a conditional mean of $0$ and do so in the following. 
 
\noindent We remark that for all jointly invariant random matrices $X$, $E_{m}(X)$ is a $m\times m$ random matrix whose diagonal entries are all equal to $\mathbb{E}\Big(X_{1,1}|\mathbb{S}(\mathbb{N})\Big)$ and whose non-diagonal entries are all equal to $\mathbb{E}\big(X_{1,2}|\mathbb{S}(\mathbb{N})\big)$. For ease of notation, we will write $F_m$ the set of all $m\times m$ matrices that are $\sigma(\mathbb{S}(\mathbb{N}))$ measurable and have respectively all the diagonal terms and non-diagonal terms equal almost surely.

We start by bounding $\aleph_m^s[\cdot|\mathbb{Z}]$. Let $K\subset \mathbb{N}\setminus \{1,2\}$ and choose $Y_1,Y_2,Y_3\in \mathcal{F}^K(X^m)$ be such that $\|Y_1\|_{\infty},\|Y_2\|_{\infty},\|Y_3\|_{\infty}\le 1$. By definition of the non-commutative conditional expectation we know that \begin{equation*}
    \begin{split}&
\tau\Big(\Big|E_m\Big(Y_1X^{1,m}Y_2X^{2,m}Y_3\Big)\Big|\Big)
\\&\le \sup_{\substack{a\in F_m\\ \|a\|_{\infty}\le 1}}\tau\Big(aY_1X^{1,m}Y_2X^{2,m}Y_3\Big)
    \end{split}
\end{equation*}Let $a\in F_m$ be such that $\|a\|_{\infty}\le 1$. By definition of $F_m$ we know that $a$ is $\sigma(\mathbb{S}(\mathbb{N}))$ measurable. Using the conditional independence of $X^{1,m}$ and $(X^{2,m},Y_1,Y_2,Y_3)$ we obtain that:
\begin{equation*}
    \begin{split}&       \tau\Big(aY_1X^{1,m}Y_2X^{2,m}Y_3\Big)
          =  \tau\Big(aY_1E_m(X^{1,m})Y_2X^{2,m}Y_3\Big)=0.
    \end{split}
\end{equation*}
Therefore we obtain that $\aleph^s_n[1|\Z]=0$. For all $\lambda>0$, we now want to bound $\aleph^{j,\lambda}_m[\cdot|\Z]$. In this goal let $z_1,z_2\in \Z$ and let $K\subset \mathbb{Z}\setminus\{z_1,z_2\}$.  Choose $\gamma\in \mathbb{C}$ with $\rm{Im}(\gamma)>\lambda$ and write $f:x\rightarrow \rm{Im}(\gamma)~[x-\gamma \rm{Id}]^{-1}$. We shorthand $f^K:=f\Big(\frac{1}{\sqrt{|K|}}\sum_{z\in K}X^{z,m}\Big)$ as well as $\overline{{f^K}}:=f^K-E_m(f^K).$  We first prove that $\|E_m(f^KX^{z_1,m}\overline{f^K}X^{z_2,m}f^K)\|_1$ is small. In this goal, let $a\in F_m$ be such that $\|a\|_{\infty}\le 1$ then we have

\begin{equation*}
    \begin{split}&
\tau\Big(a f^K\frac{ X^{z_1,m}}{\sqrt{m}}\overline{f^K}\frac{X^{z_2,m}}{\sqrt{m}}f^K\Big)
        \\&=\frac{1}{m^2} \sum_{i_{1:4}\le n }\mathbb{E}\Big[\mathbb{E}\Big((f^kaf^K)_{i_1,i_2}X^{z_1,m}_{i_2,i_3}(\overline{f^K})_{i_3,i_4}X^{z_2,m}_{i_4,i_1}\Big|\mathbb{S}(\mathbb{N})\Big)\Big]
                \\&=\frac{1}{m^2}\sum_{i_{1:4}\le n } \mathbb{E}\Big[\mathbb{E}\Big(X^{z_1,m}_{i_2,i_3}X^{z_2,m}_{i_4,i_1}\Big|\mathbb{S}(\mathbb{N})\Big)\mathbb{E}\Big((f^Kaf^K)_{i_1,i_2}(\overline{f^K})_{i_3,i_4}\Big|\mathbb{S}(\mathbb{N})\Big)\Big]
         \\&\overset{(a)}{\le}\frac{2}{m^2}\sum_{i_{1:3}\le m} \mathbb{E}\Big[\mathbb{E}\Big(X^{z_1,m}_{i_2,i_3}X^{z_2,m}_{i_2,i_1}\Big|\mathbb{S}(\mathbb{N})\Big)\mathbb{E}\Big((f^Kaf^K)_{i_1,i_2}(\overline{f^K})_{i_3,i_2}\Big|\mathbb{S}(\mathbb{N})\Big)\Big]
         \\&+ \frac{1}{m^2}\sum_{i_{2:4}\le m } \mathbb{E}\Big[\mathbb{E}\Big(X^{z_1,m}_{i_2,i_3}X^{z_2,m}_{i_4,i_2}\Big|\mathbb{S}(\mathbb{N})\Big)\mathbb{E}\Big((f^Kaf^K)_{i_2,i_2}(\overline{f^K})_{i_3,i_4}\Big|\mathbb{S}(\mathbb{N})\Big)\Big|\Big]
          \\&+ \frac{1}{m^2}\sum_{i_{2:4}\le m } \mathbb{E}\Big[\mathbb{E}\Big(X^{z_1,m}_{i_2,i_3}X^{z_2,m}_{i_4,i_2}\Big|\mathbb{S}(\mathbb{N})\Big)\mathbb{E}\Big((f^Kaf^K)_{i_3,i_4}(\overline{f^K})_{i_2,i_2}\Big|\mathbb{S}(\mathbb{N})\Big)\Big|\Big]
    \end{split}
\end{equation*}where  (a) comes from the conditional independence of $X_{i,l}^{z_1,m}$ and $X^{z_2,m}_{k,j}$ when the indexes are distinct $\{k,j\}\bigcap \{l,i\}=\emptyset$ (which is a consequence of the joint exchangeability of the entries). Firstly we can remark that as $\|f^Kaf^K\|_{\infty}, \|f^K\|_{\infty}\le 1$ then we have
\begin{align*}&
  \frac{1}{m^2}\sum_{i_{1:3}\le m} \mathbb{E}\Big[\mathbb{E}\Big(X^{z_1,m}_{i_2,i_3}X^{z_2,m}_{i_2,i_1}\Big|\mathbb{S}(\mathbb{N})\Big)\mathbb{E}\Big((f^Kaf^K)_{i_1,i_2}(\overline{f^K})_{i_3,i_2}\Big|\mathbb{S}(\mathbb{N})\Big)\Big]
  \\&\lesssim \frac{1}{m}\sup_{i,j\le m}\|X_{i,j}^{1,m}\|_2^2.
\end{align*}
Moreover we remark that the following decomposition holds:
\begin{align*}&
    \frac{1}{m^2}\sum_{i_{2:4}\le m } \mathbb{E}\Big[\mathbb{E}\Big(X^{z_1,m}_{i_2,i_3}X^{z_2,m}_{i_4,i_2}\Big|\mathbb{S}(\mathbb{N})\Big)\mathbb{E}\Big((f^Kaf^K)_{i_3,i_4}(\overline{f^K})_{i_2,i_2}\Big|\mathbb{S}(\mathbb{N})\Big)\Big|\Big]
    \\&= \frac{1}{m^2}\sum_{i_{2}\le m }\sum_{i_3,i_4\ne i_2} \mathbb{E}\Big[\mathbb{E}\Big(X^{z_1,m}_{i_2,i_3}X^{z_2,m}_{i_4,i_2}\Big|\mathbb{S}(\mathbb{N})\Big)\mathbb{E}\Big((f^Kaf^K)_{i_3,i_4}(\overline{f^K})_{i_2,i_2}\Big|\mathbb{S}(\mathbb{N})\Big)\Big|\Big]
    \\&+ \frac{1}{m^2}\sum_{i_{2}\le m }\sum_{\substack{i_3,i_4\le m\\i_3= i_2~\rm{or}~i_4=i_2}} \mathbb{E}\Big[\mathbb{E}\Big(X^{z_1,m}_{i_2,i_3}X^{z_2,m}_{i_4,i_2}\Big|\mathbb{S}(\mathbb{N})\Big)\mathbb{E}\Big((f^Kaf^K)_{i_3,i_4}(\overline{f^K})_{i_2,i_2}\Big|\mathbb{S}(\mathbb{N})\Big)\Big|\Big].
\end{align*}
We will bound each term successively. Firstly we remark that 
\begin{align*}&
  \sum_{i_2\le m}\sum_{i_3,i_4\ne i_2} \mathbb{E}\Big[\mathbb{E}\Big(X^{1,m}_{i_2,i_3}X^{1,m}_{i_4,i_2}\Big|\mathbb{S}(\mathbb{N})\Big)\mathbb{E}\Big((f^Kaf^K)_{i_3,i_4}(\overline{f^K})_{i_2,i_2}\Big|\mathbb{S}(\mathbb{N})\Big)\Big]
  \\&= \mathbb{E}\Big[\mathbb{E}\Big(X^{1,m}_{1,2}X^{1,m}_{1,3}\Big|\mathbb{S}(\mathbb{N})\Big)  \rm{cov}\Big(\sum_{i_3,i_4\le m}(f^Kaf^K)_{i_3,i_4},\sum_{i_2\le m}f^K_{i_2,i_2}\Big|\mathbb{S}(\mathbb{N})\Big)\Big]
  \\&\quad-    \mathbb{E}\Big[\mathbb{E}\Big(X^{1,m}_{1,2}X^{1,m}_{1,3}\Big|\mathbb{S}(\mathbb{N})\Big)  \sum_{i_2\le m}\sum_{\substack{i_3,i_4\le m\\i_3~\rm{or}~i_4=i_2}}\rm{cov}\Big(\sum_{i_2\le m}f^K_{i_2,i_2},(f^Kaf^K)_{i_3,i_4}\Big|\mathbb{S}(\mathbb{N})\Big)\Big].
\end{align*}
We bound each term successively. We write $S_K:=\frac{1}{\sqrt{|K|m}}\sum_{z\in K} X^{z,m}$ and for all $i,j\le m$ we write $\tau_{i,j}$ the permutation that permutes $\{i;j\}$ and leaves all the other indexes invariant. We define $X^{z,m,i,j}:=(X^{z,m}_{\tau_{i,j}(l_1),\tau_{i,j}(l_2)})_{l_1,l_2\le m}$ and set $S_K^{i,j}:=\frac{1}{\sqrt{|K|m}}\sum_{z\in K} X^{z,m,i,j}.$ By Effron Stein inequality we remark that
\begin{align*}
    \mathbb{E}\Big[\rm{var}\Big(\sum_{i_2\le m}f_{i_2,i_2}^K|\mathbb{S}(\mathbb{N})\Big)\Big]&\overset{(a_1)}{\le}\sum_{j\le m}\mathbb{E}\Big[\mathbb{E}\Big(\big[\sum_{i_2\le m}f_{i_2,i_2}^K-\sum_{i_2\le m}f_{i_2,i_2}^K(S^{j,m+1}_K)\big]^2|\mathbb{S}(\mathbb{N})\Big)\Big]
    \\&\overset{(a_2)}{=}m\mathbb{E}\Big[\big[\rm{Tr}\big(f^K-f^K(S^{1,m+1}_K)\big)\big]^2\Big]
  \\  &\overset{(a_3)}{=}m\rm{Im}(\gamma)^{-1}\mathbb{E}\Big(\big[Tr(f^K \big(S_K-S_K^{1,m+1}\big) f^K(S_K^{1,m+1}))\big]^2\Big)
  \\  &\overset{(a_4)}{\le} m\rm{Im}(\gamma)^{-1}\mathbb{E}\Big(Tr\big[ \big|S_K-S_K^{1,m+1}\big|\big]^2\Big)
   \\  &\overset{(a_5)}{\le}2 m\rm{Im}(\gamma)^{-1}\mathbb{E}\Big(Tr\big[ \big(S_K-S_K^{1,m+1}\big)^2\big]\Big)
    \\&\le4m\rm{Im}(\gamma)^{-1}\mathbb{E}\Big(\sum_{i\le m}(S^K_{i,1}-S^K_{i,m+1})^2\Big)
    \\&\le    4\rm{Im}(\gamma)^{-1} \mathbb{E}\Big(\sum_{i\le m}\Big(\frac{1}{\sqrt{|K|}}\sum_{z\in K} X^{z,m}_{i,1}-X^{z,m}_{i,m+1}\Big)^2\Big)
        \\&\le     16m\rm{Im}(\gamma)^{-1} \sup_{i,j\le m}\|X^{1,m}_{i,j}\|_2^2
\end{align*}where $(a_1)$ is a consequence of the Effron-Stein inequality, where $(a_2)$ is a consequence of the exchangeability and where to get $(a_3)$ we exploited the fact that by definition $f^K$ is a resolvent.  To obtain $(a_4)$ we used the fact that that if we write $(\lambda_i^K)$ the (random) eigenvalues of $f^K$ then we have $\max_i| \lambda_i^K|\overset{a.s}{\le }1$. Finally $(a_5)$ is a consequence of the fact that the rank of $\big|S_K-S_K^{1,m+1}\big|$ is less than 2.

\noindent Therefore by Cauchy-Schwarz we obtain that 
\begin{align*}&
\Big|\mathbb{E}\Big[\mathbb{E}\Big(X^{1,m}_{1,2}X^{1,m}_{1,3}\Big|\mathbb{S}(\mathbb{N})\Big)  \rm{cov}\Big(\sum_{i_2\le m}f^K_{i_2,i_2},\sum_{i_3,i_4\le m}(f^Kaf^K)_{i_3,i_4}\Big|\mathbb{S}(\mathbb{N})\Big)\Big]\Big|
\\&\le\sup_{i,j\le m}\|X^{1,m}_{i,j}\|_2^2  \sqrt{\mathbb{E}\Big[\rm{var}\Big(\sum_{i_2\le m}f_{i_2,i_2}^K|\mathbb{S}(\mathbb{N})\Big)\Big]}\sqrt{\mathbb{E}\Big(\big[\sum_{i_3,i_4\le m}(f^Kaf^K)_{i_3,i_4}\big]^2\Big)}
        \\&\overset{(a)}{\le}     4(m)^{3/2}\sqrt{\rm{Im}(\gamma)^{-1}} \sup_{i,j\le m}\|X^{1,m}_{i,j}\|_2^3
\end{align*} where to get (a) we used the fact that $\|f^Kaf^K\|_{\infty}\le 1$.

\noindent In addition we observe that \begin{align*}&
  \Big| \mathbb{E}\Big[\mathbb{E}\Big(X^{1,m}_{1,2}X^{1,m}_{1,3}\Big|\mathbb{S}(\mathbb{N})\Big)  \sum_{i_2\le m}\sum_{\substack{i_3,i_4\le m\\i_3~\rm{or}~i_4=i_2}}\rm{cov}\Big(\sum_{i_2\le m}f^K_{i_2,i_2},(f^Kaf^K)_{i_3,i_4}\Big|\mathbb{S}(\mathbb{N})\Big)\Big]\Big|
    \\&\le \Big|    \sum_{i_2\le m}\sum_{i_3\le m} \mathbb{E}\Big[\mathbb{E}\Big(X^{1,m}_{1,2}X^{1,m}_{1,3}\Big|\mathbb{S}(\mathbb{N})\Big)\mathbb{E}\Big((f^Kaf^K)_{i_3,i_2}(\overline{f^K})_{i_2,i_2}\Big|\mathbb{S}(\mathbb{N})\Big)\Big]\Big|
      \\&~+\Big|    \sum_{i_2\le m}\sum_{i_4\le m} \mathbb{E}\Big[\mathbb{E}\Big(X^{1,m}_{1,2}X^{1,m}_{1,3}\Big|\mathbb{S}(\mathbb{N})\Big)\mathbb{E}\Big((f^Kaf^K)_{i_4,i_2}(\overline{f^K})_{i_2,i_2}\Big|\mathbb{S}(\mathbb{N})\Big)\Big]\Big|
  \\&\le 4m\sup_{i,j\le m}\|X^{1,m}_{i,j}\|_2^2 .
\end{align*}
Similarly by exploiting the fact that $\|f^K\|_{\infty}, \|f^Kaf^K\|_{\infty}\le 1$ then we have 
\begin{align*}&
      \sum_{i_2\le m}\sum_{i_3=i_2~\rm{or}~ i_4=i_2} \mathbb{E}\Big[\mathbb{E}\Big(X^{1,m}_{i_2,i_3}X^{1,m}_{i_4,i_2}\Big|\mathbb{S}(\mathbb{N})\Big)\mathbb{E}\Big((f^Kaf^K)_{i_3,i_4}(\overline{f^K})_{i_2,i_2}\Big|\mathbb{S}(\mathbb{N})\Big)\Big]
      \\&\le \Big|    \sum_{i_2\le m}\sum_{i_3\le m} \mathbb{E}\Big[\mathbb{E}\Big(X^{1,m}_{i_2,i_3}X^{1,m}_{i_4,i_2}\Big|\mathbb{S}(\mathbb{N})\Big)\mathbb{E}\Big((f^Kaf^K)_{i_3,i_2}(\overline{f^K})_{i_2,i_2}\Big|\mathbb{S}(\mathbb{N})\Big)\Big]\Big|
      \\&~+\Big|    \sum_{i_2\le m}\sum_{i_4\le m} \mathbb{E}\Big[\mathbb{E}\Big(X^{1,m}_{i_2,i_3}X^{1,m}_{i_4,i_2}\Big|\mathbb{S}(\mathbb{N})\Big)\mathbb{E}\Big((f^Kaf^K)_{i_4,i_2}(\overline{f^K})_{i_2,i_2}\Big|\mathbb{S}(\mathbb{N})\Big)\Big]\Big|
      \\&\le 4m \sup_{i,j\le m}\|X^{1,m}_{i,j}\|_2^2.
\end{align*}
Therefore we obtain that 
\begin{align*}&\frac{1}{m^2}\sum_{i_{2:4}\le m } \mathbb{E}\Big[\mathbb{E}\Big(X^{1,m}_{i_2,i_3}X^{1,m}_{i_4,i_2}\Big|\mathbb{S}(\mathbb{N})\Big)\mathbb{E}\Big((f^Kaf^K)_{i_3,i_4}(\overline{f^K})_{i_2,i_2}\Big|\mathbb{S}(\mathbb{N})\Big)\Big|\Big]\\&\lesssim\frac{ 4}{\sqrt{m}}\sup_{i,j\le m}\|X^{1,m}_{i,j}\|_2^3+ \frac{8}{m}\sup_{i,j\le m}\|X^{1,m}_{i,j}\|_2^2.
\end{align*}
Similarly we can prove that \begin{align*}&\frac{1}{m^2}\sum_{i_{2:4}\le m } \mathbb{E}\Big[\mathbb{E}\Big(X^{1,m}_{i_2,i_3}X^{1,m}_{i_4,i_2}\Big|\mathbb{S}(\mathbb{N})\Big)\mathbb{E}\Big((f^Kaf^K)_{i_2,i_2}(\overline{f^K})_{i_3,i_4}\Big|\mathbb{S}(\mathbb{N})\Big)\Big|\Big]\\&\lesssim\frac{ 4}{\sqrt{m}}\|X^{1,m}_{i,j}\|_2^3+ \frac{8}{m}\sup_{i,j\le m}\|X^{1,m}_{i,j}\|_2^2.
\end{align*}
We now want to prove that if we write $A^{z_1,z_2}:=X^{z_1,m}E_m(f^K)X^{z_2,m}$ and shorthand $\overline{A^{z_1,z_2}}:=A^{z_1,z_2}-E_m(A^{z_1,z_2})$ then the following is small $\big\|E_m(f^K\overline{A^{z_1,z_2}}f^K)\big\|_1$. To see this we first remark that $\overline{A^{z_1,z_2}}=A^{z_1,z_2}-\mathbb{E}(A^{z_1,z_2}|\mathbb{S}(\mathbb{N}))$ and note that $A^{z_1,z_2}$ is conditionally (on $\sigma(\mathbb{S}(\mathbb{N}))$) independent from $f^K$. Therefore we directly obtain that $$\big\|E_m(f^K\overline{A^{z_1,z_2}}f^K)\big\|_1=0.$$

Therefore as we have assumed that $\sup_m\frac{\sup_{i,j\le m}\|X^{1,m}_{i,j}\|_2}{\min_{i,j\le m}\|X^{1,m}_{i,j}\|_2}<\infty$ we obtain that there is a constant $C$ such that  $$\aleph_n^{j,\lambda}[1|\Z]\le C\Big[\frac{\sup_{i,j\le m}\|X^{1,m}_{i,j}\|_2}{\sqrt{m}}+\frac{1}{m}\Big] $$
\end{proof}

\section{Proof of \cref{marre_2}}
\begin{proof} We remark that we can suppose without loss of generality that the entries of the random matrices $(X^{\z,m})$ have mean $0$ and do so in the following. We denote $E_{\mathcal{D}_m}$ the non-commutative conditional expectation on ${\mathcal{F}_{\mathcal{D}(\Z^2)}^{\rm{tail}}}(Y^m)$. We remark that for any $F\in \mathcal{F}^{d}_{\G}(Y^m)$, we have \begin{align}\label{hus4013}E_{\mathcal{D}_m}(F)=\big(\mathbb{E}(F_{i,j})\big)_{i,j\le m}.\end{align}



\noindent Let $\lambda>0$. We first upper-bound $\aleph^{*,j,\lambda}[\cdot|\mathbb{Z}^2]$. In this goal, let $\z,\z'\in \mathbb{Z}^2$  and $K\subset\mathbb{Z}^2$ be such  that $$\min_{\z^*\in K }\min_{j\le 2}\min\big(|\z_j-\z^*_j|,|\z'_j-\z^*_j|\big)\ge b.$$  Choose $\gamma\in \mathbb{C}$ with $\rm{Im}(\gamma)>\lambda$ and write $f:x\rightarrow \rm{Im}(\gamma)~[x-\gamma \rm{Id}]^{-1}$. We shorthand $f^K:=f\Big(\frac{1}{\sqrt{|K|}}\sum_{\z\in K}X^{\z,m}\Big)$ as well as $\overline{{f^K}}:=f^K-E_{\mathcal{D}_m}(f^K).$ The proof will work in two stages: Firstly we show that $\Big\|E_{\mathcal{D}_m}\Big(f^KX^{\z,m}\overline{f^K}X^{\z',m}f^K\Big)\Big\|_1$ is small and then we establish that $\Big\|E_{\mathcal{D}_m}\Big(f^K\overline{X^{\z,m}E_{\mathcal{D}_m}({f^K})X^{\z',m}}f^K\Big)\Big\|_1$ is also small.

\noindent In this goal, we remark that by definition of the $\mathbb{L}_1$ norm we have
\begin{align*}&
    \|E_{\mathcal{D}_m}(f^{K}X^{\z,m}\overline{f^{K}}X^{\z',m}f^{K})\|_1
    \\&\le \sup_{\substack{a\in \mathcal{F}_{\mathcal{D}(\Z^2)}^{tail}(Y^m)\\\|a\|_{\infty}\le 1}}\tau\Big(af^{K}X^{\z,m}\overline{f^{K}}X^{\z',m}f^{K}\Big)
\end{align*} Choose $a\in \mathcal{F}_{\mathcal{D}(\Z^2)}^{tail}(Y^m)$ such that $\|a\|_{\infty}\le 1$. We note that by \cref{hus4013} that $a$ is a deterministic matrix. Let $(\tilde X^{\z,m},\tilde X^{\z',m})$ be a copy of $(X^{\z,m},X^{\z',m})$ that is independent from $(X^{\z,m})$. Then by definition of the strong-mixing coefficients and \cref{chanson}, we have 
\begin{equation}\begin{split}\label{hpp_190}&\Big|\tau\Big(af^{K}{X^{\z,m}}{}\overline{f^{K}}{X^{\z',m}}{}f^{K}\Big)-\tau\Big(af^{K}{\tilde X^{\z,m}}{}\overline{f^{K}} {\tilde X^{\z',m}}{} f^{K}\Big)\Big|
\\&\overset{(a)}{\le} \Big|\frac{1}{m}\mathbb{E}\Big(\rm{Tr}(af^{K}{X^{\z,m}}{}\overline{f^{K}}{X^{\z',m}}{}f^{K})\Big)-\frac{1}{m}\mathbb{E}\Big(\rm{Tr}\big(af^{K}{\tilde X^{\z,m}}{}\overline{f^{K}} {\tilde X^{\z',m}}{} f^{K}\big)\Big)\Big|
\\&\overset{(b)}{\lesssim}\frac{1}{m}\alpha_m[b]^{\frac{\epsilon}{2+\epsilon}}\Big[\Big\|\rm{Tr}\big(af^{K}{\tilde X^{\z,m}}{}\overline{f^{K}} {\tilde X^{\z',m}}{} f^{K}\big)\Big\|_{1+\frac{\epsilon}{2}}+\Big\|\rm{Tr}\big(af^{K}{X^{\z,m}}{}\overline{f^{K}} { X^{\z',m}}{} f^{K}\big)\Big\|_{1+\frac{\epsilon}{2}}\Big]
\\&\overset{(c)}{\lesssim} \frac{1}{m}\alpha_m[b]^{\frac{\epsilon}{2+\epsilon}}\|\sqrt{\rm{Tr}((X^{\z,m})^2)}\|_{L_{2+\epsilon}}\|\sqrt{\rm{Tr}((X^{\z',m})^2)}\|_{L_{2+\epsilon}}
\end{split}\end{equation} 
where (a) is a consequence of the equation $\tau_m(\cdot)=\frac{1}{m}\mathbb{E}(\rm{Tr}(\cdot))$; (b) comes from \cref{chanson} and where to get (c) we used the fact that for any matrix $A,B\in M_m(\mathbb{C})$ by H\"{o}lder inequality the following holds
\begin{align*}
\frac{1}{m}
|Tr(f^Kaf^KA\overline{f^K}B)|\le 2\|A\|_2\|B\|_2
\end{align*}
Moreover we observe that by Jensen inequality we have:
\begin{align*}
  \mathbb{E}\Big({Tr((X^{\z,m})^2)}^{1+\frac{\epsilon}{2}}\Big)&\le m^{\epsilon}\mathbb{E}\Big(\sum_{i,j\le m}(X_{i,j}^{\z,m})^{2+\epsilon}\Big)
  \\&\overset{(a)}{\lesssim} \sup_{l\le p}\sup_{z\in \Z} m^{\epsilon}\sum_{i,j\le m}\mathbb{E}\Big((Y_{i,j}^{z,m})^{2l+\epsilon l}\Big)
\end{align*}
where to get $(a)$ we used the fact that $\Phi$ is a polynomial of degree $p$. As $(Z_{s,z})$ are independent normal then using Rosenthal inequality we obtain that for all $l\le p$
\begin{align*}
    \mathbb{E}\Big((Y_{i,j}^{z,m})^{l(2+\epsilon )}\Big)\lesssim \Big(\sum_{s\le N_m}(A^{s,m}_{i,j})^2\Big)^{l(1+\frac{\epsilon}{2})}+\sum_{s\le N_m}(A^{s,m}_{i,j})^{l(2+\epsilon)}.
\end{align*}
Therefore we obtain that 
\begin{equation}\begin{split}\label{hpp_190}&\Big|\tau\Big(af^{K}{X^{\z,m}}{}\overline{f^{K}}{X^{\z',m}}{}f^{K}\Big)-\tau\Big(af^{K}{\tilde X^{\z,m}}{}\overline{f^{K}} {\tilde X^{\z',m}}{} f^{K}\Big)\Big|
\\&\overset{}{\lesssim} \alpha_m[b]^{\frac{\epsilon}{2+\epsilon}}\sup_{1\le l\le p}\Big(\frac{1}{m^{1-\epsilon/2}}\sum_{i,j}\Big[\Big(\sum_{s\le N_m}(A^{s,m}_{i,j})^2\Big)^{l(1+\frac{\epsilon}{2})}+\sum_{s\le N_m}(A^{s,m}_{i,j})^{l(2+\epsilon)}\Big]\Big)^{\frac{2}{2+\epsilon}}
\end{split}\end{equation}

\noindent Moreover  we observe that \begin{equation*}\begin{split}&\Big|\tau\Big(af^{K}{\tilde X^{\z,m}}\overline{f^{K}} {\tilde X^{\z',m}} f^{K}\Big)\Big|
\\&\le \Big|\tau\Big(f^{K}af^{K}{\tilde X^{\z,m}}\overline{f^{K}} {\tilde X^{\z',m}} \Big)\Big|
\\&\le\Big| \frac{1}{m}\sum_{i_{1:4}\le m}\mathbb{E}\Big((f^{K}af^{K})_{i_1,i_2}\tilde X^{\z,m}_{i_2,i_3}\overline f^{K}_{i_3,i_4}\tilde X^{\z',m}_{i_4,i_1} \Big)\Big|
\\&\overset{(a)}{\le} \Big|\frac{1}{m}\sum_{i_{1:4}\le m}\mathbb{E}\Big( X^{\z,m}_{i_2,i_3}X^{\z',m}_{i_4,i_1}\Big)\mathbb{E}\Big((f^{K}af^{K})_{i_1,i_2}\overline f^{K}_{i_3,i_4} \Big)\Big|
\\&\overset{(b)}{\le} \frac{1}{m}\Big|\sum_{i_{1:4}\le m}\mathbb{E}\Big( X^{\z,m}_{i_2,i_3}X^{\z',m}_{i_4,i_1}\Big)\rm{cov}\Big(f^{K}_{i_3,i_4},(f^{K}af^{K})_{i_1,i_2} \Big)\Big|
\end{split}\end{equation*} where to obtain (a) we used the independence of $(\tilde X^{\z,m})$ and  $(X^{\z,m})$, and where (b) is consequence from the fact that by definition of $\overline{f^K}$ we have
\begin{align*}&\mathbb{E}\Big((f^{K}af^{K})_{i_1,i_2}\overline f^{K}_{i_3,i_4} \Big)
\\&=\mathbb{E}\Big((f^{K}af^{K})_{i_1,i_2}f^{K}_{i_3,i_4} \Big)-\mathbb{E}\Big((f^{K}af^{K})_{i_1,i_2}\Big)\mathbb{E}\Big(f^{K}_{i_3,i_4} \Big)
   \\&= \rm{cov}\Big(f^{K}_{i_3,i_4},(f^{K}af^{K})_{i_1,i_2} \Big).
\end{align*}

\noindent To bound  $\rm{cov}\Big(f^{K}_{i_3,i_4},(f^{K}af^{K})_{i_1,i_2} \Big)$ we will exploit the heat equation. 
To do so, we define $$g_1:x\rightarrow [\frac{1}{|K|^{3/2}}\sum_{i,j\in K}\Phi_m\Big(\sum_{s\le N_m} x_{s,i} A^{s,m},\sum_{s\le N_m} x_{s,j} A^{s,m}\Big) -\gamma Id]^{-1}$$ and $$g_2:x\rightarrow g_1(x)ag_1(x).$$

\noindent For ease of notation, we will write $Z^K:=(Z_{s,i})_{i\in K,s\le N_m}$ and let $( Z'_{s,i})$ be an independent copy of $(Z_{s,i})$. We define $Z_{s,i}(t)=\sqrt{t}Z_{s,i}+\sqrt{1-t^2} Z'_{s,i}$ and $Z^K(t):=(Z_{s,i}(t))_{i\in K,s\le N_m}$. Moreover we also write $Y^{i,m}(t):=\sum_{s\le N_m}Z_{s,i}(t)A^{s,m}$. {As $\Phi_m(\cdot,\cdot)$ is a polynomial, there is a polynomial  $\Phi_m'(\cdot,\cdot,\cdot)$  of degree $1$ in its last coordinate such that such that 
\begin{align}
    \frac{\Phi_m(A+H,B)- \Phi_m'(A,B,H)}{\|H\|}\xrightarrow{\|H\|\rightarrow 0}
0,\qquad \qquad \forall A,B,H\in M_m(\mathbb{C}).\end{align}} As $\Phi_m'(A,B,H)$ is of degree $1$ in $H$ there is an integer $M\in \mathbb{N}$ and  polynomials $(P_{1,l}(\cdot,\cdot))$ and $(P_{2,l}(\cdot,\cdot))$ such that: 
\begin{align*}
    \Phi_m'(A,B,H)=\sum_{l\le M}P_{1,l}(A,B)HP_{2,l}(A,B).
\end{align*}
We write $$\hat{P}^{K,i}_{1,l}(t)=\frac{1}{|K|}\sum_{j\in K}P_{1,l}(Y^{i,m}(t),Y^{j,m}(t))$$ and $$\hat{P}^{K,i}_{2,l}(t)=\frac{1}{|K|}\sum_{j\in K}P_{2,l}(Y^{i,m}(t),Y^{j,m}(t)).$$ For ease of notation when $t=0$ we shorthand $\hat{P}^{K,i}_{1,l}:=\hat{P}^{K,i}_{1,l}(0)$ and $\hat{P}^{K,i}_{2,l}:=\hat{P}^{K,i}_{2,l}(0)$. We remark that for all $(x_{s,i})$ and $(h_{s,i})$ the following holds\begin{align*}&
    g_1(x+h)-g_1(x)\\&=g_1(x)\Big[\frac{1}{|K|^{3/2}}\sum_{i,j\in K}\Phi_m\big(\sum_{s\le N_m}(x_{s,i}+h_{s,i})A^{s,m},(x_{s,j}+h_{s,j})A^{s,m}\big)\\&\qquad \quad-\frac{1}{|K|^{3/2}}\sum_{i,j\in K}\Phi_m\big(\sum_{s\le N_m}x_{s,i}A^{s,m},x_{s,j}A^{s,m}\big)\Big]g_1(x+h)
\end{align*} 
Therefore using the symmetry of $\Phi_m(\cdot,\cdot)$ we remark that $g_1$ is differentiable and that 
\begin{align*}
    \nabla_{i,s}g_1(Z^K)=\frac{2}{|K|^{1/2}}g_1(Z^K)\big[\sum_{l\le M}\hat{P}_{1,l}^{K,i}A^{s,m}\hat{P}^{K,i}_{2,l}\big]g_1(Z^K)
\end{align*}In the same way, we remark that the function $g_2$ is also differentiable and that the following identity holds
\begin{align*}
  {  \nabla_{i,s} g_2\big(Z^K(t)\big)}&=\frac{2}{|K|^{1/2}}g_2(Z^K(t))\big[\sum_{l\le M}\hat{P}_{1,l}^{K,i}(t)A^{s,m}\hat{P}^{K,i}_{2,l}(t)\big]g_1(Z^K(t))
     \\&+\frac{2}{|K|^{1/2}}g_1(Z^K(t))\big[\sum_{l\le M}\hat{P}_{1,l}^{K,i}(t)A^{s,m}\hat{P}^{K,i}_{2,l}(t)\big]g_2(Z^K(t)).
\end{align*}
We write $\rho_{\substack{l_{1:2},s'}}=\rm{Cov}\big(Z_{s', l_1},Z_{s', l_2}\big)$. Using the heat equation \cite{ledoux2001concentration}[section 5.5] we know that 
\begin{align}\label{ts_sad0}
 &\rm{cov}\Big(f^{K}_{i_3,i_4},(f^Kaf^{K})_{i_1,i_2}\Big)
  \\&=\rm{Im}(\gamma)^3\int_0^1\sum_{l_1,l_2\in K}\sum_{s'\le N_m}\rho_{\substack{l_{1:2},s'}}\mathbb{E}\Big(\nabla_{s',l_1}g_1(Z^K)_{i_3,i_4}\nabla_{s',l_2}g_2(Z^K(t))_{i_1,i_2}\Big)dt\nonumber
    \\&=\frac{4\rm{Im}(\gamma)^3}{K}\int_0^1\sum_{s'\le N_m}\sum_{l_1,l_2\in K}\rho_{\substack{l_{1:2},s'}}\mathbb{E}\Big[\big(g_1(Z^K)\big[\sum_{l\le M}\hat{P}_{1,l}^{K,l_1}A^{s',m}\hat{P}^{K,l_1}_{2,l}\big]g_1(Z^K)\big)_{i_3,i_4}\nonumber\\&\qquad \qquad \big(g_1(Z^K(t))\big[\sum_{l\le M}\hat{P}_{1,l}^{K,l_2}(t)A^{s',m}\hat{P}^{K,l_2}_{2,l}(t)\big]g_2(Z^K(t))\big)_{i_1,i_2}\Big]dt\nonumber
   \\&+\frac{4\rm{Im}(\gamma)^3}{K}\int_0^1\sum_{s'\le N_m}\sum_{l_1,l_2\in K}\rho_{\substack{l_{1:2},s'}}\mathbb{E}\Big(\big(g_1(Z^K)\big[\sum_{l\le M}\hat{P}_{1,l}^{K,l_1}A^{s',m}\hat{P}^{K,l_1}_{2,l}\big]g_1(Z^K)\big)_{i_3,i_4}\nonumber\\&\qquad \qquad \big(g_2(Z^K(t))\big[\sum_{l\le M}\hat{P}_{1,l}^{K,l_2}(t)A^{s',m}\hat{P}^{K,l_2}_{2,l}(t)\big]g_1(Z^K(t))\big)_{i_1,i_2}\Big)dt.\nonumber
   \end{align}
We can also use the heat equation \cite{ledoux2001concentration}[section 5.5] to reexpress $\rm{cov}\Big(X^{\z,m}_{i_2,i_3}, X^{\z',m}_{i_4,i_1}\Big)$. In this goal for all $i,j\in \mathbb{N}$ we can write $$H_{i,j}:x\rightarrow \Phi_m\big(\sum_{s\le N_m}x_{s,i}A^{s,m},\sum_{s\le N_m}x_{s,j}A^{s,m}\big).$$We remark that $H_{z_1,z_2}(Z)=X^{(z_1,z_2),m}$ and $H_{z'_1,z'_2}(Z)=X^{(z'_1,z'_2),m}$. We also note that $H_{i,j}$ is differentiable and for all $l\in \mathbb{N}$ we have\begin{align*}
     \nabla_{l,s}  H_{i,j}(Z)=\begin{cases}\sum_{l\le M}{P}_{1,l}(Y^{i,m},Y^{j,m})A^{s,m}{P}_{2,l}(Y^{i,m},Y^{j,m})\qquad~ \rm{if} ~l=i\\\sum_{l\le M}{P}_{1,l}(Y^{j,m},Y^{i,m})A^{s,m}{P}_{2,l}(Y^{j,m},Y^{i,m})\qquad~ \rm{if} ~l=j\\0 ~\rm{otherwise}.\end{cases}
   \end{align*} We write 
   $\sum_{l\le M}P_{1,l}^{k_{1:2},k_3}(t)A^{s,m}P_{2,l}^{k_{1:2},k_3}(t):=\nabla_{k_3,s}H_{k_1,k_2}(Z(t))$. Using the heat equation we also obtain that
   \begin{align}\label{ts_sad}&
\rm{Cov}\Big[X^{(z_1,z_2),m}_{i_2,i_3},X^{(z'_1,z'_2),m}_{i_4,i_1}\Big]\\&=\int_0^1\sum_{s\le N_m}\sum_{l_1\in \{\z_1,\z_2\},l_2\in \{\z'_1,\z'_2\}}\rho_{\substack{l_1,l_2,s}}\mathbb{E}\Big(\nabla_{s,l_1}(H_{z_1,z'_1})_{i_2,i_3}(Z)\nabla_{s,l_2}H_{z_1,z'_1}(Z(t))_{i_4,i_1}\Big)dt \nonumber
\\&=\int_0^1\sum_{s\le N_m}\sum_{\substack{l_1\in \{\z_1,\z_2\}\\l_2\in \{\z'_1,\z'_2\}}}\rho_{\substack{l_1,l_2,s}}\mathbb{E}\Big(\big[\sum_{l\le M}P_{1,l}^{z_{1:2},l_1}A^{s,m}P_{2,l}^{z_{1:2},l_1}\big]\big[\sum_{l\le M}P_{1,l}^{\z_{1:2},l_2}(t)A^{s,m}P_{2,l}^{\z_{1:2},l_2}(t)\big]_{i_4,i_1}\Big)dt \nonumber
   \end{align} We will designate by $\tilde P_{1,l}^{k_{1:2},k_3}$ and $\tilde P_{1,l}^{k_{1:2},k_3}(t)$ independent (from $X^{\z,m}$) copies of respectively $P_{1,l}^{k_{1:2},k_3}$ and $P_{1,l}^{k_{1:2},k_3}(t)$. 
~ For a matrix $A\in M_m(\mathbb{C})$ we write $\sigma(A):=\|A\|_{\infty}$ the spectral norm of $A$.
~ Using \cref{ts_sad} and \cref{ts_sad0} we obtain that

\begin{equation*}\begin{split}&\Big|\tau\Big(af^{K}{\tilde X^{z_1,z_2,m}}\overline{f^{K}} {\tilde X^{z'_1,z'_2,m}} f^{K}\Big)\Big|
\\&\lesssim \frac{\rm{Im}(\gamma)^3}{m}\int_{[0,1]^2}\sum_{\substack{s\le N_m\\s'\le N_m}}\sum_{\substack{l_1,l_2\in K\\l_3\in \{\z_1,\z_2\}\\l_4\in \{\z'_1,\z'_2\}\\l_{5:8}\le M}}\frac{|\rho_{\substack{l_{1:2},s'}}||\rho_{\substack{l_{3:4},s}}|}{|K|}\Big|\mathbb{E}\Big(\rm{Tr}\Big(\mathbb{E}\Big[\big[g_2(Z^K(t))\hat{P}_{1,l_5}^{K,l_1}(t)A^{s',m}\hat{P}^{K,l_1}_{2,l_5}(t)g_1(Z^K(t))\big]\\&\qquad\big[\tilde P_{1,l_6}^{\z_{1:2},l_3}A^{s,m}\tilde P_{2,l_6}^{\z_{1:2},l_3}\big]\big[g_1(Z^K)\hat{P}_{1,l_7}^{K,l_3}A^{s',m}\hat{P}^{K,l_2}_{2,l_7}g_1(Z^K)\big] \big[\tilde P_{1,l_8}^{\z'_{1:2},l_4}(t')A^{s,m}\tilde P_{2,l_8}^{\z'_{1:2},l_4}(t')\big]\Big]\Big)dt' dt\Big|
\\&+\frac{\rm{Im}(\gamma)^3}{m}\int_{[0,1]^2}\sum_{\substack{s\le N_m\\s'\le N_m}}\sum_{\substack{l_1,l_2\in K\\l_3\in \{\z_1,\z_2\}\\l_4\in \{\z'_1,\z'_2\}\\l_{5:8}\le M}}\frac{|\rho_{\substack{l_{1:2},s'}}||\rho_{\substack{l_{3:4},s}}|}{|K|}\Big|\mathbb{E}\Big(\rm{Tr}\Big(\mathbb{E}\Big[\big[g_1(Z^K(t))\hat{P}_{1,l_5}^{K,l_1}(t)A^{s',m}\hat{P}^{K,l_1}_{2,l_5}(t)g_2(Z^K(t))\big]\\&\qquad  \big[\tilde P_{1,l_6}^{\z_{1:2},l_3}A^{s,m}\tilde P_{2,l_6}^{\z_{1:2},l_3}\big]\big[g_1(Z^K)\hat{P}_{1,l_7}^{K,l_3}A^{s',m}\hat{P}^{K,l_2}_{2,l_7}g_1(Z^K)\big] \big[\tilde P_{1,l_8}^{\z'_{1:2},l_4}(t')A^{s,m}\tilde P_{2,l_8}^{\z'_{1:2},l_4}(t')\big]\Big]\Big)dt' dt\Big|
\end{split}\end{equation*} We remark that $\frac{g_1(Z^K(t))\hat P_{1,l_5}^{K,l_1}(t)}{\sigma\big(g_1(Z^K(t))\hat P_{l,l_5}^{K,l_1}(t)\big)}$ is almost surely a matrix with a spectral radius of less than $1$. Similar statements can be made for 
the other random matrices. We, therefore, obtain that 
\begin{equation*}\begin{split}&\Big|\tau\Big(af^{K}{\tilde X^{z_1,z_2,m}}\overline{f^{K}} {\tilde X^{z'_1,z'_2,m}} f^{K}\Big)\Big|
\\&\lesssim 2M^4\rm{Im}(\gamma)^{-1}\sup_{\substack{l\le M\\ k\in K}}\|\sigma\big(\hat P^{K,k}_{1,l})\sigma\big(\hat P^{K,k}_{2,l})\big\|_4^2\sup_{\substack{l\le M\\ \z^*\in \{\z,\z'\}\\k\in \{\z^*_1,\z^*_2\}}}\|\sigma\big(P^{\z^*_{1:2},k}_{1,l})\sigma\big(P^{\z^*_{1:2},k}_{2,l})\big\|_4^2\\&\qquad \times \sup_{\substack{Y_{1:4}\\\max_{i\le 4}\|Y_i\|_{\infty}\le 1}}\Big\|\sum_{s,s'\le N_m}\sum_{\substack{l_1,l_2\in K\\l_3\in \{\z_1,\z_2\},l_4\in \{\z'_1,\z'_2\}}}\frac{|\rho_{\substack{l_{1:2},s'}}||\rho_{\substack{l_{3:4},s}}|}{|K|}A^{s,m}Y_1A^{s',m}Y_2A^{s,m}Y_3A^{s',m}Y_4\Big\|_1
\\&\overset{(a)}{\lesssim} 2M^4\rm{Im}(\gamma)^{-1}\sup_{\substack{l\le M\\ k\in K}}\|\sigma\big(\hat P^{K,k}_{1,l})\sigma\big(\hat P^{K,k}_{2,l})\big\|_4^2\sup_{\substack{l\le M\\ \z^*\in \{\z,\z'\}\\k\in \{\z^*_1,\z^*_2\}}}\|\sigma\big(P^{\z^*_{1:2},k}_{1,l})\sigma\big(P^{\z^*_{1:2},k}_{2,l})\big\|_4^2\\&\qquad \times \sup_{\substack{U_{1:4}\in U(m)}}\Big\|\sum_{s,s'\le N_m}\sum_{\substack{l_1,l_2\in K\\l_3\in \{\z_1,\z_2\},l_4\in \{\z'_1,\z'_2\}}}\frac{|\rho_{\substack{l_{1:2},s'}}||\rho_{\substack{l_{3:4},s}}|}{|K|}A^{s,m}U_1A^{s',m}U_2A^{s,m}U_3A^{s',m}U_4\Big\|_1
\end{split}\end{equation*}
where to obtain (a) we used the fact that every matrix $\|Y\|_{\infty}\le 1$ can be written as a convex combination of unitary matrices.
We bound each term successively. Firstly we remark for any matrix $A\in M_n(\mathbb{R})$ we have that $\sigma(A)\le \max_{i,j}|A_{i,j}|.$ Moreover as the polynomials $(P_{1,l})_{l\le M}$ and $(P_{2,l})_{l\le M}$ are all of degree $p-1$ for all $l\le M$ and all $i\in K$ we have \begin{align*}\max(\sigma(\hat{P}_{1,l}^{K,i}),\sigma(\hat{P}_{2,l}^{K,i}))\lesssim& \Big[ \sup_{s\le N_m}|Z_{s,i}|^{p-1}\wedge 1+ \frac{1}{|K|}\sum_{j\in K}\sup_{s\le N_m}|Z_{s,j}|^{p-1}\wedge 1\Big]\\&\qquad \max_{i,j\le m}\Big(\sum_{s\le N_m}|A^{s,m}_{i,j}|\Big)^{p-1}\wedge 1.\end{align*}For ease of notation we write $s_{m,p-1}:=\max_{i,j\le m}\Big(\sum_{s\le N_m}|A^{s,m}_{i,j}|\Big)^{p-1}\wedge 1$. Moreover similarly for all $\z^*\in \{\z,\z'\}$ and all $k\in\{\z_1^*,\z^*_2\}$ we also have $$\max(\sigma(\hat{P}_{1,l}^{\z^*_{1:2},k}),\sigma(\hat{P}_{2,l}^{\z^*_{1:2},k}))\lesssim ( \sup_{\substack{s\le N_m\\ k\in \{\z_{1:2},\z'_{1:2}\}}}|Z_{s,k}|^{p-1}\wedge 1)s_{m,p-1}.$$
Moreover we know that  $(Z_{s,i})$ are i.i.d standard normal random variables therefore for every $i\in \mathbb{N}$ the following holds $$\mathbb{E}(\max_{s\le N_m}|Z_{s,i}|^{p-1})\lesssim\sqrt{1+\log(N_m)}^{p-1}.$$
 Therefore using Cauchy-Swartz inequality we obtain that
 \begin{align*}&
    2M^4\sup_{\substack{l\le M\\ k\in K}}\|\sigma\big(\hat P^{K,k}_{1,l})\sigma\big(\hat P^{K,k}_{2,l})\big\|_4^2\sup_{\substack{l\le M\\ \z^*\in \{\z,\z'\}\\k\in \{\z^*_1,\z^*_2\}}}\|\sigma\big(P^{\z^*_{1:2},k}_{1,l})\sigma\big(P^{\z^*_{1:2},k}_{2,l})\big\|_4^2
    \\&\lesssim s_{m,p-1}^4(1+\log(N_m))^{2(p-1)}.
 \end{align*}
We now move on to upper-bounding $$\sup_{\substack{U_{1:4}\in U(m)}}\Big\|\sum_{s,s'\le N_m}\sum_{\substack{l_1,l_2\in K\\l_3\in \{\z_1,\z_2\},l_4\in \{\z'_1,\z'_2\}}}\frac{|\rho_{\substack{l_{1:2},s'}}||\rho_{\substack{l_{3:4},s}}|}{|K|}A^{s,m}U_1A^{s',m}U_2A^{s,m}U_3A^{s',m}U_4\Big\|_1.$$ In this goal, we write $T_{K,s'}^2:=\sum_{\substack{l_1,l_2\in K}}\frac{|\rho_{\substack{l_{1:2},s'}}|}{|K|}$ and denote $\tilde \rho_s:=\sum_{\substack{l_3\in \{\z_1,\z_2\}\\l_4\in \{\z'_1,\z'_2\}}}{|\rho_{\substack{l_{3:4},s}}|}$. Let $(N_i)$ be a sequence of standard normal random variables and denote $D_K:= \sum_{s\le N_m}T_{K,s}N_sA^{s,m}$.  We then remark that for any choice of $U_1,U_2,U_3,U_4$ we have
\begin{align*}&
\Big\|\sum_{s,s'\le N_m}\sum_{\substack{l_1,l_2\in K\\l_3\in \{\z_1,\z_2\},l_4\in \{\z'_1,\z'_2\}}}\frac{|\rho_{\substack{l_{1:2},s'}}||\rho_{\substack{l_{3:4},s}}|}{|K|}A^{s,m}U_1A^{s',m}U_2A^{s,m}U_3A^{s',m}U_4\Big\|_1\\&=
    \Big\|\sum_{s,s'\le N_m}T_{K,s'}^2\tilde \rho_sA^{s,m}U_1A^{s',m}U_2A^{s,m}U_3A^{s',m}U_4\Big\|_1
    \\&=\Big\|\sum_{s\le N_m}\tilde \rho_s\mathbb{E}\Big(A^{s,m}U_1D_KU_2A^{s,m}U_3D_KU_4\Big)\Big\|_1.
\end{align*}
Therefore the value of $\Big\|\sum_{s,s'\le N_m}T_{K,s'}^2\rho_sA^{s,m}U_1A^{s',m}U_2A^{s,m}U_3A^{s',m}U_4\Big\|_1$ only depends on the distribution of $D_K$ and $(A^{s,m})$. We remark that $D_K$ can be seen as a $m^2$ dimensional Gaussian vector with the variance-covariance matrix $\Sigma$. It follows that if we write $(C_i)$ the (unnormalized) orthogonal eigenvectors of $\Sigma$ and if $(Z_i)$ is a sequence of i.i.d standard normal random variables then we observe that $D_K\overset{d}{=}\sum_{i\le m^2}Z_iC_i$.
This directly implies that the following holds \begin{align*}&
\Big\|\sum_{s,s'\le N_m}\sum_{\substack{l_1,l_2\in K\\l_3\in \{\z_1,\z_2\},l_4\in \{\z'_1,\z'_2\}}}\frac{|\rho_{\substack{l_{1:2},s'}}||\rho_{\substack{l_{3:4},s}}|}{|K|}A^{s,m}U_1A^{s',m}U_2A^{s,m}U_3A^{s',m}U_4\Big\|_1
    \\&= \Big\|\sum_{s\le N_m,s'\le m^2}\tilde \rho_sA^{s,m}U_1C_{s'}U_2A^{s,m}U_3C_{s'}U_4\Big\|_1.
\end{align*}
Now note that by definition we have $Tr(C_iC_j)=0$ if $i\ne j$ are differently valued. Therefore we can note that for any matrix $Y$ we have \begin{align}\label{lo}\sum_{i\le m^2}\frac{1}{\|C_i\|_{HS}^2}C_i^*YC_i=Tr(Y)\rm{Id}\end{align}(see lemma 4.8 in \cite{bandeira2021matrix}). Moreover we remark that \begin{align*}&
    \Big\|\sum_{s\le N_m,s'\le m^2}\tilde \rho^sA^{s,m}U_1C_{s'}U_2A^{s,m}U_3C_{s'}U_4\Big\|_1
    \\&\le \sup_{\|x\|,\|y\|\le 1}\Big|\sum_{s\le N_m,s'\le m^2}\tilde \rho_s\Big<U_3C_{s'}U_4x,~A^{s,m}U_2^*C_{s'}^*A^{s,m}U_1^*y\Big>\Big|
    \\&\le\sup_{\|x\|,\|y\|\le 1} \Big(\sum_{ s'\le m^2}\|U_3C_{s'}U_4x\|^2\Big)^{1/2}\Big(\sum_{ s'\le m^2}\Big\|\sum_{s\le N_m}\tilde \rho^s~A^{s,m}U_2^*C_{s'}^*A^{s,m}U_1^*y\Big\|^2\Big)^{1/2}.
\end{align*} 
Note that the following holds\begin{align*}&
    \sum_{s'\le m^2}\Big\|\sum_{s\le N_m}\tilde \rho^s~A^{s,m}U_2^*C_{s'}^*A^{s,m}U_1^*y\Big\|^2
    \\&=    \sum_{s'\le m^2}\sum_{s_1,s_2\le N_m}\tilde\rho^{s_1}\tilde\rho^{s_2}y^*U_1A^{s_1,m}C_{s'}U_2A^{s_1,m}A^{s_2,m}U_2^*C_{s'}^*A^{s_2,m}U_1^*y
    \\&\overset{(a)}{\le } \max_i\|C_i\|_{HS}^2 \Big|\sum_{s_1,s_2\le N_m}\tilde\rho^{s_1}\tilde\rho^{s_2}y^*U_1A^{s_1,m}A^{s_2,m}U_1^*y Tr\Big(A^{s_1,m}A^{s_2,m}\Big)\Big|
       \\&\overset{(b)}{=} \max_i\|C_i\|_{HS}^2 \Big|\sum_{s \le N_m}(\tilde\rho^{s})^2y^*U_1A^{s,m}A^{s,m}U_1^*y Tr\Big((A^{s,m})^2\Big)\Big|
\end{align*} where to obtain (a) we used \cref{lo} and where to obtain $(b)$ we used the fact that the matrices $(A^{s,m})$ are orthogonal. Therefore we obtain that \begin{align*}&
\sup_{\|y\|\le 1} \Big|\sum_{s\le N_m}(\tilde \rho^{s})^2y^*U_1A^{s,m}A^{s,m}U_1^*y Tr\Big((A^{s,m})^2\Big)\Big|
     \\&=\sup_{\|y\|\le 1} \sum_{s\le N_m}(\tilde\rho^{s})^2y^*(A^{s,m})^2y Tr\Big((A^{s,m})^2\Big)
     \\&\le \sup_{s\le N_m}Tr((A^{s,m})^2)\sup_{\|y\|\le 1}\sum_{s\le N_m}(\tilde \rho^{s})^2\|A^{s,m}y\|_2^2 
     \\&\le\sup_{s\le N_m}(\tilde \rho_s)^2\sup_{s\le N_m}Tr((A^{s,m})^2)\sup_{\|y\|\le 1}\sum_{s\le N_m}\|A^{s,m}y\|_2^2 
     \\&\lesssim\mathcal{V}_m(Y^m)^2\sigma_m(Y^m)^2.
\end{align*} where to get the last inequality we used the fact that $\sup_{s\le N_m}\tilde \rho_s\le 4$ and that  by definition we have $$\mathcal{V}_m(Y^m)^2=\sup_{Tr(|M|^2)\le 1}\sum_{i\le N_m}Tr(A^{i,m} M)^2,$$ and the fact that if we define $M_s=\frac{A^{s,m}}{\sqrt{Tr((A^{s,m})^2)}}$ then we have $Tr(M_s^2)=1$ and $\sum_{j\le N_m}Tr(A^{j,m}M_s)^2=Tr((A^{s,m})^2)$.  Moreover, we remark that we also have:
\begin{align*}\max_{i\le m^2}\|C_i\|_{\rm{HS}}^2&\le \sup_{Tr(|M|^2)\le 1}\sum_{s\le N_m}T_{K,s}^2\big|Tr(A^{s,m}M)\big|^2\\&\le\max_{s\le N_m} T_{K,s}^2\sup_{Tr(|M|^2)\le 1}\sum_{s\le N_m}\big|Tr(A^{s,m}M)\big|^2\\&= \max_{s\le N_m}T_{K,s}^2\mathcal{V}_m(Y^m)^2.\end{align*} Moreover according to \cref{chanson} for all $s\le N_m$ we have
  \begin{align*} T_{K,s}^2\lesssim  \|Z\|_{2+\epsilon}^2\sum_{b\le \infty}\alpha_m[b]^{\frac{\epsilon}{2+\epsilon}}
\end{align*} This directly implies that\begin{align*}\max_{i\le m^2}\|C_i\|_{\rm{HS}}^2&\le \mathcal{V}_m(Y^m)^2\|Z\|_{2+\epsilon}^2\sum_{b\le \infty}\alpha_m[b]^{\frac{\epsilon}{2+\epsilon}}.\end{align*} Finally we also remark that  \begin{align*}&
   \sup_{\|x\|\le 1}  \sum_{s'\le m^2}\|U_3C_{s'}U_4x\|^2\\&= \sup_{\|x\|\le 1}  \sum_{s'\le m^2}\|C_{s'}x\|^2
  \\&=\|\sum_{s\le N_m}T_{K,s}^2 (A^{k,s})^2\|_{\infty}\le \max_{s\le N_m}T_{K,s}^2\|\sum_{s\le N_m} (A^{K,s})^2\|_{\infty}.\end{align*} Moreover according to \cref{chanson} for all $s\le N_m$ we have
  \begin{align*} T_{K,s}^2\lesssim  \|Z\|_{2+\epsilon}^2\sum_{b\le \infty}\alpha_m[b]^{\frac{\epsilon}{2+\epsilon}}
\end{align*}Therefore we obtain that\begin{align*}
  \sup_{\|x\|\le 1}  \sum_{s'\le m^2}\|C_{s'}x\|^2\le \|Z\|_{2+\epsilon}^2\sum_{b\le \infty}\alpha_m[b]^{\frac{\epsilon}{2+\epsilon}} \sigma_m(Y^m)^2.\end{align*} 
Therefore we obtain that there is a constant $K$ that does not depend on $m$ such that  \begin{align*}&
   \Big|\tau\Big(af^{K}{\tilde X^{\z,m}}\overline{f^{K}} {\tilde X^{\z',m}} f^{K}\Big)\Big|\\&\le K\sum_{l\ge 0}\alpha_m[l]^{\frac{\epsilon}{2+\epsilon}}\sigma_m(Y^m)^2\mathcal{V}_m(Y^m)^2s_{m,p-1}^4(1+\log(N_m))^{2(p-1)}.
\end{align*}

\noindent Combined with \cref{hpp_190} we obtain that 
\begin{align*}&
  \Big|\tau\Big(af^{K}{X^{\z,m}}\overline{f^{K}} { X^{\z',m}} f^{K}\Big)\Big|
  \\&\lesssim\sum_{l\ge 0}\alpha_m[l]^{\frac{\epsilon}{2+\epsilon}}\sigma_m(Y^m)^2\mathcal{V}_m(Y^m)^2s_{m,p-1}^4(1+\log(N_m))^{2(p-1)}
  \\&+ \alpha_m[b]^{\frac{\epsilon}{2+\epsilon}}\sup_{1\le l\le p}\Big(\frac{1}{m^{1-\epsilon/2}}\sum_{i,j}\Big[\Big(\sum_{s\le N_m}(A^{s,m}_{i,j})^2\Big)^{l(1+\frac{\epsilon}{2})}+\sum_{s\le N_m}(A^{s,m}_{i,j})^{l(2+\epsilon)}\Big]\Big)^{\frac{2}{2+\epsilon}}
\end{align*}

\noindent Now we bound $\Big\|E_{\mathcal{D}_m}\Big(f^K\overline{X^{\z,m}E_{\mathcal{D}_m}({f^K})X^{\z',m}}f^K\Big)\Big\|_1$. In this goal we write $$A^{\z,\z'}:=X^{\z,m}E_{\mathcal{D}_m}({f^K})X^{\z',m}.$$ Firstly we remark that by definition of the $\mathbb{L}_1$ norm we have 
\begin{align*}&
\|E_{\mathcal{D}_m}\Big(f^K\overline{X^{\z,m}E_{\mathcal{D}_m}({f^K})X^{\z',m}}f^K\Big)\|_1
    \\&\le \sup_{\substack{a\in \mathcal{F}^{tail}_{\mathcal{D}(\Z)^2}(X^m)\\\|a\|_{\infty}\le 1}}\tau\Big(af^{K}\overline{A^{\z,\z'}}f^K\Big).
\end{align*} Choose $a\in \mathcal{F}^{tail}_{\mathcal{D}(\Z^2)}(X^m)$ such that $\|a\|_{\infty}\le 1$. We note that by \cref{hus4013} that $a$ is a deterministic matrix. Let $(\tilde X^{\z,m},\tilde X^{\z',m})$ be a copy of $(X^{\z,m},X^{\z',m})$ that is independent from $(X^{\z,m})$ and write $\tilde A^{\z,\z'}:=X^{\z,m}E_{\mathcal{D}_m}({f^K})X^{\z',m}$. Then by definition of the strong-mixing coefficients we have that there is  constant $C>0$ that does not depend on $m$ such that 
\begin{equation}\begin{split}&\Big|\tau\Big(af^{K}A^{\z,\z'}f^K-af^{K}\tilde A^{\z,\z'}f^K\Big)\Big|
\\&\overset{(a)}{\le} \frac{1}{m}\Big|\sum_{i_{1:3}\le m} \rm{cov}\Big(A^{\z,\z'}_{i_2,i_3} ,~(af^K)_{i_{1:2}}f^K_{i_3,i_1}\Big)\Big|
\\&\overset{(b)}{\le} \frac{8}{m}\alpha_m[b]^{\frac{\epsilon}{2+\epsilon}}\|\sqrt{\rm{Tr}((X^{\z,m})^2)}\|_{L_{2+\epsilon}}\|\sqrt{\rm{Tr}((X^{\z',m})^2)}\|_{L_{2+\epsilon}}
\\&\lesssim  \alpha_m[b]^{\frac{\epsilon}{2+\epsilon}}\sup_{1\le l\le p}\Big(\frac{1}{m^{1-\epsilon/2}}\sum_{i,j}\Big[\Big(\sum_{s\le N_m}(A^{s,m}_{i,j})^2\Big)^{l(1+\frac{\epsilon}{2})}+\sum_{s\le N_m}(A^{s,m}_{i,j})^{l(2+\epsilon)}\Big]\Big)^{\frac{2}{2+\epsilon}}
\end{split}\end{equation} 
where (a) is a consequence of the equation $\tau_m(\cdot)=\frac{1}{m}\mathbb{E}(\rm{Tr}(\cdot))$; and where to get (b) we used \cref{chanson} and the fact that $\|af^{K}\|_{\infty},\|f^{K}\|_{\infty}\le 1$. 
Moreover, we remark that
\begin{align*}\tau\Big(af^{K}\overline{\tilde A^{\z,\z'}}f^K\Big)&\overset{(a)}{=}\frac{1}{m}\sum_{i_{1:3}\le m} \mathbb{E}\Big((af^K)_{i_{1:2}}f^K_{i_3,i_1}\Big)\Big[\mathbb{E}(A^{\z,\z'}_{i_2,i_3})-\mathbb{E}(A^{\z,\z'}_{i_2,i_3})\Big]
\\&=0
\end{align*} where (a) is a consequence of the independence of $\tilde A^{\z,\z'}$ and $f^K$.
\noindent This directly implies that\begin{align*}
   \aleph_m^{*,j,\lambda}[b|\mathbb{Z}^2]&\lesssim \sum_{l\ge 0}\alpha_m[l]^{\frac{\epsilon}{2+\epsilon}}\sigma_m(Y^m)^2\mathcal{V}_m(Y^m)^2s_{m,p-1}^4(1+\log(N_m))^{2(p-1)}
   \\&\quad+ \alpha_m[b]^{\frac{\epsilon}{2+\epsilon}}\sup_{1\le l\le p}\Big(\frac{1}{m^{1-\epsilon/2}}\sum_{i,j}\Big[\Big(\sum_{s\le N_m}(A^{s,m}_{i,j})^2\Big)^{l(1+\frac{\epsilon}{2})}+\sum_{s\le N_m}(A^{s,m}_{i,j})^{l(2+\epsilon)}\Big]\Big)^{\frac{2}{2+\epsilon}}
\end{align*}
\noindent We now move on to bounding $\aleph^{*s}_m[\cdot|\mathbb{Z}^2]$. Let $\z,\z'\in \mathbb{Z}$  and $K\subset\mathbb{Z}$ be such  that $$\min_{\z^*\in K\cup\{\z'\} }\min_{j_1,j_2}d(\z^*_{j_1},\z_{j_2})\ge b.$$  Choose $Y_1,Y_2,Y_3\in \mathcal{F}^{\rm{tail}}_{\mathcal{D}_m(\Z^2)}(Y^m)$ be operators satisfying $\|Y_1\|_{\infty},\|Y_2\|_{\infty},\|Y_3\|_{\infty}\le 1$. Firstly we remark that by definition of the $\mathbb{L}_1$ norm we have
\begin{align*}&
    \|E_{\mathcal{D}_m}(Y_1X^{\z,m}Y_2X^{\z',m}Y_3)\|_1
    \\&\le \sup_{\substack{a\in \mathcal{F}^{tail}_{\mathcal{D}(\mathbb{Z})}(Y^m)\\\|a\|_{\infty}\le 1}}\tau\Big(aY_1X^{\z,m}Y_2X^{\z',m}Y_3\Big)
\end{align*} Choose $a\in \mathcal{F}^{tail}_{\mathcal{D}(\Z^2)}(Y^m)$ such that $\|a\|_{\infty}\le 1$. We note that by \cref{hus4013} that $a$ is a deterministic matrix.

\noindent Let $\tilde X^{\z,m}$ be a copy of $X^{\z,m}$ that is independent from $(X^{\z,m})$. Then by definition of the strong-mixing coefficients we have:
\begin{equation*}\begin{split}&\Big|\tau\Big(aY_1{X^{\z,m}}{}Y_2{X^{\z',m}}{}Y_3\Big)-\tau\Big(aY_1{\tilde X^{\z,m}}{}Y_2{ X^{\z',m}}{}Y_3\Big)\Big|
 \\&\le \frac{1}{m}\Big|\mathbb{E}\Big(\rm{Tr}(aY_1{X^{\z,m}}{}Y_2{X^{\z',m}}{}Y_3)\Big)-\mathbb{E}\Big(\rm{Tr}(aY_1\tilde {X^{\z,m}}{}Y_2{X^{\z',m}}{}Y_3)\Big)\Big|
  \\&\lesssim \frac{\alpha_m[b]^{\frac{\epsilon}{2+\epsilon}}}{m}\Big[\Big\|\rm{Tr}(aY_1{X^{\z,m}}{}Y_2{X^{\z',m}}{}Y_3)\Big\|_{1+\frac{\epsilon}{2}}+\rm{Tr}(aY_1{\tilde X^{\z,m}}{}Y_2{X^{\z',m}}{}Y_3)\Big\|_{1+\frac{\epsilon}{2}}\Big]
  \\&\overset{(a)}{\lesssim} \frac{1}{m}\alpha_m[b]^{\frac{\epsilon}{2+\epsilon}}\|\sqrt{\rm{Tr}((X^{\z,m})^2)}\|_{L_{2+\epsilon}}\|\sqrt{\rm{Tr}((X^{\z',m})^2)}\|_{L_{2+\epsilon}}
\\&\lesssim  \alpha_m[b]^{\frac{\epsilon}{2+\epsilon}}\sup_{1\le l\le p}\Big(\frac{1}{m^{1-\epsilon/2}}\sum_{i,j}\Big[\Big(\sum_{s\le N_m}(A^{s,m}_{i,j})^2\Big)^{l(1+\frac{\epsilon}{2})}+\sum_{s\le N_m}(A^{s,m}_{i,j})^{l(2+\epsilon)}\Big]\Big)^{\frac{2}{2+\epsilon}}
\end{split}\end{equation*}
where to get (a) we used \cref{chanson}. Moreover we remark that $\tau\Big(aY_1{\tilde X^{\z,m}}{}Y_2 { X^{\z',m}}{} Y_3\Big)=0$. This directly implies that \begin{equation*}\begin{split}&\Big|E_{\mathcal{D}_m}\Big(Y_1{X^{\z,m}}{}Y_2{X^{\z',m}}{}Y_3\Big)\Big|
   \\&\lesssim\alpha_m[b]^{\frac{\epsilon}{2+\epsilon}}\sup_{1\le l\le p}\Big(\frac{1}{m^{1-\epsilon/2}}\sum_{i,j}\Big[\Big(\sum_{s\le N_m}(A^{s,m}_{i,j})^2\Big)^{l(1+\frac{\epsilon}{2})}+\sum_{s\le N_m}(A^{s,m}_{i,j})^{l(2+\epsilon)}\Big]\Big)^{\frac{2}{2+\epsilon}}
\end{split}\end{equation*}
This directly implies that
\begin{align}
    \aleph^{*s}[b|\Z^2]\lesssim    \alpha_m[b]^{\frac{\epsilon}{2+\epsilon}}\sup_{1\le l\le p}\Big(\frac{1}{m^{1-\epsilon/2}}\sum_{i,j}\Big[\Big(\sum_{s\le N_m}(A^{s,m}_{i,j})^2\Big)^{l(1+\frac{\epsilon}{2})}+\sum_{s\le N_m}(A^{s,m}_{i,j})^{l(2+\epsilon)}\Big]\Big)^{\frac{2}{2+\epsilon}}
\end{align}
\end{proof}
\end{document}